\documentclass[twoside,a4paper,reqno,11pt]{amsart} 
\usepackage[utf8]{inputenc}
\usepackage{amsmath}
\usepackage{amsthm}
\usepackage{hyperref}
\usepackage{float}
\usepackage{color}
\hypersetup{linktocpage}
\usepackage{bm}
\usepackage[shortlabels]{enumitem}
\usepackage{mathtools}
\usepackage{amsfonts}
\usepackage[normalem]{ulem}
\useunder{\uline}{\ul}{}
\usepackage{rotating}
\usepackage{booktabs}
\usepackage{pdflscape}
\usepackage{amssymb}
\usepackage{lipsum}
\usepackage{xcolor}
\usepackage[most]{tcolorbox}
\usepackage{subcaption}
\usepackage{array}
\usepackage{graphicx}
\usepackage[left=2cm,right=2cm,top=2cm,bottom=2cm]{geometry}
\usepackage{caption}

\newtheorem{theorem}{Theorem}[section]
\newtheorem{lemma}[theorem]{Lemma}
\newtheorem{proposition}[theorem]{Proposition}
\newtheorem{corollary}[theorem]{Corollary}

{\theoremstyle{definition}
\newtheorem{remark}[theorem]{Remark}
}
\numberwithin{table}{section}

\newcommand{\ep}{\epsilon}
\newcommand{\pgl}{\mathrm{PGL}}
\newcommand{\gl}{\mathrm{GL}}
\newcommand{\emax}{\mathcal{E}_{{\rm max}}}
\newcommand{\pom}{\mathrm{P} \Omega}
\renewcommand{\a}{\alpha}
\newcommand{\F}{\mathbb{F}}
\usepackage{mathtools}
\DeclarePairedDelimiter\ceil{\lceil}{\rceil}
\DeclarePairedDelimiter\floor{\lfloor}{\rfloor}

\setlist[enumerate,1]{label={(\roman*)}}

\title{Regular orbits of quasisimple linear groups I}
\begin{document}
\author{Melissa Lee}
 \address{Department of Mathematics,
    Imperial College, London SW7 2BZ, UK}
 \email{m.lee16@imperial.ac.uk}

%\tableofcontents
%\newpage

%%%%%%%%%%%%%%%%%%%5
% ALTERNATING GROUP TABLE

%
%\begin{table}[h!]
%	\begin{tabular}{lllll}
%		\hline
%		$b(G)$ & $G$ & $d$ & $r$ & $\log |G|/\log |V|$ \\ \hline
%		2	& $A_8$, $S_8$ & 14 & 2 & 2 \\
%			& $c\times A_8$, $c\in \{3,5,15\}$ & 4 & 16 & 2 if $c \neq 3$, 1 if so. \\
%			& $A_8$, $\F_8^\times \times A_8$ & 4 & 8 & 2 \\
%			& $A_8$ & 4 & 4 & 2 \\
%			& $S_8 \leq G\leq \F_16^\times \times S_8$ & 6 & 16 & 1 \\
%			& $A_8 \leq G \leq \F_8^{\times} \times S_8$ & 6 & 8 & 1 except $\F_8^\times \times S_8$ \\
%			& $A_8$, $\F_4^\times \times A_8$ & 6 & 4 & 2 \\
%		3 	& $\F_4^\times \times A_8$ & 4 & 4 & 2 \\
%			& $S_8$, $\F_4^\times \times S_8$ & 6 & 4 & 2 \\
%		4	& $A_8$ & 4 & 2 & 4 \\
%			& $A_8$ & 6 & 2 & 3 \\
%			5 & $S_8$ & 6 & 2 & 3 \\
%		\bottomrule
%	\end{tabular}
%\end{table}

\begin{abstract}
Let $G \leq \mathrm{GL}(V)$ be a group with a unique subnormal quasisimple subgroup $E(G)$ that acts absolutely irreducibly on $V$. A base for $G$ acting on $V$ is a set of vectors with trivial pointwise stabiliser in $G$. 
In this paper we determine the minimal base size of $G$ when $E(G)/Z(E(G))$ is a finite simple group of Lie type in cross-characteristic. 
%We build on work of K\"ohler and Pahlings, who analyse the case where $(|G|,|V|)=1$. 
We show that $G$ has a regular orbit on $V$, with specific exceptions, for which we find the base size.
\end{abstract}
\maketitle
\section{Introduction}
Suppose $G \leq \mathrm{Sym}(\Omega)$ is a permutation group.
We say that $G$ has a \textit{regular orbit} on $\Omega$ if there exists $\omega \in \Omega$ with trivial stabiliser in $G$. 
The study of regular orbits arises in a number of contexts, particularly when $\Omega$ is a vector space $V$, and $G$ acts linearly on $V$.
For example, if $V$ is a finite vector space, $G\leq \mathrm{GL}(V)$, and all of the orbits of $G$ on $V\setminus \{0\}$ are regular, then the affine group $GV$ is a Frobenius group with Frobenius complement $G$, and such $G$ were classified by Zassenhaus \cite{MR3069653}. Regular orbits are also used heavily in the proof of the well-known $k(GV)$-conjecture \cite{MR2078936}, which states that the number of conjugacy classes $k(GV)$ of $GV$, with $|G|$ coprime to $|V|$, is at most $|V|$. One of the major cases was where $G$ is an almost quasisimple group acting irreducibly on $V$. Here, proving the existence of a regular orbit of $G$ on $V$ was sufficient to prove the $k(GV)$-conjecture. A classification of such pairs $(G,V)$ was completed by K\"{o}hler and Pahlings \cite{MR1829482}, based on work of Goodwin \cite{goodwin1,goodwin2} and Liebeck \cite{MR1407889}.

We say that $B \subset V$ is a \textit{base} for $G$ if its pointwise stabiliser in $G$ is trivial, and call the minimal size of a base for $G$ the \textit{base size} and denote it by $b(G)$. Obviously if $G$ has a regular orbit on $V$, then $b(G)=1$.
Since each element of a group is characterised by its action on a base, we have $b(G) \geq \lceil \log |G|/\log |V| \rceil$.
Notice that if $G$ has a regular orbit on an irreducible $G$-module $V$, then $GV\leq \mathrm{AGL}(V)$ has $b(GV)=2$.
In recent years there has been significant progress towards the more general aim of classifying finite primitive groups $H$ with $b(H)=2$. This includes a partial classification for diagonal type groups \cite{MR2998958}, a complete classification for primitive actions of $S_m$ and $A_m$ \cite{MR2781219, MR2214474} and sporadic groups \cite{MR2684423} and also progress for almost simple classical groups \cite{MR3219555}.  Recall that a group $H$ is \textit{almost simple} if there exists a non-abelian finite simple group $S$ such that $S\leq H\leq \mathrm{Aut}(S)$.

We say that a group $G$ is \textit{quasisimple} if it is perfect and $G/Z(G)$ is a non-abelian simple group.  We further define $G$ to be \textit{almost quasisimple} if $G$ has a unique quasisimple subnormal subgroup, which forms the layer $E(G)$ of $G$, and the quotient $G/F(G)$ of $G$ by its Fitting subgroup $F(G)$ is almost simple. 
Since the Fitting subgroup $F(G)$ of $G$  commutes with the quasisimple subgroup $E(G)$, if we have $G \leq  \mathrm{GL}(V)$ with absolutely irreducible restriction to $E(G)$, then by \cite[Lemma 2.10.1]{KL}, $F(G)$ is equal to the group $F_0(G)$ of scalar matrices in $G$. 

A large open case in the classification of primitive groups $H$ with $b(H)=2$ is where $H=GV$ is of affine type, $(|G|,|V|)>1$ and $G \leq \mathrm{GL}(V)$ is an almost quasisimple group, whose quasisimple layer $E(G)$ acts absolutely irreducibly on $V$. 
For groups of affine type, Aschbacher's Theorem \cite{asch_thm} is a useful tool in classifying groups $G$ with a regular orbit, since it determines the possiblilities for irreducible subgroups of $ \gl(V)$. The groups $G$ that we consider in this paper are members of the $\mathcal{C}_9$ class of Aschbacher's theorem.

The classification of pairs $(G,V)$ where $G$ has a regular orbit on $V$ has been completed for most groups $G$ with $\mathrm{soc}(G/F(G))$ a sporadic group, or an alternating group $Alt_n$ \cite{MR3893366,MR3500766}.  In these cases, the authors adopt a stricter definition of almost quasisimple than is given here by requiring $G/Z(G)$ to be almost simple. Therefore,  further work is required to complete a classification according to our definition of almost quasisimple when $\mathrm{soc}(G/F(G))$ is isomorphic to either $Alt_6$ or $Alt_7$, or a sporadic group with Schur multiplier of order greater than 2, and outer automorphism group of order 2.
%$G$ with $\mathrm{soc}(G/F(G))$ a sporadic group, or an alternating group $Alt_n$ with $n\neq 6,7$ \cite{MR3893366,MR3500766}. The latter also covers $G$ with $\mathrm{soc}(G/F(G)) \cong Alt_6, Alt_7$, where also $G/Z(G)$ is almost simple. The other almost quasisimple groups with $\mathrm{soc}(G/F(G)) \cong Alt_6$ are not considered because the authors of \cite{MR3500766} adopt a stricter definition of almost quasisimple than is given here. 
In the case of linear algebraic groups, Guralnick and Lawther \cite{bigpaper} classified $(G,V)$ where $G$ is a simple algebraic group that has a regular orbit on the irreducible $G$-module $V$, as well as determining their generic stabilisers in these cases. Their methods rely heavily on detailed analyses of highest weight representations of these simple algebraic groups. 

This paper is the first in a series of three, which complete the analysis of base sizes of $(G,V)$, where $G$ has a unique subnormal quasisimple subgroup $E(G)$ which is a group of Lie type that acts absolutely irreducibly on $V$. The present paper deals with the case where $V$ is an absolutely irreducible module for $E(G)$ in cross characteristic. The other two papers will address representations of the groups of Lie type in defining characteristic. 
%Throughout, we will assume that $E(G)/Z(E(G))$ is not isomorphic to an alternating group, since this has been covered in \cite{MR3500766}.

Our main result is as follows.
\begin{theorem}
\label{mainthm}
Let $V= V_d(r)$ be a vector space of dimension $d$ over $\mathbb{F}_r$. Suppose $G\leq \mathrm{GL}(V)$ and that the layer $E(G)$ of $G$ is quasisimple of Lie type in characteristic $p \nmid r$ acting absolutely irreducibly on $V$. Then one of the following holds.
\begin{enumerate}
\item $b(G) = \lceil \log |G|/\log |V| \rceil$,
\item $b(G) =  \lceil \log |G|/\log |V| \rceil+1$ and $(G, V)$ is  in Table \ref{allbad},
\item $(G,V) =( U_4(2) .2, V_5(3))$, $(2\times  U_4(2) , V_5(3))$ or $( U_4(2) .2, V_6(2))$  and $b(G)= \lceil \log |G| /\log |V| \rceil+2=4,4$ or $5$ respectively.
%\item {\color{red} Needs checking } $E(G)/Z(E(G)) \cong \pom_8^+(2)$, $d=8$ and $r \in \{3,5,7,9,25,27\}$, or
%(2\times \mathrm{PSp}_6(2), V_7(3))$, $

\end{enumerate}
\end{theorem}
{\tiny
\begin{table}[h!]
\begin{tabular}{lllll} 
	\toprule
	$b(G)$  & $G$  & $d$  & $r$  &  \\ 
	\midrule
	2 
	%& $\F_r^{\times} \circ (2.L_2(4))$  & 2 & 11 & $\dag$ \\ %log/log >1
	& $3 \circ (2.L_2(4))$,  & 2 & 19 & $\dag$ \\ %$c\in \{3,9,18\}$ 9,18 |G|>|V|
	& $6 \circ(2. L_2(4))$  & 2 & 25 &  \\
	& $c\circ (2.L_2(4).2)$,  $c\in \{1,2,4\}$  & 2 & 25 &  \\ %$c \mid 24$, $c\leq 8$ c=3,6,8 |G|>|V|
	& $4\circ(2.L_2(4).2i)$ & 2 & 25 &  \\
	& $4 \circ (2.L_2(4))$  & 2 & 29 & $\dag$ \\ 
	& $20 \circ (2.L_2(4))$  & 2 & 41 & $\dag$ \\ 
	& $c \circ (2.L_2(4))$, $c\in \{12, 24\}$  & 2 & 49 & $\dag$ \\
	& $\F_{61}^{\times} \circ (2.L_2(4))$  & 2 & 61 & $\dag$ \\
	& $L_2(4)\leq G< 2\times L_2(4).2$  & 3 & 5 &  \\
	& $c \times L_2(4)$, $1\neq c\mid 8$  & 3 & $9$  &  \\
	& $\F_{11}^\times \times L_2(4)$  & 3 & 11 & $\dag$ \\
	& $L_2(4)$  & 4 & $3$  &  \\
	& $L_2(4) < G < \F_4^\times \times L_2(4).2$  & 4 & $4$  &  \\
	& $\F_5^\times \circ(2.L_2(4))$  & 4 & 5 &  \\
	& $c \times L_2(4).2$, $c\in \{2,6\}$  & 4 & 7 & $\dag$ \\
	& $3\times L_2(8)$  & 2 & 64 &  \\
	& $\F_8^\times \times L_2(8)$  & 4 & 8 &  \\
	& $\F_3^\times \times L_2(8)$  & 7 & 3 &  \\
	& $L_2(8).3$  & 7 & 3 &  \\
	& $10 \circ (2.L_2(9))$  & 2 & 81 &  \\
&	$c\circ(2. L_2(9).2_2)$, $c\in \{1,2,4,8\}$             & 2&81 \\
	& $L_2(9) \leq G <2\times L_2(9).2_2$. & 3 & 9 &  \\
	%& $\F_{16}^\times \circ(3.L_2(9))$  & 3 & 16 &  \\ |V|<|G|
	& $c \circ (3.L_2(9))$, $c\in \{2,6,18\}$  & 3 & 19 & $\dag$ \\
	& $c \circ (3.L_2(9))$, $c \mid 24$ even  & 3 & 25 &  \\
	& $c \circ (3.L_2(9).2_3)$, $c \in \{2,4,12\}$  & 3 & 25 &  \\
	& $c \circ (3.L_2(9))$, $c \in \{10,30\}$  & 3 & 31 & $\dag$ \\
	& $\F_7^\times \circ (2.L_2(9))$ &4&7& $\dag$\\
	& $L_2(9).2_1$  & 4 & 8 &  \\
	& $L_2(9).2_1\leq G\leq L_2(9).2^2$  & 4 & 9 &  \\
	& $2\times L_2(9) \leq G < \F_9^\times \times L_2(9).2^2$  & 4 & 9 &  \\
	& $L_2(9) <G\leq \F_5^\times \times L_2(9).2_1$  & 5 & 5 &  \\
	& $c\times L_2(9).2_1$, $c\in \{2,3,6\}$  & 5 & 7 & $\dag\star_1$ \\
	& $c\times L_2(9).2^2$,  $c\in\{1,2\}$ & 9 & $3$  & $\ddag$ \\
	& $L_2(11)$  & 5 & 4 &  \\
	& $c\times L_2(11) $, $c\mid 4$   & 5 & 5 &  \\
%	& $2.L_2(11)$  & 6 & 3 &  \\ log/log>1
%	& $L_2(11)\leq G \leq L_2(11).2$  & 10 & 2 &  \\log/log>1
   & $L_2(11)$  & 10 & 2 &  \\
	& $\F_4^{\times}\times L_2(13)$  & 6 & 4 &  \\
	& $c \times L_2(13)$,  $c\in \{1,2\}$  & 7 & 3 &  \\
	& $L_2(13).2$  & 14 & 2 &  \\
	& $L_2(31)$  & 15 & 2 &  \\
	& $c\circ( 2.L_3(2))$, $c\in \{4,8,12\}$  & 2 & 49 &  \\
	& $4\circ (2.L_3(2).2)$  & 2 & 49 &  \\
	& $L_3(2) \leq G < 2\times L_3(2).2$  & 3 & 7 &  \\
	& $c\times L_3(2)$, $c\in \{2,4\}$  & 3 & 9 &  \\
& $2\times L_3(2)$  & 3 & 11 & $\dag$ \\
	& $L_3(2) \leq G \leq \F_3^\times \times L_3(2).2$  & 6 & 3 &  \\
	& $L_3(2)$  & 8 & 2 &  \\
	& $2.L_3(4) \leq G< \F_9^\times \circ (2.L_3(4).2^2)$  & 6 & 9 &  \\
	& $4_1.L_3(4).2_3$  & 8 & 5 &  \\
	& $L_4(2) \leq G <2 \times L_4(2).2$  & 7 & 5 &  \\
	& $L_4(2) < G \leq \F_7^{\times} \times L_4(2).2$  & 7 & 7 &  \\
	\bottomrule
\end{tabular}
\quad 
\begin{tabular}{lllll} 
	\toprule
	$b(G)$  & $G$  & $d$  & $r$  &  \\ 
	\midrule
	2 
	& $c\times  \times L_4(2).2$, $1\neq c \mid 8$  & 7 & 9 &  \\
	%& $2.L_4(2) \leq G \leq \F_5^\times \circ (2.L_4(2).2)$  & 8 & 5 &  \\
	&$ \F_5^\times \circ (2.L_4(2).2)$  & 8 & 5 &  \\

	& $L_4(3) \leq G \leq L_4(3).2_2$  & 26 & 2 &  \\ 
	\cmidrule{2-5}
	& $c \times U_3(3)$, $c=2k \mid 80$  & 3 & 81 &  \\
	& $2\times U_3(3)$  & 6 & 5 & $\dag$ \\
	& $\F_7^\times \times U_3(3)$  & 6 & 7 &  \\
	& $c \times (U_3(3).2)$, $c\in \{2,6\}$  & 6 & 7 &  \\
		& $c\times U_3(3).2$, $c\mid 4$  & 7 & 5 & $\dag\star_2$ \\
	& $c\times U_3(3)$, $c \mid 4$  & 7 & 5 & $\dag *$ \\
	& $U_3(3)\leq G\leq U_3(3).2$ & 14 & 2 &  \\
	& $ \F_3^\times\times U_3(4)$  & 12 & 3 &  \\
	& $U_3(5)< G\leq U_3(5).S_3$,  & 20 & 2 &  \\
	& $ U_4(2) $  & 4 & 16 &  \\
	& $c \times  U_4(2) $, $c\in \{6,12\}$  & 4 & 25 &  \\
	& $2. U_4(2)  \leq G\leq 2. U_4(2) .2$  & 4 & 27 &  \\
	& $ \mathbb{F}_r^\times \circ(2. U_4(2) )$  & 4 & 31,37 & $\dag$ \\
	& $ 12\circ(2. U_4(2) )$  & 4 & 37 & $\dag$ \\
	& $c\times  U_4(2) $, $c\in \{3,9,21,63\}$  & 4 & 64 &  \\
	& $ U_4(2)  \leq G < 2\times  U_4(2) .2 $  & 5 & 9 &  \\
	& $ U_4(2)  < G \leq \F_{13}^\times\times  U_4(2) $  & 5 & 13 & $\dag$ \\
	& $c\times  U_4(2) , c\in \{6,18\}$  & 5 & 19 & $\dag$ \\
	& $c\times  U_4(2) $, $c=2k\mid 26$  & 5 & 27 &  \\
	& $c\times( U_4(2) .2) $, $c \mid 26$  & 5 & 27 &  \\
	& $ U_4(2)  \leq G \leq 2\times  U_4(2) .2$  & 6 & 7 & $\dag$ \\
& $ 3\times U_4(2) $  & 6 & 7 & $\dag$ \\
	& $ U_4(2) \leq G<\F_8^\times \times U_4(2).2$ & 6 & 8 &  \\
	& $ \F_{11}^\times\times  U_4(2) $  & 6 & 11 & $\dag$ \\
%	& $ U_4(2) .2$  & 6 & 11 & $\dag$ \\
	& $ c \times  U_4(2) .2$, $1\neq c\mid (r-1)$  & 6 & 11,13 & $\dag\star_3$ \\
	& $ c \times  U_4(2) .2$,  $c\mid 15$  & 6 & 16 &  \\
	& $ U_4(2)  \leq G < \F_3^\times\times U_4(2) .2$  & 10 & 3 &  \\
	& $3_1.U_4(3)$  & 6 & 16 &  \\
	& $6_1.U_4(3)\leq G \leq 6_1.U_4(3).2_2$  & 6 & 19 & $\dag$ \\
	& $c\circ(6_1.U_4(3))$, $c \mid 24$  & 6 & 25 &  \\
	& $c\circ(6_1.U_4(3).2_2)$, $c \mid (r-1)$  & 6 & 25,49 &  \\
	& $\F_{31}^\times \circ (6_1.U_4(3))$& 6 & 31 & $\dag$ \\
	& $c \circ(6_1.U_4(3).2_2)$, $c\mid (r-1)$  & 6 & 31,37 & $\dag$ \\
%	& $U_4(3).D_8$  & 20 & 2 &  \\ |V|<|G|
	& $c \times U_5(2).2$, $1\neq c \mid 8$  & 10 & 9 &  \\
	& $\F_9^\times \times U_5(2)$  & 10 & 9 &  \\

	& $c \times U_5(2)$, $1\neq c \mid 6$  & 10 & 7 & $\dag$ \\
	& $U_6(2)$  & 21 & 3 &  \\ 
		\cmidrule{2-5}
	 & $c\times \mathrm{PSp}_6(2)$, $c\in \{1,2\}$  & 7 & 9 &  \\
	& $c \times \mathrm{PSp}_6(2)$, $c\mid (r-1)$  & 7 & 11,13 & $\dag$ \\
		& $c \times \mathrm{PSp}_6(2)$, $1\neq c \mid 16$  & 7 & 17 & $\dag$ \\
	& $c\times \mathrm{PSp}_6(2)$, $c\in \{6,18\}$  & 7 & 19 & $\dag$ \\
	& $c\circ (2.\mathrm{PSp}_6(2))$, $c\in\{1,2\}$  & 8 & 7 &  \\
		& $c\circ (2.\mathrm{PSp}_6(2))$, $c\mid 8$  & 8 & 9 &  \\
	& $c\times \mathrm{PSp}_6(2)$, $c \in \{1,2\}$  & 14 & 3 &  \\
	& $c\times \mathrm{PSp}_6(3)$, $c\in \{2,3,6\}$  & 13 & 7 &  \\
	& $2.\mathrm{PSp}_6(3) \leq G \leq \F_7^\times \circ (2.\mathrm{PSp}_6(3).2)$  & 14 & 7 &  \\ 
	%	\cmidrule{2-5}
%	& $c\circ( 2.\pom_8^+(2))$, $c\mid 10$  & 8 & $11$  & $\dag$ \\|V|<|G|
&&&\\
	\bottomrule
\end{tabular}
\caption{Pairs $(G,V)$ with $G$ an almost quasisimple group of Lie type for which $b(G) = \lceil \log |G|/\log |V| \rceil+1$,  Part I.}
\label{allbad}
\end{table}
\begin{table}[h!]
\ContinuedFloat
\begin{tabular}{lllll} 
	\toprule
	$b(G)$  & $G$  & $d$  & $r$  &  \\ 
	\midrule
	2	& $c\circ( 2.\pom_8^+(2))$, $c \in \{2,4\}$  & 8 & $13$  & $\dag$ \\
& $2.\pom_8^+(2).2$ & 8 & $13$  & $\dag$ \\
& $2.\pom_8^+(2).2i$ & 8 & $13$  & $\dag$ \\
	& $c\circ( 2.\pom_8^+(2).2i)$, $c\mid 16$  & 8 & $17$  & $\dag$ \\
	& $2.\pom_8^+(2) \leq G\leq \mathbb{F}_r^\times\circ( 2.\pom_8^+(2).2)$  & 8 & $17,19, 23$  & $\dag$ \\
		& $c\circ( 2.\pom_8^+(2).2i)$, $c\mid 24$   & 8 & $25$  &  \\
	& $2.\pom_8^+(2) \leq G\leq \mathbb{F}_r^\times\circ( 2.\pom_8^+(2).2)$  & 8 & $25,27$  &  \\ 
	& $c\circ( 2.\pom_8^+(2).2)$, $c \mid 28$,  & 8 & $29$  &  $\dag$\\
	& $c\circ( 2.\pom_8^+(2).2i)$, $c=4k \mid 28$,  & 8 & $29$  &$\dag$  \\ 
	& $c\circ( 2.\pom_8^+(2))$, $c \mid 30$, $c>2$  & 8 & $31$  &$\dag$ \\ 
		& $c\circ( 2.\pom_8^+(2).2)$, $c \mid 30$, $c>2$  & 8 & $31$  & $\dag$ \\ 
	\cmidrule{2-5}
	& $c\circ( 2.^2B_2(8) )$, $c \mid 4$  & 8 & 5 &  \\
	& $ \F_7^\times\circ (2.G_2(4))\leq G\leq \F_7^\times\circ (2.G_2(4).2)$  & 12 & 7 &  \\ 
	\midrule
	3 &$L_2(4).2$  & 4 & 2 &  \\
	& $\F_3^\times \times L_2(9)\leq G\leq \F_3^\times \times L_2(9).2_3$  & 4 & 3 &  \\
	& $c\times L_2(9).2_1$, $c\in \{1,2\}$  & 4 & 3 &  \\
	& $2.L_3(4)\leq G \leq 2.L_3(4).2^2$  & 6 & 3 &  \\
	\bottomrule
	\end{tabular}
	\begin{tabular}{lllll} 
	\toprule
	$b(G)$  & $G$  & $d$  & $r$  &  \\ 
	\midrule
	3& $\F_3^\times \times L_4(2).2$  & 7 & 3 &  \\ 
	\cmidrule{2-5}
	& $c \times U_3(3)$, $c\in \{2,4,8\}$  & 3 & 9 &  \\
	& $ U_4(2) $  & 4 & 4 &  \\
	& $ U_4(2) $  & 5 & 3 &  \\
	& $ U_4(2) .2 \leq G \leq \F_4^\times \times U_4(2) .2 $  & 6 & 4 &  \\
	%& $U_4(2) $, $c\in \{1,3\}$  & 6 & 4 &  \\ b(G)=2 for both - ceiling of log/log
	& $3_1.U_4(3)$  & 6 & 4 &  \\
	& $ \F_3^\times\times U_5(2).2$  & 10 & 3 &  \\ 
	\cmidrule{2-5}
	& $c\times \mathrm{PSp}_6(2)$, $c\in \{1,2\}$  & 7 & 3 &  \\
	& $2.\mathrm{PSp}_6(2)$  & 8 & 3 &  \\ 
	\cmidrule{2-5}
	& $2.\pom_8^+(2) \leq G\leq \F_5^\times\circ( 2.\pom_8^+(2).2)$  & 8 & 5 &  \\ 
	\midrule
	4 & $L_2(9).2_1$  & 4 & 2 &  \\
	\cmidrule{2-5}
	& $2. U_4(2) \leq G\leq 2.U_4(2).2$ & 4 & 3 &  \\
	& $ \F_4^\times\times  U_4(2) $  & 4 & 4 &  \\
	& $ \F_3^\times\times  U_4(2) .2 $  & 5 & 3 &  \\
	& $ U_4(2) $  & 6 & 2 &  \\ 
	\cmidrule{2-5}
	& $2.\pom_8^+(2) \leq G\leq 2.\pom_8^+(2).2$  & 8 & 3 &  \\
%	&&&&\\
	\bottomrule
	\end{tabular}
	\caption{Pairs $(G,V)$ with $G$ an almost quasisimple group of Lie type for which $b(G) = \lceil \log |G|/\log |V| \rceil+1$,  Part II.}
\end{table}
}
%\begin{table}[h!]
%\begin{tabular}{@{}ccc@{}}
%\toprule
%$G$ & $(d,r)$ & $b(G)$ \\ \midrule
%$2.\mathrm{PSp}_4(3),2.\mathrm{PSp}_4(3).2$ & $V_4(3)$ & 4 \\
%$\mathrm{PSp}_4(3)$ & $V_4(4)$ & 3 \\
%$3\circ \mathrm{PSp}_4(3)$ & $V_4(4)$ & 4 \\
%$\mathrm{PSp}_4(3)$ & $V_4(16)$ & 2 \\
%$\mathrm{PSp}_4(3)$,  $c\in \{6,12,24\}$ & $V_4(25)$ & 2 \\
%$2.\mathrm{PSp}_4(3) \leq G\leq 2.\mathrm{PSp}_4(3).2$ & $V_4(27)$ & 2 \\
%$c.\mathrm{PSp}_4(3)$, $c\in \{3,9,21,63\}$ & $V_4(64)$ & 2 \\
%$\mathrm{PSp}_4(3)$ & $V_5(3)$ & 3 \\
%$\mathrm{PSp}_4(3)< G <2\circ \mathrm{PSp}_4(3).2 $ & $V_5(3)$ & 4 \\
%$2\circ \mathrm{PSp}_4(3).2 $ & $V_5(3)$ & 4 \\
%$\mathrm{PSp}_4(3) \leq G < 2\circ  \mathrm{PSp}_4(3).2$ & $V_5(9)$ & 2 \\
%$\mathrm{PSp}_4(3).2 $ & $V_5(27)$ & 2 \\
%$2\circ \mathrm{PSp}_4(3)< G \leq 26\circ \mathrm{PSp}_4(3).2 $ & $V_5(27)$ & 2 \\
%$\mathrm{PSp}_4(3)$ & $V_6(2)$ & 4 \\
%$\mathrm{PSp}_4(3).2 \leq 3 \circ\mathrm{PSp}_4(3).2  $ & $V_6(4)$ & 3 \\
%$c\circ \mathrm{PSp}_4(3)$, $c\in \{1,3\}$ & $V_6(4)$ & 3 \\
%$\mathrm{PSp}_4(3), \mathrm{PSp}_4(3).2$ & $V_6(8)$ & 2 \\
%$\mathrm{PSp}_4(3).2 \leq G \leq 15 \circ \mathrm{PSp}_4(3).2$ & $V_6(16)$ & 2 \\
%$\mathrm{PSp}_4(3) \leq G < 2\circ\mathrm{PSp}_4(3).2$ & $(10,3)$ & 2 \\ \bottomrule
%\end{tabular}
%\caption{Base sizes of irreducible $G$-modules with $E(G)/Z(E(G)) \cong \mathrm{PSp}_4(3)$ where $b(G)>1$.}
%\label{s43-noros}
%\end{table}

Combining the cases in Table \ref{allbad} with the cases where $b(G) = \lceil \log |G|/ \log |V| \rceil >1$, we obtain the following corollary.

{\scriptsize
\begin{table}[h!]
\begin{tabular}{lllll} 
	\toprule
	$b(G)$  & $G$  & $d$  & $r$  &  \\ 
	\midrule
	2 & $L_2(4)$  & 2 & 4 &  \\
	& $2.L_2(4) \leq G \leq \F_5^\times \circ(2. L_2(4))$  & 2 & 5 &  \\
	& $2.L_2(4) \leq G \leq \F_9^\times \circ(2.L_2(4))$  & 2 & 9 &  \\
	& $c \circ (2.L_2(4))$, $c\mid r-1$, $c\neq 1,2$  & 2 & 11, 19 & $\dag$ \\
		& $c \times L_2(4)$, $c \in \{5,15\}$  & 2 & 16 &  \\
	%& $c \times L_2(4).2$, $c \in \{5,15\}$  & 2 & 16 &  \\ Mod doesn't extend to .2
	& $c \circ(2. L_2(4))$, $c\in \{6,12,24\}$  & 2 & 25 &  \\
	& $c \circ (2.L_2(4))$, $c\in \{4,28\}$  & 2 & 29 & $\dag$ \\
	& $c\circ (2.L_2(4))$, $c\in \{\frac{r-1}{2}, r-1\}$  & 2 & 31, 41 & $\dag$ \\
	& $\F_{61}^{\times} \circ (2.L_2(4))$  & 2 & 61 & $\dag$ \\
	& $c \circ (2.L_2(4))$, $c\in \{12, 24, 48\}$  & 2 & 49 & $\dag$ \\
	&  $c\circ(2.L_2(4).2)$, $c\mid 24$  & 2 & 25 &  \\
		& $c\circ(2.L_2(4).2i)$, $c\mid 24$, $c>2$  & 2 & 25 &  \\
	& $L_2(4) \leq G\leq \F_5^\times \times L_2(4).2$  & 3 & 5 &  \\
	& $L_2(4)< G \leq \F_9^\times \times L_2(4)$  & 3 & $9$  &  \\
	& $\F_{11}^\times \times L_2(4)$  & 3 & 11 & $\dag$ \\
	& $L_2(4)$  & 4 & 2 &  \\
	& $L_2(4)\leq G\leq \F_3^\times\times L_2(4).2$  & 4 & 3 &  \\
	& $L_2(4) < G \leq \F_4^\times \times L_2(4).2$  & 4 & $4$  &  \\
	& $\F_5^\times \circ(2.L_2(4))$  & 4 & 5 &  \\
	& $c\times L_2(4).2$, $c \in\{2,6\}$  & 4 & 7 & $\dag$ \\
	& $c\times L_2(8)$, $c\mid 7$  & 2 & 8 &  \\
	& $c\times L_2(8)$, $c=3k \mid 63$  & 2 & 64 &  \\
	& $ \F_8^\times\times L_2(8)$  & 4 & 8 &  \\
	& $ \F_3^\times\times L_2(8)$  & 7 & 3 &  \\
	& $L_2(8)\leq G\leq L_2(8).3$& 8 & 2 &  \\
	& $c\times(L_2(8).3)$, $c \in \{1,2\}$  & 7 & 3 &  \\
	& $2.L_2(9)\leq G\leq 2.L_2(9).2_2$  & 2 & 9 &  \\
	& $c \circ (2.L_2(9))$, $c\in \{10, 20,40, 80\}$  & 2 & 81 &  \\
	& $c \circ (2.L_2(9)).2_2$, $c\mid 80$  & 2 & 81 &  \\
	& $3.L_2(9)$  & 3 & 4 &  \\
	& $L_2(9) \leq G \leq \F_9^\times \times L_2(9).2_2$  & 3 & 9 &  \\
	& $\F_{16}^\times \circ(3.L_2(9))$  & 3 & 16 &  \\
%	& $c \circ (3.L_2(9))$, $c\mid 18$  & 3 & 19 & $\dag$ \\
%	& $c \circ (3.L_2(9))$, $c \in \{5,15,30\}$  & 3 & 31 & $\dag$ \\
	& $c \circ (3.L_2(9))$, $c\in \{2,6,18\}$  & 3 & 19 & $\dag$ \\
	& $c \circ (3.L_2(9))$, $c \in \{10,30\}$  & 3 & 31 & $\dag$ \\
	& $c \circ (3.L_2(9))$, $c \mid 24$ even  & 3 & 25 &  \\
	& $c \circ (3.L_2(9).2_3)$, $c \mid 24$ even  & 3 & 25 &  \\
	& $L_2(9)\leq G\leq L_2(9).2_3$  & 4 & 3 &  \\
	& $L_2(9) \leq G \leq \F_4^\times \times L_2(9).2_1$  & 4 & 4 &  \\
	& $2.L_2(9) \leq G \leq \F_5^\times \circ(2.L_2(9).2_1)$  & 4 & 5 &  \\
	& $\F_7^\times \circ(2.L_2(9))$  & 4 & 7 & $\dag$ \\
		& $\F_7^\times \circ(2.L_2(9).2_1)$  & 4 & 7 & $\dag$ \\
	& $c\times L_2(9).2_1$,  $c\in\{1,7\}$ & 4 & 8 &  \\
	& $L_2(9).2_1\leq G\leq L_2(9).2^2 $  & 4 & 9 &  \\
	& $2\times L_2(9) \leq G \leq \F_9^\times \times L_2(9).2^2$  & 4 & 9 &  \\
	& $L_2(9) <G\leq \F_5^\times \circ(L_2(9).2_1)$  & 5 & 5 &  \\
	& $c\times L_2(9).2_1$, $c\in \{2,3,6\}$  & 5 & 7 & $\dag \star_1$ \\
	& $c\times L_2(9).2^2$,  $c \in \{1,2\}$ & 9 & 3 & $\ddag$ \\
		& $c \times L_2(11)$, $c\mid (r-1)$  & 5 & 3,4,5 &  \\
	& $2.L_2(11)$  & 6 & 3 &  \\
	& $L_2(11)\leq G\leq L_2(11).2$  & 10 & 2 &  \\
	& $2.L_2(13)$  & 6 & 3 &  \\
	& $ \F_4^\times\times L_2(13)$  & 6 & 4 &  \\
	& $c \times L_2(13)$, $c\in \{1,2\}$  & 7 & 3 &  \\
	& $ L_2(13).2$ & 14 & 2 &  \\
	& $L_2(17)$  & 8 & 2 &  \\
	& $L_2(23)$  & 11 & 2 &  \\
	& $L_2(25) \leq G\leq L_2(25).2_2$  & 12 & 2 &  \\
	& $L_2(31)$  & 15 & 2 &  \\
	& $c\circ( 2.L_3(2))$, $c\mid 6$  & 2 & 7 &  \\
	& $c\circ( 2.L_3(2))$, $c\mid 48$, $c\neq \{2,3,6\}$  & 2 & 49 &  \\
	& $c\circ( 2.L_3(2).2)$, $c\mid 48$, $c\neq \{2,3,6\}$  & 2 & 49 &  \\
	& $c\times L_3(2)$, $c \mid 3$  & 3 & 4 &  \\
	& $L_3(2) \leq G \leq \F_7^\times\times L_3(2).2$  & 3 & 7 &  \\
	& $ \F_8^\times\times L_3(2)$  & 3 & 8 &  \\
	&&&\\
	\bottomrule
\end{tabular}
\quad 
\begin{tabular}{lllll} 
	\toprule
	$b(G)$  & $G$  & $d$  & $r$  &  \\ 
	\midrule
	2  		& $c\times L_3(2)$, $1\neq c \mid 8$  & 3 & 9 &  \\
& $c\times L_3(2)$, $c\in\{2,10\}$  & 3 & 11 & $\dag$ \\
	& $L_3(2) \leq G \leq \F_3^\times\times L_3(2).2$  & 6 & 3 &  \\
	& $L_3(2)\leq G\leq L_3(2).2$  & 8 & 2 &  \\
	& $L_3(3) \leq G \leq L_3(3).2$  & 12 & 2 &  \\
	&$4_2.L_3(4) \leq G \leq \mathbb{F}_9^\times \circ (4_2.L_3(4).2_2)$ &4&9\\
	& $2.L_3(4) \leq G\leq \F_9^\times\circ (2.L_3(4).2^2)$  & 6 & 9 &  \\
	& $6.L_3(4) \leq G \leq 6.L_3(4).2_1$  & 6 & 7 &  \\
	& $4_1.L_3(4).2_3$  & 8 & 5 &  \\
	& $L_4(2)\leq G <\F_3^{\times} \times L_4(2).2$  & 7 & 3 &  \\
	& $L_4(2) \leq G \leq \F_5^{\times} \times L_4(2).2$  & 7 & 5 &  \\
	& $L_4(2) < G \leq \F_7^{\times} \times L_4(2).2$  & 7 & 7 &  \\
	& $L_4(2).2 < G\leq \F_9^\times \times L_4(2).2$  & 7 & 9 &  \\
	& $2.L_4(2)$  & 8 & 3 &  \\
%	& $2.L_4(2) \leq G \leq \F_5^\times \circ (2.L_4(2).2)$  & 8 & 5 &  \\
	& $\F_5^\times \circ (2.L_4(2).2)$  & 8 & 5 &  \\
	& $L_4(3) \leq G \leq L_4(3).2_2$  & 26 & 2 &  \\ 
	\cmidrule{2-5}
	& $U_3(3)$  & 3 & 9 &  \\
	& $c \times U_3(3)$, $c=2k \mid 80$  & 3 & 81 &  \\
	& $U_3(3)\leq G\leq \F_4^\times \times U_3(3).2$  & 6 & 4 &  \\
	& $U_3(3)<G\leq\F_5^\times\times U_3(3)$  & 6 & 5 & $\dag$ \\ %According to ref of x-char paper, doesn't extend. Also 1 RO of U_3(3)
	& $ \F_7^\times \times U_3(3)$  & 6 & 7 &  \\
	& $c \times U_3(3).2$, $c\in \{2,6\}$  & 6 & 7 &  \\
	& $U_3(3) \leq G \leq \F_3^\times\times U_3(3).2$  & 7 & 3 &  \\
	& $c\times U_3(3).2 $, $c\mid 4$  & 7 & $5$  & $\dag\star_2$ \\
	& $U_3(3) \leq G\leq \F_5^\times\times U_3(3)$  & 7 & $5$  & $\dag*$ \\
	& $U_3(3)\leq G\leq U_3(3).2$ & 14 & 2 &  \\
	& $ \F_3^\times\times U_3(4)$  & 12 & 3 &  \\
	& $U_3(5)<G\leq U_3(5).S_3$  & 20 & 2 &  \\
	& $c\circ(2. U_4(2) )$, $c \mid (r-1)$  & 4 & 7,13 &$\dag$  \\
	& $2. U_4(2)  \leq G \leq \F_9^\times\circ (2. U_4(2) .2)$  & 4 & 9 &  \\
		& $c\times  U_4(2) $, $c\mid 15$  & 4 & 16 &  \\
	& $c\circ(2. U_4(2) )$, $c \mid 18$, $c>2$  & 4 & 19 & $\dag$  \\
	& $c \circ (2. U_4(2)) $, $c=3k\mid 24$  & 4 & 25 &  \\
	& $2. U_4(2)  \leq G \leq \F_{27}^\times\circ (2. U_4(2) .2)$  & 4 & 27 &  \\
	& $ \mathbb{F}_r^\times \circ(2. U_4(2) )$  & 4 & 31,37 & $\dag$ \\
	& $ 12 \circ(2. U_4(2) )$  & 4 & 37 & $\dag$ \\
	& $c\times  U_4(2) $, $c\in \{3,9,21,63\}$  & 4 & 64 &  \\
	& $U_4(2)$& 5 &7&$\dag$\\
	& $ U_4(2)  < G \leq \mathbb{F}_r^\times \times  U_4(2) $  & 5 & 7,13 & $\dag$ \\
	& $ U_4(2)  \leq G \leq \F_9^\times\times U_4(2) .2$  & 5 & 9 &  \\
	& $c\times  U_4(2) , c\in \{6,18\}$  & 5 & 19 & $\dag$ \\
%	& $ U_4(2)  \leq G \leq \F_{27}^\times\times  U_4(2) .2$  & 5 & 27 &  \\ Superseeded by two next rows
	& $c\times( U_4(2) )$, $c\in \{2,26\}$  & 5 & 27 &  \\
	& $c\times( U_4(2) .2)$, $c\mid 26$  & 5 & 27 &  \\
	& $c\times  U_4(2) $, $c\in \{1,3\}$  & 6 & 4 &  \\
	& $ U_4(2) \leq G \leq \F_5^\times\times  U_4(2) .2$  & 6 & 5 &  \\
	& $ U_4(2)  \leq G \leq \F_7^\times\times  U_4(2) .2$  & 6 & 7 & $\dag$ \\
	& $ U_4(2) \leq G \leq \F_8^\times\times  U_4(2) .2$  & 6 & 8 &  \\
	& $ \F_{11}^\times\times  U_4(2) $  & 6 & 11 & $\dag$ \\
	& $ U_4(2) .2 < G \leq \mathbb{F}_r^\times \times  U_4(2) .2$  & 6 & 11,13 & $\dag \star_3$ \\
	& $c \times (  U_4(2) .2)$, $c\mid 15$  & 6 & 16 &  \\
	& $ U_4(2)  \leq G \leq \F_3^\times\times U_4(2) .2$  & 10 & 3 &  \\
	& $ U_4(2)\leq G\leq U_4(2).2$ & 14 & 2 &  \\
	& $6_1.U_4(3)\leq G\leq 6_1.U_4(3).2_2$  & 6 & 7 &  \\
	& $c\circ(3_1.U_4(3))$, $c\mid 15$  & 6 & 16 &  \\
	& $c\circ(3_1.U_4(3).2_2)$, $c\mid 15$  & 6 & 16 &  \\
	& $c \circ(6_1.U_4(3))$, $c\mid (r-1)$  & 6 & 13,19 & $\dag$ \\
	& $\F_{31}^\times \circ(6_1.U_4(3))$& 6 & 31 & $\dag$ \\
	& $c \circ(6_1.U_4(3).2_2)$, $c\mid (r-1)$  & 6 & 13,19 & $\dag$ \\
	& $c \circ(6_1.U_4(3).2_2)$, $c\mid (r-1)$  & 6 & 31,37 & $\dag$ \\
	& $c\circ(6_1.U_4(3))$, $c \mid 24$  & 6 & 25 &  \\
	& $c\circ(6_1.U_4(3).2_2)$, $c \mid (r-1)$  & 6 & 25,49 &  \\
	& $U_4(3)\leq G\leq U_4(3).D_8$  & 20 & 2 &  \\
	& $U_5(2)\leq G < \F_3^\times\times U_5(2).2$  & 10 & 3 &  \\
		& $U_5(2)\leq G \leq \F_5^\times \times U_5(2).2$  & 10 & 5 &  \\
\bottomrule
\end{tabular}
\caption{Pairs $(G,V)$ in Theorem \ref{mainthm} with $b(G) >1$, Part I. }
\label{noros}
\end{table}
\begin{table}[h!]
\ContinuedFloat
\begin{tabular}{lllll} 
	\toprule
	$b(G)$  & $G$  & $d$  & $r$  &  \\ 
	\midrule
		2					& $c \times U_5(2)$, $1\neq c\mid 6$  & 10 & 7 & $\dag$ \\
& $ \F_9^\times \times U_5(2)$  & 10 & 9 &  \\	
		& $c \times U_5(2).2$, $1<c \mid 8$  & 10 & 9 &  \\
& $U_6(2) \leq G \leq \F_3^\times\times (U_6(2).S_3)$  & 21 & 3 &  \\ 
		\cmidrule{2-5}
	& $\mathrm{PSp}_4(7)$  & 24 & 2 &  \\
	& $c\times \mathrm{PSp}_6(2)$, $c\mid (r-1)$  & 7 & 5,9 &  \\
	& $c \times \mathrm{PSp}_6(2)$, $c\mid (r-1)$  & 7 & 11,13 & $\dag$ \\
	& $c\times \mathrm{PSp}_6(2)$, $1\neq c \mid 16$  & 7 & 17 & $\dag$ \\
	& $c\times \mathrm{PSp}_6(2)$, $c \times \{6,18\}$  & 7 & 19 & $\dag$ \\
	& $c\circ (2.\mathrm{PSp}_6(2))$, $c\mid (r-1)$  & 8 & 5,7,9 &  \\
	& $c\times \mathrm{PSp}_6(2)$, $c \in \{1,2\}$  & 14 & 3 &  \\
	& $c\times \mathrm{PSp}_6(3)$, $c\mid 3$  & 13 & 4 &  \\
	& $c\times \mathrm{PSp}_6(3)$, $c\in \{2,3,6\}$  & 13 & 7 &  \\
	& $2.\mathrm{PSp}_6(3) \leq G \leq \F_7^\times\circ (2.\mathrm{PSp}_6(3).2)$  & 14 & 7 &  \\
		\cmidrule{2-5}
	& $2.\pom_8^+(2) \leq G\leq \mathbb{F}_r^\times\circ( 2.\pom_8^+(2).2)$  & 8 & $7,9,25,27$  &  \\
%	& $2.\pom_8^+(2) \leq G\leq \mathbb{F}_r^\times\circ( 2.\pom_8^+(2).2)$  & 8 & $13,17$  & $\dag$ \\
	& $2.\pom_8^+(2) \leq G\leq \mathbb{F}_r^\times\circ( 2.\pom_8^+(2).2)$  & 8 & $11, 13, 17$  & $\dag$ \\
		& $2.\pom_8^+(2) \leq G\leq \mathbb{F}_r^\times\circ( 2.\pom_8^+(2).2)$  & 8 & $19,23$  & $\dag$ \\
%			& $c\circ( 2.\pom_8^+(2).2i)$, $c\mid 16$  & 8 & $17$  & $\dag$ \\ included in previous lines
	%	& $c\circ( 2.\pom_8^+(2).2i)$, $c\mid 24$   & 8 & $25$  &  \\included in previous lines
	& $c\circ( 2.\pom_8^+(2).2)$, $c \mid 28$,  & 8 & $29$  &  $\dag$\\
	& $c\circ( 2.\pom_8^+(2).2i)$, $c=4k \mid 28$,  & 8 & $29$  &$\dag$  \\ 
	& $c\circ( 2.\pom_8^+(2))$, $c \mid 30$, $c>2$  & 8 & $31$  &$\dag$ \\ 
		& $c\circ( 2.\pom_8^+(2).2)$, $c \mid 30$, $c>2$  & 8 & $31$  & $\dag$ \\ 
%	& $2.\pom_8^+(2) \leq G\leq \mathbb{F}_r^\times\circ( 2.\pom_8^+(2).2)$  & 8 & $25,27$  &  \\ 
	\cmidrule{2-5}
	& $c\circ( 2.^2B_2(8) )$, $c \mid 4$  & 8 & 5 &  \\
	&$G_2(3)\leq G\leq G_2(3).2$ & 14&2\\
	& $2.G_2(4) \leq G\leq 2.G_2(4).2$  & 12 & $3$ &  \\
	& $2.G_2(4) \leq G \leq \F_5^\times\circ (2.G_2(4).2)$  & 12 & $5$ &  \\
	& $ \F_7^\times\circ (2.G_2(4))\leq G\leq  \F_7^\times\circ (2.G_2(4).2)$  & 12 & 7 &  \\ 
	\bottomrule
	\end{tabular}
	\quad
	\begin{tabular}{lllll} 
	\toprule
	$b(G)$  & $G$  & $d$  & $r$  &  \\ 
	\midrule
		3	&$L_2(4).2$  & 4 & 2 &  \\
		& $L_3(2)$  & 3 & 2 &  \\
	& $L_2(9)$  & 4 & 2 &  \\
	& $\F_3^\times \times L_2(9)\leq G \leq \F_3^\times \times L_2(9).2_3$  & 4 & 3 &  \\
	& $c \times L_2(9).2_1$, $c\in \{1,2\}$  & 4 & 3 &  \\
		& $2.L_3(4)\leq G \leq 2.L_3(4).2^2$  & 6 & 3 &  \\
	& $\F_3^\times \times L_4(2).2$  & 7 & 3 &  \\ 
	\cmidrule{2-5}
	& $c \times U_3(3)$, $c\in \{2,4,8\}$  & 3 & 9 &  \\
		& $U_3(3)\leq G\leq U_3(3).2$  & 6 & 2 &  \\
	& $ U_4(2) $  & 4 & 4 &  \\
	& $ U_4(2) $  & 5 & 3 &  \\
	& $ U_4(2) .2 \leq G \leq \F_4^\times \times U_4(2) .2 $  & 6 & 4 &  \\
	& $3_1.U_4(3)\leq G\leq 3_1.U_4(3).2_2$  & 6 & 4 &  \\
	& $ \F_3^\times\times U_5(2).2$  & 10 & 3 &  \\ 
	\cmidrule{2-5}
	& $c\times \mathrm{PSp}_6(2)$, $c\in \{1,2\}$  & 7 & 3 &  \\
	& $2.\mathrm{PSp}_6(2)$  & 8 & 3 &  \\ 
	\cmidrule{2-5}
	& $2.\pom_8^+(2) \leq G\leq \F_5^\times\circ( 2.\pom_8^+(2).2)$  & 8 & 5 &  \\ 
	\midrule
	4 & $L_2(9).2_1$  & 4 & 2 &  \\ 
	\cmidrule{2-5}
	& $2. U_4(2) \leq G\leq 2. U_4(2).2$  & 4 & 3 &  \\
	& $ \F_4^\times\times  U_4(2) $  & 4 & 4 &  \\
	& $ U_4(2) < G \leq \F_3^\times\times  U_4(2) .2 $  & 5 & 3 &  \\
	& $ U_4(2) $  & 6 & 2 &  \\
		%\cmidrule{2-5}
	%& $ \F_3^\times\times \mathrm{PSp}_6(2)$  & 7 & 3 &  \\ b(G)=3
	\cmidrule{2-5}
	& $2.\pom_8^+(2) \leq G\leq 2.\pom_8^+(2).2$  & 8 & 3 &  \\ 
	\midrule
	5 & $ U_4(2) .2$  & 6 & 2 &  \\
	%&&&\\
		\bottomrule
	\end{tabular}
\caption{Pairs $(G,V)$ in Theorem \ref{mainthm} with $b(G) >1$, Part II. }
\end{table}
}

\begin{corollary}
\label{maincoroll}
Let $V= V_d(r)$ be as above. Suppose $G\leq \mathrm{GL}(V)$ and that the layer $E(G)$ of $G$ is quasisimple of Lie type in characteristic $p \nmid r$ acting absolutely irreducibly on $V$. Then either $G$ has a regular orbit on $V$, or $(G,V)$ appears in Table \ref{noros}.
\end{corollary}

\begin{remark}
\label{gcd}
\begin{enumerate}
\item The almost quasisimple groups $G$ in Tables \ref{allbad} and \ref{noros} where $|F(G)|>|Z(E(G))|$ have the extension by $F(G)$ denoted by direct product notation $\times$ if $|Z(E(G))|=1$ and central  product notation $\circ$ otherwise.

 \item  The entries in Tables \ref{allbad} and \ref{noros} marked with $\dag$ can be computed from \cite{MR1829482}, which provides the largest group for each module where there is no regular orbit. We construct the relevant modules in GAP \cite{GAP4} and determine completely where there are regular orbits.  However, the authors of \cite{MR1829482} omit the proof for the 8-dimensional modules $V=V_8(q)$  of $2.\mathrm{P}\Omega^+_8(2).2$ in their paper. For $q$ prime,  these omitted modules are analysed by L\"ubeck \cite{lubeck2021orbits}, who also shows that there is no regular orbit of $c \circ (2.\mathrm{P}\Omega^+_8(2))$ for $c>2$ and $q=31$, correcting \cite{MR1829482}. For the prime powers $q=121, 169$, the existence of a regular orbit was instead confirmed using Magma \cite{Magma}.
 This deals entirely with the case where $(|G|, |V|)=1$, and the remainder of this paper is devoted to the case where $(|G|, |V|)>1$.
 \item We use $2.\pom_8^+(2).2$ Tables  \ref{allbad} and \ref{noros} to denote the isoclinic variant isomorphic to the Weyl group of $E_8$, and use $2.\pom_8^+(2).2i$ to denote the other group of this shape.  Isoclinic variants of other groups are also considered in this paper, but it is only here that the base sizes vary between isoclinic variants.
\item  In the cases $L_2(4) \cong L_2(5)$, $L_2(9) \cong \mathrm{PSp}_4(2)'$, $L_3(2) \cong L_2(7)$, $\mathrm{PSp}_4(3) \cong U_4(2)$, $L_2(8) \cong {}^2G_2(3)'$ and $U_3(3) \cong G_2(2)'$ where there is an exceptional isomorphism between two groups of different characteristics, only one of the groups appears in Tables \ref{allbad}, and \ref{noros}, and we complete the analysis in all characteristics dividing the order of the group. 
%As mentioned earlier, we do not consider groups of Lie type that are isomorphic to alternating groups.
\item We require some additional information for the asterisked entries in Tables \ref{allbad} and \ref{noros}. From the Brauer character table in \cite{modatlas}, $U_3(3)$ has three absolutely irreducible 7-dimensional modules over $\F_5$. Only one of these may be extended to $U_3(3).2$. The first asterisked line in Table \ref{noros} refers to this module, while the asterisked line in Table \ref{allbad}, and the second such line in Table \ref{noros} refers to one of the two other modules.
		\item The rows marked with $\star_1$ also require further explanation. There are four irreducible modules $V_5(7)$ for $G = c\times L_2(9).2_1$ with $c \mid 6$. There is never a regular orbit of $G$ on $V$ when $c=2,6$ and also no regular orbit for $c=3$ if the corresponding Brauer character value on class 2B is either  3 or $-1$. In these cases, $b(G)=2$. There is a regular orbit of $G$ on $V$ in the remaining cases. 
		\item In the rows marked with $\star_2$, there is no regular orbit of $G=U_3(3).2$ on $V_7(5)$ if the corresponding Brauer character value on class 2B is equal to 1 and instead $b(G)=2$.
		\item In the rows marked with $\star_3$, if $r=11$ then $U_4(2).2$ has no regular orbit on $V=V_6(r)$ if and only if the corresponding Brauer character has value 4 on the class 2C. Moreover, if $r=13$, then for $c \mid 12$ there is a regular orbit of $c\times U_4(2).2$ on $V$, if and only if either $c=1$, or $c=3$ and the corresponding Brauer character has value $-4$ on class 2C.
		\item In the rows of Tables \ref{allbad} and \ref{noros} marked with $\ddag$, we find that $L_2(9).2^2$ has a regular orbit on three of the four absolutely irreducible  $\F_3L_2(9)$-modules $V=V_9(3)$. The module with $b(G)=2$ has Brauer character values 3 and 1 on classes $2B$ and $2D$ respectively. 
\item The aforementioned papers \cite{MR3893366,MR3500766,goodwin1,goodwin2,MR1829482} all assume that $V$ (as in Theorem \ref{mainthm}) is a faithful $n$-dimensional $\F_{r_0}G$-module, with $r_0$ prime such that $E(G)$ acts irreducibly, but not necessarily absolutely irreducibly, on $V$. We can reconcile this with our study of absolutely irreducible modules for $E(G)$ over arbitrary finite fields, following \cite[\S3]{goodwin1}. Let $k = End_{\F_{r_0}G}(V)$, $K = End_{\F_{r_0}E(G)}(V)$, $t = |K:k|$ and $d = \dim_K(V)$. Then $E(G) \le \mathrm{GL}_d(K)$ is absolutely irreducible, and $G \le GL_d(K)\langle \phi \rangle$, where $\phi$ is a field automorphism of order $t$. In Theorem 1.1 we have adopted the slightly stronger hypothesis that $G \le \mathrm{GL}_d(K)$, so that $t=1$.
\end{enumerate}

\end{remark}

The following result gives some consequences of Table \ref{noros}.

\begin{corollary}
	\label{cross_char_corol}
	Let $V= V_d(r)$ be as above. Suppose $G\leq \mathrm{GL}(V)$ and that the layer $E(G)$ of $G$ is quasisimple of Lie type in characteristic $p \nmid r$ acting absolutely irreducibly on $V$. Then the following statements hold.
	\begin{enumerate}
		\item $b(G)\leq 4$, except for $(G,V)=(U_4(2).2, V_6(2))$, where $b(G)=5$.
		\item If either $d>24$ or $|V|>31^8$ then $G$ has a regular orbit on $V$.
	\end{enumerate}
\end{corollary}

Note that part (i) of Corollary \ref{cross_char_corol} is an improvement to \cite{ourpaper}, where the authors show that $G$ as in Theorem \ref{mainthm} has $b(G)\leq 6$.

 The remainder of the paper is organised as follows. In Section \ref{prelims}, we give some preliminary results from the literature concerning properties of groups of Lie type, such as minimal degrees of irreducible representations, generation by conjugates and bounds on the number of certain prime order elements.
 In Section \ref{techniques}, we outline the techniques that we will use to prove Theorem \ref{mainthm}, and present some key propositions which provide the mechanism for the vast majority of the proof.
 We then proceed with the proof of Theorem \ref{mainthm} in Sections 4--8 by analysing, in order, the linear, unitary, symplectic, orthogonal and exceptional groups.
 
\section{Preliminaries}
\label{prelims}
From now on, we will let $G$ be an almost quasisimple group, and $V=V_d(r)$ be an irreducible $\mathbb{F}_rG$-module, that is also absolutely irreducible under the action of $E(G)$. By Remark \ref{gcd}, we will also assume that $(|G|,|V|)>1$. 
%There are several exceptional isomorphisms of small groups of Lie type where the isomorphic groups are over fields of different characteristics. In these cases, we will examine modules over all characteristics dividing the order of the group.
For $g \in G$, let $\emax^V(g)$ (or just $\emax(g)$) denote the dimension of the largest eigenspace of $g$ on $\overline{V} = V \otimes \overline{\mathbb{F}_r}$. For a prime $r_0$ dividing $|G|$, we will denote by $g_{r_0}$ an element of order $r_0$, and by $g_{r_0'}$ an element of prime order not $r_0$. For a set of primes $P$, we will denote $g_{P}$ an element whose order lies in $P$, and will use $g_{P'}$ to denote an element of prime order whose order does not lie in $P$. We say that $x \in G$ is of \textit{projective prime order} if $xF(G)$ has prime order in $G/F(G)$.
%\begin{lemma}[{\cite[Lemma 1.8i)]{MR3821142}}]
%\label{conjgen}
%Let $G$ be a finite group, and $C_1, \dots C_d$ be conjugacy classes of $G$ with representatives $c_1, \dots , c_d$. For $g\in G$, the number of solutions $(x_1, \dots , x_d) \in C_1 \times \dots \times C_d$ to $g = x_1x_2 \dots x_d$ is 
%\[
%\frac{\prod |C_i|}{|G|} \sum_{\chi \in \mathrm{Irr}(G)} \frac{\chi(c_1)\dots \chi(c_d)\chi(g^{-1})}{\chi(1)^{d-1}}
%\]
%where $\mathrm{Irr}(G)$ is the set of irreducible complex characters of $G$.
%\end{lemma}

\subsection{Minimal degrees of cross-characteristic representations}
The degrees of irreducible cross-characteristic representations of quasisimple groups of Lie type are not known in general. However, we are able to rely on some partial results. Hiss and Malle \cite{HM} give a complete list of irreducible representations of quasisimple groups of degree at most 250. This suffices for our examination of the exceptional groups and small classical groups. For large classical groups, we instead utilise results giving the minimal degree of a representation in cross-characteristic, and lower bounds on the degree of the next smallest representation. We outline some of these results below.

\begin{proposition}[{\cite[Tables I, II]{tiep2001low}}]
\label{mindegprop}
Suppose $H =H(q)$, $q=p^e$, is a quasisimple group of Lie type acting irreducibly on the $\mathbb{F}_rH$-module $V$ with $p\nmid r$. Then $\dim V \geq d_1(H)$, where $d_1(H)$ is given in Table \ref{mindeg}.
%\begin{table}[h!]
%\centering
%\begin{tabular}{|c|c|}
%\hline 
%$G \leq$ & $d_{r}^1(G)$ \\ 
%\hline 
%$\mathrm{SL}_n(q)$ & $\frac{q^n-q}{q-1}-1$ \\ 
%\hline 
%$\mathrm{SU}(n,q)$ & $\lfloor \frac{q^n-1}{q+1} \rfloor$ \\ 
%\hline 
%$\mathrm{Sp}_{2n}(q)$ $2 \nmid q$ & $\tfrac{1}{2}(q^n-1)$ \\ 
%$\mathrm{Sp}_{2n}(q)$ $2 \mid q$ & $\frac{(q^n-1)(q^n-q)}{2(q+1)}$ \\ 
%\hline 
%$\mathrm{Spin}_{2n}^+(q)$, $q > 3$ & $\frac{(q^n-1)(q^{n-1}+q)}{q^2-1}-2$ \\ 
%$\mathrm{Spin}_{2n}^+(q)$, $q \leq 3$ & $\frac{(q^n-1)(q^{n-1}-1)}{q^2-1}-2$ \\ 
%\hline 
%$\mathrm{Spin}_{2n}^-(q)$,  & $\frac{(q^n+1)(q^{n-1}-q)}{q^2-1}-1$ \\ 
%\hline 
%$\mathrm{Spin}_{2n+1}(q)$, $q > 3$ & $\frac{q^{2n}-1}{q^2-1}-2$ \\ 
%$\mathrm{Spin}_{2n+1}(q)$, $q = 3$ & $\frac{(q^n-1)(q^n-q)}{q^2-1}$ \\ 
%\hline 
%$^3D_4(q)$ & $q^5-q^3+q-1$ \\ 
%\hline 
%\end{tabular}
%\quad
%\begin{tabular}{|c|c|}
%\hline 
%$G \leq$ & $d_{r}^1(G)$ \\ 
%\hline 
%$^3D_4(q)$ & $q^5-q^3+q-1$ \\ 
%\hline 
%$E_6(q)$ & $(q^5+q)(q^6+q^3+1)-1$ \\ 
%\hline 
%$E_7(q)$ & $q^{17}-q^{15}$ \\ 
%\hline 
%$E_8(q)$ & $q^{29}-q^{27}$ \\ 
%\hline 
%$F_4(q)$, $q$ odd & $q^8+q^4-2$ \\ 
%$F_4(q)$, $q$ even & $\tfrac{1}{2}(q^3-1)(q^8-q^7)$ \\ 
%\hline 
%$G_2(q)$, $q \equiv 0 (3)$ & $q^4+q^2$ \\ 
%$G_2(q)$, $q \equiv 1 (3)$ & $q^3-1$ \\ 
%$G_2(q)$, $q \equiv 2 (3)$ & $q^3$ \\ 
%\hline 
%$^2B_2(q)$ & $(q-1) \sqrt{q/2}$ \\ 
%\hline 
%$^2G_2(q)$ & $q(q-1)$ \\ 
%\hline 
%$^2E_6(q)$ & $(q^5+q)(q^6-q^3+1)-1$ \\ 
%\hline 
%$^2F_4(q)$ & $(q^5-q^4)\sqrt{q/2}$ \\ 
%\hline 
%\end{tabular} 
%\caption{Lower bounds $d_{r}^1(G)$ on representation degrees in characteristic $r$.}
%\label{mindeg}
%\end{table}
\end{proposition}
\begingroup
\renewcommand*{\arraystretch}{1.5}
{\tiny
\begin{table}[h!]
\begin{tabular}{cccc}
\toprule
 &  & \multicolumn{2}{c}{Exceptions} \\
 \cmidrule{3-4}
$H/Z(H)$ & $d_1(H)$ & $(n,q)$ & $d_1(H) $ \\
\midrule
$L_2(q)$ & $(q-1)/(2,q-1)$ & $(2,4)$ & 2 \\
 &  & $(2,9)$ & 3 \\
 \midrule
$L_n(q)$, $n\geq 3$ & $\frac{q^n-q}{q-1}-1$ & $(3,2)$ & 2 \\
 &  & (3,4) & 4 \\
 &  & $(4,2)$ & 6 \\
 &  & (4,3) & 26 \\
 \midrule
$U_n(q)$, $n\geq 3$ & $\lfloor \frac{q^n-1}{q+1} \rfloor$ & $(4,2)$ & 4 \\
 &  & (4,3) & 6 \\
 \midrule
$\mathrm{PSp}_{2n}(q)$, $n\geq 2$ , $2 \nmid q$& $\tfrac{1}{2}(q^n-1)$ &  &  \\
$\mathrm{PSp}_{2n}(q)$, $n\geq 2$,  $2 \mid q$ & $\frac{(q^n-1)(q^n-q)}{2(q+1)}$ & $(2,2)$ & 4 \\
\midrule
$\mathrm{P}\Omega_{2n}^+(q)$, $n\geq 4$, $q > 3$ & $\frac{(q^n-1)(q^{n-1}+q)}{q^2-1}-2$ &  &  \\
$\mathrm{P}\Omega_{2n}^+(q)$, $n\geq 4$, $q \leq 3$ & $\frac{(q^n-1)(q^{n-1}-1)}{q^2-1}$ & $(4,2)$ & 8 \\
\midrule
$\mathrm{P}\Omega_{2n}^-(q)$, $n\geq 4$ & $\frac{(q^n+1)(q^{n-1}-q)}{q^2-1}-1$ &  &  \\
\midrule
$\mathrm{P}\Omega_{2n+1}(q)$,$n\geq 3$, $q > 3$ & $\frac{q^{2n}-1}{q^2-1}-2$ &  &  \\
$\mathrm{P}\Omega_{2n+1}(q)$, $n\geq 3$, $q = 3$ & $\frac{(q^n-1)(q^n-q)}{q^2-1}$ & $(3,3)$ & 27 \\
\bottomrule
\end{tabular}
\quad %
\begin{tabular}{cccc}
\toprule
 &  & \multicolumn{2}{c}{Exceptions} \\
 \cmidrule{3-4}
$H/Z(H)$ & $d_1(H)$ & $q$ & $d_1(H)$ \\
\midrule
$^3D_4(q)$ & $q^5-q^3+q-1$ &  &  \\
\midrule
$E_6(q)$ & $(q^5+q)(q^6+q^3+1)-1$ &  &  \\
\midrule
$E_7(q)$ & $q^{17}-q^{15}$ &  &  \\
\midrule
$E_8(q)$ & $q^{29}-q^{27}$ &  &  \\
\midrule
$F_4(q)$, $q$ odd & $q^8+q^4-2$ &  &  \\
$F_4(q)$, $q$ even & $\tfrac{1}{2}(q^3-1)(q^8-q^7)$ & 2 & 52 \\
\midrule
$G_2(q)$, $q \equiv 0 (3)$ & $q^4+q^2$ & 3 & 14 \\
$G_2(q)$, $q \equiv 1 (3)$ & $q^3$ & 4 & 12 \\
$G_2(q)'$, $q \equiv 2 (3)$ & $q^3-1$ &  &  \\
\midrule
$^2B_2(q)$ & $(q-1) \sqrt{q/2}$ & 8 & 8 \\
\midrule
$^2G_2(q)'$ & $q(q-1)$ &  &  \\
\midrule
$^2E_6(q)$ & $(q^5+q)(q^6-q^3+1)-2$ & 2 & 1938 \\
\midrule
$^2F_4(q)'$ & $(q^5-q^4)\sqrt{q/2}$ & 2 & 26\\
&&&\\
\bottomrule
\end{tabular}
\caption{Lower bounds $d_1(G)$ on representation degrees in cross-characteristic $r$.}
\label{mindeg}
\end{table}
}
\endgroup
For an integer $n$ and a prime $p$, we denote the $p'$-part of $n$ by $n_{p'}$. That is, if $p^k$ is the largest power of $p$ dividing $n$, then $n_{p'}=n/p^k$. For linear groups, the following result will also be useful.

\begin{proposition}[{\cite[Theorem 4.1]{tiep2001low}}]
 \label{rep2sl}
 Let $G=\mathrm{SL}_n(q)$ with $n\geq 5$, and let $r=r_0^f$ be a prime power coprime to $q$. Set $\epsilon_{n,q,r}=1$ if $r_0$ divides $(q^n-1)/(q-1)$ and $\epsilon_{n,q,r}=0$ otherwise. If $V$ is a non-trivial irreducible $\mathbb{F}_rG$-module of dimension less than
 \[
 \begin{cases}
 217 & (n,q)=(6,2)\\
 6292 & (n,q)=(6,3)\\
 (q^{n-1}-1)(\frac{q^{n-2}-q}{q-1}-\epsilon_{n-2,q,r}) & \textrm{otherwise},
 \end{cases}
 \]
 then $V$ either has dimension $(q^n-q)/(q-1)-\epsilon_{n,q,r}$ or is one of $(q-1)_{r_0'}$ modules with dimension $(q^n-1)/(q-1)$.
 \end{proposition}
\subsection{Properties of prime order elements in groups of Lie type}
\label{poe}
For $G_0$ a simple group, and non-identity $x \in {\rm Aut}(G_0)$, define $\a(x)$ to be the minimal number of $G_0$-conjugates of $x$ needed to generate the group $\langle G_0, x\rangle$. For an element $g$ of an almost quasisimple group $G$, define $\a(g)=\a(gF(G))$, and also define
\[
\a(G) = {\rm max}\,\{ \a(x)\,\mid\,1\ne x \in G/F(G) \} .
\]
We now give some results which compute bounds on $\alpha(G)$ for $G$ a group of Lie type.
\begin{proposition}[{\cite[Lemma 3.1]{gs}}]
\label{L2lemma} 
Let $H$ be an almost simple  group with $\mathrm{soc}(H) \cong L_2(q)$, and let $x\in H$ have prime order $r$. Then $\alpha(x) \leq 3$ unless:
\begin{enumerate}
\item $x$ is an involutory field automorphism and $\alpha(x) \leq 4$, except $\alpha(x)=5$ for $q=9$, or 
\item $q=5$, $x$ is an involutory diagonal automorphism, and $\alpha(x)=4$.
\end{enumerate}
Moreover, if the order of $x$ is odd, then $\alpha(x) = 2$, unless $q=9$, $r=3$ and $\alpha(x) = 3$.
\end{proposition}

\begin{proposition}[{\cite[Theorems 4.1, 5.1]{gs}}]\label{alphas} Let $G$ be an almost simple classical group over $\F_q$ with socle $\mathrm{soc}(G) = G_0$ and natural module of dimension $n\geq 3$, taking $n\geq 7$ for $G$ an orthogonal group. Then for $x\in G\setminus\{1\}$, one of the following holds:
\begin{enumerate}
\item$\a(x) \leq n$,
\item$G_0 = PSp_n(q)$, $q$ even,and $\a(x) =n+1$,
\item $G_0 = L_3(q)$, $x$ is an involutory graph-field automorphism and $\alpha(x)\leq 4$,
\item$G_0= L_4^\ep(q)$ with $q\geq 3$, and $x$ is an involutory graph automorphism and $\alpha(x)\leq 6$,
\item $G_0 = PSp_4(q)$, $q\neq 3$, $x$ is an involution and $\a(x) \leq 5$, or
\item $G_0 = U_3(3)$, $x$ is an inner involution with $\alpha(x) = 4$, 
\item  $G_0 = L_4^\ep(2)$.
\end{enumerate}
Furthermore, if $G_0$ is a simple exceptional group of Lie type of untwisted Lie rank $l$, then $\a(x) \leq l+3$, except $a(x) \leq 8$ for $G_0 = F_4(q)$. Moreover $\alpha(x)\leq 3$ for $G_0 =\, ^2G_2(q)$ or $^2B_2(q)$.
\end{proposition}

We also have these more precise bounds for symplectic and orthogonal groups.
\begin{proposition}[{\cite[Theorem 4.3]{gs}}]
\label{sympalphas}
Suppose $G_0=\mathrm{PSp}_{2m}(q)$, and $x \in \mathrm{Aut}(G_0)$. Then either:
\begin{enumerate}
\item $\alpha(x) \leq m+3$,
\item $x$ is a transvection and $\alpha(x) = 2m$ or $2m+1$ if $q$ is odd or even respectively, or 
\item $(m,q)=(2,3)$ and $x$ is an involution with $\alpha(x) = 6$.
\end{enumerate}
\end{proposition}

\begin{proposition}[{\cite[Theorem 4.4]{gs}}]
\label{gsorthog}
Suppose $G_0=\mathrm{P}\Omega^\ep_n(q)$ with $n=2m$ where $m\geq 4$,or $n=2m+1$ with $m\geq 3$ and $q$ odd. Then if $x \in \mathrm{Aut}(G_0)$, either 
\begin{enumerate}
\item $\alpha(x) \leq m+3$,
\item $q$ is odd, $x$ is a reflection and $\alpha(x) = n$, or 
\item $q$ is even and $x$ is a transvection with $\alpha(x) = n$.
\end{enumerate}
\end{proposition}

The remaining lemmas in this section count certain types of elements of prime order in groups of Lie type. Let $i_p(G)$ denote the number of elements of order $p$ in the group $G$. We adopt the definition of Frobenius endomorphism found in \cite[Definition 2.1.9]{GLS}.
\begin{proposition}[{\cite[Proposition 1.3]{MR1922740}}]
\label{invols}
Let $\bar{G}$ be a simple algebraic group over $\bar{\mathbb{F}}_p$ with associated root system $\Phi$. Let $\sigma$ be a Frobenius endomorphism of $\bar{G}$ such that $G_0=\bar{G}{^\sigma}'$ is a finite simple group of Lie type over $\mathbb{F}_q$.  Assume that $G_0$ is not of type $^2F_4$, $^2G_2$ or $^2B_2$. Then
\begin{enumerate}
\item $i_2(\mathrm{Aut}(G_0)) <2(q^{N_2} + q^{N_2-1})$, where $N_2 = \dim \bar{G} - \frac{1}{2}|\Phi|$, and 
\item $i_3(\mathrm{Aut}(G_0)) <2(q^{N_3} + q^{N_3-1})$, where $N_3 = \dim \bar{G} - \frac{1}{3}|\Phi|$.
\end{enumerate}
\end{proposition}
 
We define \textit{graph} and \textit{graph-field} automorphisms of groups of Lie type following \cite[Definition 2.5.13]{GLS}.

 \begin{lemma}
 	For $A=\mathrm{Aut}(L^\epsilon_n(q))$, with $n\geq 3$, let $\gamma(A)$ denote the number of involutory graph automorphisms in $A$. Then
 \label{graph}
% \begin{enumerate}
% \item  \label{lgraph} \[
% \gamma(\mathrm{Aut}(L_n(q))) \leq \begin{cases}
%q^{(n^2+n)/2} & \textrm{if } n \text{ even, } q \text{ odd}\\
%2 q^{(n^2+n)/2-1} &\textrm{if }  n \text{ even, } q \text{ even}\\
%q^{(n^2+n)/2-1} &\textrm{if }  n \text{ odd }\\
%\end{cases}
%\]
%and the number of graph--field automorphisms in $\mathrm{Aut}(L_n(q))$ is less than $2q^{(n^2-1)/2}$.
%%\item In the case where $H = \mathrm{Aut}(L_4(q))$ and $q$ is odd, $\gamma(H) =(q-1)q^2(q^2+q+1)(q^4+1)$.
% \item\label{ugraph} \[
%% \gamma(\mathrm{Aut}(U_n(q))) \leq \begin{cases}
%%2q^{\frac{n^2-n}{2}+1}(q^n+1) & n \text{ even, } q \text{ odd}\\
%%4q^{\frac{n(n+1)}{2}} & n \text{ even, } q \text{ even}\\
%%2q^{\frac{n^2+n}{2}-1} & n \text{ odd }\\
%%\end{cases}
% \gamma(\mathrm{Aut}(U_n(q))) \leq \begin{cases}
%2q^{\frac{n^2+n}{2}}) & n \text{ even, } q \text{ odd}\\
%4q^{\frac{n(n-1)}{2}} & n \text{ even, } q \text{ even}\\
%2q^{\frac{n^2+n}{2}-1} & n \text{ odd }\\
%\end{cases}
%\]
% 
% \end{enumerate}

 \begin{enumerate}
	\item  \label{lgraph}
	$\gamma(\mathrm{Aut}(L_n(q))) \leq 2 q^{(n^2+n)/2-1}$,

	%\item In the case where $H = \mathrm{Aut}(L_4(q))$ and $q$ is odd, $\gamma(H) =(q-1)q^2(q^2+q+1)(q^4+1)$.
	\item\label{ugraph} 
	% \gamma(\mathrm{Aut}(U_n(q))) \leq \begin{cases}
	%2q^{\frac{n^2-n}{2}+1}(q^n+1) & n \text{ even, } q \text{ odd}\\
	%4q^{\frac{n(n+1)}{2}} & n \text{ even, } q \text{ even}\\
	%2q^{\frac{n^2+n}{2}-1} & n \text{ odd }\\
	%\end{cases}
	$\gamma(\mathrm{Aut}(U_n(q))) \leq 4 q^{(n^2+n)/2-1}$, and 
	\item for $q$ a square, the number of graph--field automorphisms in $\mathrm{Aut}(L_n(q))$ is less than $2q^{(n^2-1)/2}$.

\end{enumerate}
 \end{lemma}
 \begin{proof}
First let $H=\mathrm{PGL}_n(q)$. We adopt the notation of \cite[\S 3.2.5]{bg}. When $n$ is even and $q$ is odd there are three distinct $H$-conjugacy classes of graph automorphisms  \cite[Table B.3]{bg} satisfying 
  \begin{align*}
|\gamma_1^H| &= \frac{|\mathrm{PGL}_n(q)|}{|\mathrm{PGSp}_n(q)|} = q^{\frac{n^2}{4} - \frac{n}{2}} \prod _{i=1}^{n/2-1}(q^{2i+1}-1)< q^{\frac{n^2-n}{2}-1}, \\
|\gamma_2^H| &= \frac{|\mathrm{PGL}_n(q)|}{|\mathrm{PGO}^+_n(q)|} =\frac{1}{2} q^{\frac{n^2}{4}} (q^{n/2}+1) \prod_{i=2}^{n/2} (q^{2i-1}-1)< \frac{1}{2}q^{\frac{n^2}{2} -1} (q^{n/2}+1),\\
|(\gamma_2')^H| &= \frac{|\mathrm{PGL}_n(q)|}{|\mathrm{PGO}^-_n(q)|} =\frac{1}{2} q^{\frac{n^2}{4}} (q^{n/2}-1) \prod_{i=2}^{n/2} (q^{2i-1}-1)< \frac{1}{2}q^{\frac{n^2}{2} -1} (q^{n/2}-1),\\
\end{align*}

where $\mathrm{PGSp}_n(q)$ and $\mathrm{PGO}^\pm_n(q)$ are the conformal symplectic and orthogonal groups respectively.
We have $|\gamma_1^H| + |\gamma_2^H| +  |(\gamma_2')^H| < q^{\frac{n^2-n}{2}-1}+ \frac{1}{2}q^{\frac{n^2}{2} -1} (q^{n/2}+1)+ \frac{1}{2}q^{\frac{n^2}{2} -1} (q^{n/2}-1)< 2q^{\frac{n^2+n}{2}-1}$ as required. If instead $n$ is even and $q$ is even, then there are two $\mathrm{PGL}_n(q)$-conjugacy classes of graph automorphisms, namely $\gamma_1^H$ (as above) and $\gamma_3^H$. The total number of graph automorphisms in this case is
\[
|\gamma_1^H| + |\gamma_3^H|<  q^{\frac{n^2-n}{2}-1} + \frac{|\mathrm{PGL}_n(q)|}{q^{n^2/4} \prod _{i=1}^{n/2-1} (q^{2i}-1)} =q^{\frac{n^2-n}{2}-1} +  q^{\frac{n^2}{4}-\frac{n}{2}}(q^n-1)\prod_{i=2}^{n/2} (q^{2i-1}-1)< 2 q^{\frac{n^2+n}{2}-1} .
\]
If $n$ is odd, then there is only one $\mathrm{PGL}_n(q)$-conjugacy class of graph automorphisms, satisfying
\[
|\gamma_1^H| \leq \frac{|\mathrm{PGL}_n(q)|}{|\mathrm{Sp}_{n-1}(q)|} = q^{\frac{n^2-1}{4}} \prod _{i=2}^{(n+1)/2}(q^{2i-1}-1)< q^{\frac{n^2+n}{2}-1}. \\
\]
Now let $H = \mathrm{PGU}_n(q)$. By \cite[Table B.4]{bg}, for a given $n$ and $q$, the number and centraliser orders of involutory graph automorphisms lying in $\mathrm{Aut}(U_n(q))$ are the same as those in $\pgl_n(q)$. Since $\frac{1}{2} q^{n^2-1}< |\pgl^\epsilon_n(q)| < q^{n^2-1}$ (see \cite[Proposition 3.9(ii)]{tim2} for example), we have $|\mathrm{PGU}_n(q)|/|\pgl_n(q)| < 2$, and the result for part \ref{ugraph} follows.
Finally, if $q$ is a square, there is one $\mathrm{PGL}_n(q)$-conjugacy class of graph--field automorphisms and the size of this class is
\[
\frac{|\mathrm{PGL}_n(q)|}{|\mathrm{PGU}_n(q^{1/2})|}< 2q^{\frac{n^2-1}{2}}.
\]
 \end{proof}
% The following lemma is proved similarly to Lemma \ref{graph}.
% \begin{lemma}
%\[
% \gamma(\mathrm{Aut}(U_n(q))) \leq \begin{cases}
%2q^{\frac{n^2-n}{2}+1}(q^n+1) & n \text{ even, } q \text{ odd}\\
%4q^{\frac{n(n+1)}{2}} & n \text{ even, } q \text{ even}\\
%2q^{\frac{n^2+n}{2}-1} & n \text{ odd }\\
%\end{cases}
%\]
% \label{ugraph}
%\end{lemma}

%Finally, we have the following more specific lemma for $\mathrm{Aut}(L_4(q))$.
%\begin{lemma}
%\label{graphsl}
%The number of graph automorphisms of prime order in 
%\end{lemma}
%\begin{proof}
%By \cite[Table B.3.]{bg} (and using the notation employed there), we must consider graph automorphisms of type $\gamma_1$, $\gamma_2$ and $\gamma_2'$. Let $H=\mathrm{PGL}_n(q)$. We have 
%\begin{align*}
%|\gamma_1^H|+|\gamma_2^H|+|(\gamma_2')^H| &= \frac{|H|}{\mathrm{PGSp}(4,q)} + \frac{|H|}{\mathrm{PGO^+}(4,q)}+\frac{|H|}{\mathrm{PGO^-}(4,q)}\\
%&= q^2(q^3-1)+\frac{1}{2} (q-1)q^4(q^2+1)(q^2+q+1)+\frac{1}{2} (q-1)^2q^4(q+1)(q^2+q+1)\\
%& = (q-1)q^2(q^2+q+1)(q^4+1).
%\end{align*}
%Finally, we multiply by 2 to account for scalars.
%\end{proof}

\subsection{Weil representations of symplectic and unitary groups}
\label{weilbackground}
We give some background on Weil representations of $\mathrm{Sp}_{2m}(q)$, $q$ odd, and $\mathrm{SU}_m(q)$, beginning with the former. Suppose $q=p^e$ is odd, and let $r_0 \nmid q$ be a prime. Let $S=q^{1+2m}$ be a special group with exponent $p$ and $Z(S) = S'$ of order $q$. For every non-trivial irreducible character $\chi$ of $Z(S)$, there is a unique irreducible $q^m$-dimensional representation of $S$ over $\overline{\mathbb{F}}_{r_{0}}$ where $Z(S)$ acts via $\chi$ \cite[\S 5]{MR1947325}. Now, $\mathrm{Sp}_{2m}(q)$ acts as a group of automorphisms of $S$, so preserves the character of the $q^m$-dimensional representation. Let $H = S:\mathrm{Sp}_{2m}(q)$. We may construct an irreducible $\overline{\mathbb{F}}_{r_{0}}H$-module $V_0$ which restricts to $S$ as the $q^m$-dimensional representation.

The restriction of $V_0$ to $\mathrm{Sp}_{2m}(q)$ is called a \textit{Weil module} for $\mathrm{Sp}_{2m}(q)$, with corresponding \textit{Weil character}. The non-trivial irreducible constituents of this Weil representation correspond to the minimal degree representations of $\mathrm{Sp}_{2m}(q)$, $q$ odd,  in characteristic $r_0$, unless $(m,q)=(2,2)$ (cf. Table \ref{mindeg}). In our work, we will additionally use the terms \textit{Weil representation} and \textit{Weil module} to refer to one of these (absolutely) irreducible constituents. If $r_0 \neq 2$, there are four inequivalent such Weil modules of $\mathrm{Sp}_{2m}(q)$ in characteristic $r_0$ -- two of dimension $(q^m-1)/2$, and two of dimension $(q^m+1)/2$. If instead $r_0=2$, then there are two Weil modules up to equivalence, both of dimension $(q^m-1)/2$ \cite{tiep2001low}.

The Weil modules for a unitary group $\mathrm{SU}_m(q) \leq \mathrm{Sp}_{2m}(q)$ for $q$ odd are the constituents of the restriction of $V_0$ to $\mathrm{SU}_m(q)$. 
 These Weil modules for $\mathrm{SU}_m(q)$ are of dimension $(q^m-(-1)^m)/(q+1)$ or $(q^m+q(-1)^m)/(q+1)$, with the number of modules of each dimension dependent on $r_0$, $m$ and $q$ \cite{MR1813499}. If $m\geq 4$ and $(m,q) \neq (4,3)$, then these irreducible Weil representations are the minimal degree representations of $\mathrm{SU}_m(q)$, $q$ odd,  in characteristic $r_0$.  
 
For unitary groups, we may additionally define Weil modules for $q$ even \cite[\S 4]{MR1449955}.
In this case, $r_0$ is odd, and there is an irreducible embedding of the symplectic type group $4\circ 2^{1+2m}$ into $\mathrm{GL}_{2^m}(\overline{\mathbb{F}}_{r_0})$, with normaliser $\overline{\mathbb{F}}^\times_{r_0}\circ 2^{1+2m}.\mathrm{Sp}_{2m}(q)$. This is a non-split extension. However, $4\circ 2^{1+2m}.\mathrm{SU}_{m}(q)$ is a split extension, and the constituents of the restriction of $V_0 = V_{2^m}(\overline{\mathbb{F}}_{r_0})$ to $\mathrm{SU}_{m}(q)$  give a set of \textit{Weil modules} for $\mathrm{SU}_m(q)$. Similarly to $q$ odd, if $m\geq 4$ and $(m,q) \neq (4,2)$, then these irreducible Weil representations are the minimal degree representations of $\mathrm{SU}_m(q)$ in characteristic $r_0$.

As in the case of the symplectic groups, we will also sometimes use the term \textit{Weil module} to refer to the $q^m$-dimensional restriction of $V_0$ to $\mathrm{SU}_m(q)$.

We can often infer the values of the Brauer character of a $q^m$ dimensional Weil representation over $\overline{\mathbb{F}}_r$ from those of the ordinary (complex) character. This allows us to place upper bounds on the eigenspace dimensions of group elements on the irreducible constituents. The following results give us the means to compute these ordinary character values.

\begin{proposition}[{\cite[Theorem 4.8]{MR332945}}]
\label{speqn}
Suppose $q$ is an odd prime power. Then $\mathrm{Sp}_{2m}(q)$ has irreducible ordinary Weil characters $\chi_1$, $\chi_2$ of degrees $(q^m \pm 1)/2$. Moreover, $|\chi_1(g)+\chi_2(g)|^2 = |C_W(g)|$ for all $g\in \mathrm{Sp}_{2m}(q)$, where $W$ is the natural module for $\mathrm{Sp}_{2m}(q)$ over $\mathbb{F}_q$.
\end{proposition}

\begin{proposition}[{\cite[Corollary 4.9.2]{GERARDIN197754}}]
\label{weilunitary}
For $g\in \mathrm{SU}_n(q)$, the ordinary Weil character of degree $q^n$ of $\mathrm{SU}_n(q)$ has value $(-1)^n(-q)^{\dim C_W(g)}$, where $W$ is the natural module for $\mathrm{SU}_n(q)$.
\end{proposition}

%If $g$ is of odd prime order not equal to $r_0$, we can use \cite[\S 3.3.1, 3.4.1]{bg} to limit the dimension of $ C_W(g)$.

\begin{proposition}[{\cite[Propositions 3.3.2, 3.4.3]{bg}}]
	\label{itoeig}
	Suppose $g$ is a semisimple element of odd prime order $r_0$ in $\mathrm{Sp}_{n}(q)$ or $\mathrm{SU}_{n}(q)$, and let $i=\delta(r_0,q)$. 
	%Let $t=1/2$ if $g\in \mathrm{SU}_n(q)$, $i \equiv 2 \, (4)$ and $i>2$, let $t=1$ if $i=2$ or $i \equiv 0 \, (4)$, and let $t=2$ if $i$ is odd.  
	Let 
	\[
	t = \begin{cases} 
	1/2 & g\in \mathrm{SU}_n(q),  \, i \equiv 2 \, (4) \textrm{ and } i>2,\\
	1 & i \textrm{ even otherwise},\\
	2 & $i$ \textrm{ odd}.\\
	\end{cases}
	\]
Then unless $g \in \mathrm{SU}_{n}(q)$ and $i=2$, we have
	\[
	\dim C_W(g) \in \left\{n-tij \mid 1 \leq j \leq \left\lfloor\frac{n}{ti}\right\rfloor \right\}.
	\]
\end{proposition}
%\begin{proof}
%Follows immediately from \cite[Propositions 3.3.2, 3.4.3]{bg}.
%\end{proof}
%\begin{proposition}
%Suppose that $\chi$ is an irreducible Weil character of $H=\mathrm{Sp}(4,q)$, $q$ odd. Then for $g \in H\setminus Z(H)$, $\chi(g) \in \{0,\pm 1, \pm q, \pm \tfrac{1}{2} (1\pm q), $

%\end{proposition}
 For large symplectic and  unitary groups it will also be useful to bound the degree of the smallest representation that is not a Weil representation.
 \begin{proposition}[{\cite{MR1813499}}]
 \label{d2unitary}
 Suppose $G=\mathrm{SU}_n(q)$ with $n\geq 4$ and $(n,q)\neq (4,2), (4,3)$. Let $\mathbb{F}$ be an algebraically closed field of characteristic $r\neq p$. Suppose that $V$ is a non-trivial absolutely irreducible $\mathbb{F}G$-module of degree less than 
 \[
 \begin{cases}
 (q^2+1)(q^2-q+1)/(2,q-1)-1 & \textrm{if } n=4,\\
(q^{n-2}-1)(q-1) \lfloor(q^{n-2}-1)/(q+1)\rfloor & \textrm{if } n\geq 5\\
 \end{cases}
 \]
 Then $V$ is a Weil module of degree $(q^n-(-1)^n)/(q+1)$ or $(q^n+q(-1)^n)/(q+1)$.
 \end{proposition}

 \begin{proposition}[{\cite{MR1947325} }]
 \label{d2symp}
 Suppose $G=\mathrm{Sp}_{2m}(q)$ with $n\geq 2$, $q$ odd and $(m,q) \neq (2,3)$. Let $\mathbb{F}$ be an algebraically closed field of characteristic $r\neq p$.  If $V$ is a non-trivial absolutely irreducible $\mathbb{F}G$-module of dimension less than $(q^m-1)(q^m-q)/2(q+1)$, then $V$ is a Weil module of dimension $(q^m \pm 1)/2$.
 \end{proposition}

 We also have the following useful result about the restriction of Weil modules for unitary groups to certain geometric subgroups.
 
 \begin{proposition}[{\cite[Corollary 3.4]{GERARDIN197754}}]
 \label{su_weil_tensor}
The $q^m$-dimensional Weil module for $G=\mathrm{SU}_m(q)$ restricts to the subgroup $\mathrm{SU}_k(q) \times \mathrm{SU}_{m-k}(q)$ as the tensor product of a $q^k$-dimensional Weil module for $\mathrm{SU}_k(q)$, and a $q^{m-k}$-dimensional Weil module for $\mathrm{SU}_{m-k}(q)$.
 \end{proposition}
\section{Techniques}
\label{techniques}
Let $G$ be an almost quasisimple group with $E(G)$ a quasisimple group of Lie type acting absolutely irreducibly on $V_d(r)$ in cross characteristic. Recall that we also assume $(r, |G|)>1$, since the co-prime case was completed by \cite{MR1829482}.
Our method of proof for Theorem \ref{mainthm} relies heavily on the observation that if $G$ has no regular orbit on $V$, then $V$ is the union of the fixed points spaces $C_V(g)$ for $g\in G\setminus\{ 1\}$. Since conjugates in $G$ have fixed point spaces of the same dimension, we have 
\begin{equation}
\label{ogeqn}
|V| \leq \sum_{x\in \mathcal{X}} |x^G||C_V(x)|, 
\end{equation}
where $\mathcal{X}$ is a set of non-identity conjugacy class representatives of $G$.

We now present the main tools used to prove Theorem \ref{mainthm}. For $G$ an almost quasisimple group, let $G_s$ denote the set of elements of $G$ of projective prime order $s$, and $G_{s'}$ denote the set of elements of $G$ with projective prime order coprime to $s$. 

%Note that the following proposition adopts the assumption that $G\leq \Gamma \mathrm{L}(V)$, which is weaker than the assumptions of Theorem \ref{mainthm}.  This is so that we are able to make use of this proposition in other papers in this series where we allow semilinear transformations.
\begin{proposition}
	\label{tools}
	Let $G\leq  \mathrm{GL}(V)$ be an almost quasisimple group, acting on the $d$-dimensional module $V=V_d(r)$ over $\mathbb{F}_r$, where $r = r_0^e$ for $r_0$ prime. Suppose also that the restriction of $V$ to $E(G)$ is absolutely irreducible. Set $H=G/F(G)$ and let $\mathcal{P}$ be a complete set of conjugacy class representatives of elements of projective prime order in $G$. For $x \in G$, let $\bar{x} = xF(G) \in H$, and denote the order of $\bar{x}$ by $o(\bar{x})$. Then each of the following inequalities hold.
	\begin{enumerate}
		\item \label{alphabound} For $\nu \in \overline{\mathbb{F}}_r$, the $\nu$-eigenspace $E_\nu(x)$ of $x \in G$ on $\overline{V}$  satisfies 
		\[ 
		\dim_{\overline{\mathbb{F}}_r}(E_\nu(x)) \leq \left\lfloor\dim_{\overline{\mathbb{F}}_r}(\overline{V}) \left( 1-\frac{1}{\a(x)}\right)\right\rfloor.
		\] 
	\end{enumerate}
	Further, if $G$ has no regular orbit on $V$, then:
	\begin{enumerate} \setcounter{enumi}{1}
		\item \label{eigsp1} \[
		|V| \leq \sum_{x\in \mathcal{P}} \sum_{\kappa \in \mathbb{F}_r^\times} \frac{1}{o(\bar{x})-1} |\bar{x}^H| |C_V(\kappa x)|.
		\]
		\item \label{eigsp2}
		\[
		|V| \leq \sum_{x\in \mathcal{P} \cap G_{r_0'}} \frac{o(\bar{x} )}{o(\bar{x} )-1} |\bar{x} ^H|\mathrm{max}\{|C_V(\kappa x)| \mid \kappa \in \mathbb{F}_r^{\times}\}+ \sum_{x\in \mathcal{P} \cap G_{r_0}} \frac{1}{o(\bar{x} )-1}  |\bar{x}^H| |C_V(x)|.
		\]
		\item \label{qsgood} \[
		|V| \leq 2\sum_{x\in \mathcal{P} \cap G_{r_0'}} |\bar{x} ^H|\mathrm{max}\{|C_V(\kappa x)| \mid \kappa \in \mathbb{F}_r^{\times}\} + \sum_{x\in \mathcal{P} \cap G_{r_0}} \frac{1}{o(\bar{x})-1}  |\bar{x}^H| |C_V(x)|.
		\]
		\item \label{crude} \begin{equation}
			\label{crudeeqn}
			|V| = r^d \leq 2\sum_{x\in \mathcal{P} \cap G_{r_0'}} |\bar{x}^H| r^{\lfloor(1-1/\alpha(x))d\rfloor} + \sum_{x\in \mathcal{P} \cap G_{r_0}} \frac{1}{o(\bar{x})-1}  |\bar{x} ^H|r^{\lfloor(1-1/\alpha(x))d\rfloor},
		\end{equation}
		
		and for fixed $r$, if this inequality fails for a given $d$, then it fails for all $d_1 \geq d$.
	\end{enumerate}
\end{proposition}
\begin{proof}
	Part \ref{alphabound} follows from \cite[p. 454]{MR1188385}. We will prove the other parts of the proposition simultaneously.
	Let $x \in G$ be of projective prime order. Then there exists a prime $a$ such that $x^a = z \in F(G)$. Now, we may choose $\hat{z} \in \overline{\mathbb{F}}_r^\times$ so that $\hat{z}^a=z^{-1}$, and so $g = \hat{z}x \in \overline{\mathbb{F}}_r^\times.G$ has prime order $a$. Then $g$ has at most $a$ eigenspaces on $\overline{V}$, and therefore $x$ has at most $a$ eigenspaces on $V$. Recall that $H=G/F(G)$, and let $\mathcal{H}$ be a set of coset representatives of elements of prime order in $H$.
	Suppose that $G$ has no regular orbit on $V$. Then
	\begin{equation}
		V = \bigcup_{x \in \mathcal{H}} \bigcup_{\kappa \in F(G)} C_V(\kappa x).
		\label{unioneqn}
	\end{equation}
From \eqref{unioneqn}, noting that $C_V(\kappa x)$ is precisely the $\kappa^{-1}$ eigenspace of $x$, and that $C_V(g) \subseteq C_V(g^i)$ for all $g \in G$ and every positive integer $i$, we have
	\[
	|V|  \leq \sum_{x\in \mathcal{H}} \sum_{\kappa \in \mathbb{F}_r^{\times}} \frac{1}{o(\bar{x})-1} |C_V(\kappa x)|= \sum_{x\in \mathcal{P}} \sum_{\kappa \in \mathbb{F}_r^\times} \frac{1}{o(\bar{x} )-1} |\bar{x} ^H| |C_V(\kappa x)|.
	\]
	
	Therefore,
	\[|V| \leq \sum_{x\in \mathcal{P} \cap G_{r_0'}} \frac{o(\bar{x} )}{o(\bar{x})-1} |\bar{x} ^H|\mathrm{max}\{|C_V(\kappa x)| \mid \kappa \in \mathbb{F}_r^{\times}\}+ \sum_{x\in \mathcal{P} \cap G_{r_0}} \frac{1}{o(\bar{x})-1}  |\bar{x} ^H| |C_V(x)|,
\]
	since unipotent elements in $G$ have only one eigenspace.
	We have now established \ref{eigsp1} and \ref{eigsp2}. Further, \ref{qsgood} follows from \ref{eigsp2} since $o(\bar{x}) \geq 2$, and \ref{crude} follows from \ref{qsgood} and \ref{alphabound}. 
	
	Finally, we prove that if we fix $r$ and \ref{crude} fails for a given $d$, then it fails for every $d_1\geq d$. Write $d_1=d+k$. Comparing \ref{crude} for $d$ and $d_1$, we see that the proof reduces to showing that for $x\in \mathcal{P}$ we have $\lfloor (1-1/\alpha(x))d \rfloor +k\geq    \lfloor (1-1/\alpha(x))(d+k) \rfloor$. We observe that $\lfloor (1-1/\alpha(x))d \rfloor +k = \lfloor (1-1/\alpha(x))d+k\rfloor$ and the result follows. 	
\end{proof}

The next two results present additional methods of bounding $\dim C_V(x)$ for $x$ of projective prime order in $G$. 

\begin{proposition}
\label{compfactors}
Suppose that $G$ is a finite group and $V$ is an $\mathbb{F}_rG$-module. Let $V_1, \dots V_k$ be the composition factors of $V$ under the action of $G$. Then for an element $x\in G$,
\[
\dim C_V(x) \leq \sum_{i=1}^k \dim C_{V_i}(x).
\]
\end{proposition}
\begin{proof}
The proof follows from an elementary induction argument on the number of composition factors of $V$.
%We will proceed by induction on the number $m$ of composition factors of $V$. If $m=1$, then $V$ is irreducible and the result is clear. Assume the result holds for $m=k-1$, and now suppose $m=k$. Let $V_1 \subset V$ be a maximal submodule of $V$. We will show that for $x \in G$, $\dim C_V(x) \leq \dim C_{V/V_1}(x)+\dim C_{V_1}(x)$. Let $\{v_1, \dots v_j\}$ be a basis of $C_{V_1}(x) = C_V(x) \cap V_1$, and extend to a basis $B=\{v_1, \dots v_s\}$ of $C_V(x)$. Then $\{v_i + V_1 \mid i>j\}$ is a  subset of $C_{V/V_1}(x)$ and is linearly independent in $V/V_1$ otherwise $B$ is linearly dependant. So $\dim C_{V/V_1}(x)+\dim C_{V_1}(x) \geq (s-j)+j = \dim C_V(x) $ as required and the result follows by induction.
\end{proof}
When considering the restriction of a module $V$ for $G$ to a subgroup $H\leq G$, we will often denote the composition factors of $V$ by their dimensions and separate composition factors by the symbol $/$. We also sometimes use the notation $d^k$ to denote $k$ copies of a composition factor of dimension $d$. See the proof of Proposition \ref{prop_u102}, for example.

For $g \in G$, let $\emax^V(g)$ (or just $\emax(g)$) denote the dimension of the largest eigenspace of $g$ on $\overline{V} = V \otimes \overline{\mathbb{F}_r}$.

\begin{proposition}[{\cite[Lemma 3.7]{MR1639620}}]
\label{tensorcodim}
Let $V_1$ and $V_2$ be vector spaces over $\mathbb{F}_r$ of dimension $d_1$ and $d_2$ respectively. Then if $g = g_1\otimes g_2 \in \mathrm{GL}(V_1) \otimes \mathrm{GL}(V_2)$ is an element of projective prime order that acts on $V=V_1\otimes V_2$, then 
\[
\emax^V(g) \leq \mathrm{min} \left\{ d_2\emax^{\overline{V_1}}(g_1), d_1\emax^{\overline{V_2}}(g_2) \right\}.
\]
where $\overline{V_i} = V_i \otimes \overline{\mathbb{F}_r}$.
\end{proposition}
\begin{proof}
This follows from {\cite[Lemma 3.7]{MR1639620}} with the observation that any element of projective prime order is of the form $\lambda x$ for $\lambda \in  \overline{\mathbb{F}_r}$ and $x$ an element of prime order in $\mathrm{GL}(V_1) \otimes \mathrm{GL}(V_2)$.
\end{proof}

We now give a brief discussion of Brauer characters. In order to define a Brauer character in characteristic $r_0>0$, we first require a mapping $\varphi$ from the multiplicative group $\overline{\F}_{r_0}$ to the set of roots of unity with order coprime to $r_0$ in $\mathbb{C}$.  Given an irreducible representation $\psi$ for a group $G$ over $\overline{\F}_{r_0}$, we can then define the Brauer character $\rho_\psi$ corresponding to $\psi$ to be the class function mapping each conjugacy class of $r_0'$-order elements of $G$ to the sum of the images of their eigenvalues under $\varphi$.  This gives us a way of recovering the eigenvalues of an element $\psi(g)\in \psi(G)$ of $r_0'$-order.  For example, we have the following proposition.
 \begin{proposition}
	\label{intchar}
	Suppose $G$ is a group acting irreducibly on the $\overline{\mathbb{F}_r}G$-module $V=V_d(\overline{\mathbb{F}}_r)$ with corresponding Brauer character $\rho$. Suppose $x\in G$ is of prime order $p\nmid r$. Let $E_x$ denote the set of eigenvalues of $x$ (including multiplicities),  $\overline{\Omega}$ denote the set of primitive $p$th roots of unity in $\overline{\mathbb{F}}_r$, and $\Omega  = \overline{\Omega} \cup \{1\}$. Let $c\in \mathbb{N}$. Then:
	\begin{enumerate}
		\item If $\rho(x) = c$, then $E_x= (1^c, \Omega^{(d-c)/p})$.
		\item If $\rho(x) = -c$, then $E_x=( \overline{\Omega}^c,\Omega^{(d-cp+c) /p})=(\overline{\Omega}^{(d+c)/p},1^{(d-cp+c) /p})$.
		%\item If  $|\rho(y)|=|\rho(z)|$ for $y,z$ elements of the same projective prime order in $G$, then $E_y = \lambda E_z$ for some $\lambda \in \overline{\mathbb{F}}_r$. In particular, $y$ and $z$ have the same set of eigenspace dimensions.
	\end{enumerate}
\end{proposition}
\begin{proof}
	Notice that the sets $E_x$ in parts i) and ii) sum to $\rho(x)$ in each case, so it remains to show that $E_x$ is uniquely determined.
	The primitive $p$th roots of unity are linearly independent over $\mathbb{Z}$, which implies uniqueness. 
	%The argument for iii) is similar.
\end{proof}

It is possible to restrict an irreducible representation $\psi$ for $G$ to a finite subfield $\F_{r_0^k} \subset\overline{\mathbb{F}_{r_0}}$ as long as $\F_{r_0^k}$ contains all of the preimages of the Brauer character values corresponding to $\psi$ under $\varphi$.  In particular, non-integral Brauer character values (which we term \textit{irrationalities}), may determine the subfields of $\overline{\mathbb{F}_{r_0}}$ over which the representation is realisable.  We adopt the notation for irrationalities found in \cite[\S 10]{ATLAS}, and in a slight abuse of terminology, we sometimes say that a finite field does or does not contain a certain irrationality, while we really mean that the finite field contains the preimage of the irrationality under $\varphi$.
%
%
%Let $G$ be a finite group and  $V = V_1\otimes V_2$ a faithful $\overline{\mathbb{F}_r}G$-module. Then for $g\in G$,  if $\emax(g) \leq d_1$ on $V_1$, then $\emax(g) \leq d_1\dim V_2$ on $V$.
%\end{proposition} 
%\begin{proof}
%Suppose $\theta$ is an eigenvalue of $g \in G$ on $V$ with largest eigenspace. Given an eigenvalue $\nu$, denote by $E_{\nu}^i(g)$ the $\nu$-eigenspace of $g$ on $V_i$. Then  $E_\theta(g) = \sum_{\nu \in \mathbb{F}_r^\times} \dim E_\nu^1(g) \cdot\dim E_{\nu^{-1}\theta}^2(g)\leq  \sum_{\nu\in \mathbb{F}_r^\times} d_1\cdot \dim E_{\nu^{-1}\theta}^2(g) = d_1 \dim V_2$. 
%\end{proof}

%We present one further result which allows us to give a crude upper bound on the dimension of an eigenspace of $x \in G$.
%\begin{lemma}[{\cite[Lemma 3.2i)]{MR2173972}}]
%\label{alphabound}
%Let $G$ be an almost quasisimple group, acting on a faithful irreducible $\mathbb{F}G$-module $V$ for some field $\mathbb{F}$. Then the $\nu$-eigenspace $E_\nu(x)$ of $x \in G$ satisfies 
%\[ 
%\dim_\mathbb{F}(E_\nu(x)) \leq \dim_\mathbb{F}(V) \left( 1-\frac{1}{\a(x)}\right)
%\] 
%\end{lemma}
Given a family of almost quasisimple groups of Lie type $G$, our proof of Theorem \ref{mainthm} proceeds as follows. 
We first determine an upper bound on $\alpha(G)$ and a lower bound on $d_1(G)$ for groups $G$ in the family, and note that from Proposition \ref{tools}\ref{crude},  if $G$ has no regular orbit on $V=V_d(r)$, then 
\[
r^d- 2\sum_{x\in \mathcal{P}} |x^H| r^{\lfloor(1-1/\alpha(x))d\rfloor}\leq 0.
\]
The left hand side is increasing with respect to $\dim V$, so by substituting in $d_1(G)$ and $\alpha(G)$, we may easily prove that most groups in the family satisfy Theorem \ref{mainthm}, leaving us with a finite collection of groups $G$ and irreducible modules $V = V_d(r)$ to consider.
We consider these cases individually. We start by checking whether $|G|>|V|$, since this implies the non-existence of a regular orbit. 
If this is the case, we can usually determine the base size either computationally in GAP \cite{GAP4} or Magma \cite{Magma}, or by using the following lemma.

\begin{lemma}
\label{fieldext}
Suppose $V$ is a $d$-dimensional vector space over $\mathbb{F}_r$, and let $G\leq \mathrm{GL}(V)$. If $G$ has a base of size $c$ on $V \otimes \mathbb{F}_{r^i}$, then $G$ has a base of size at most $ci$ on $V$.
\end{lemma}
\begin{proof}
Let $\{v_1, \dots v_c\}$ be a base for $G$ acting on $V \otimes \mathbb{F}_{r^i}$. Write each $v_j$ as $v_j = \sum_{k=1}^i \eta_k w_{j,k}$, where $w_{j,k} \in V$ and $\eta_1, \dots \eta_i$ is a basis for $\mathbb{F}_{r^i}$ over $\mathbb{F}_r$. Then $\{w_{j,k}\mid 1\leq j \leq c, 1\leq k\leq i\}$ is a base for $G$, since any non-trivial $g\in G$ stabilising the set pointwise must also stabilise $\{v_1, \dots v_c\}$.
\end{proof}
%If this is the case, we attempt to show that $G$ has a base of size $c= \left\lceil \log|G|/\log|V| \right\rceil$ by proving an equivalent assertion, namely that $G$ has a regular orbit on $V^c$. Since $|C_{V^c}(x)|$ = $|C_V(x)|^c$, we see that applying Proposition \ref{tools}\ref{eigsp1}--\ref{crude} to $G$ acting on $V^c$ gives the same calculation as applying the result to $G$ acting on $V_d(r^c)$.

If instead $|G|<|V|$, we aim to apply Proposition \ref{tools} to prove the existence of a regular orbit. In order to apply Proposition \ref{tools}, we usually need to compute the sizes of conjugacy classes of elements of prime order in $G/F(G)$, and the eigenspace dimensions of their preimages in $G$ on $V$. We determine conjugacy class orders using GAP \cite{GAP4} or \cite{bg}.
Eigenspace dimensions of projective prime order elements are determined (or at least bounded from above) using one of a number of methods. We may use Proposition \ref{tools}\ref{alphabound} along with results from Section \ref{poe}, or refine upper bounds of $\alpha(x)$ for $x\in G$ in GAP. If the Brauer character of $V$ is available in \cite{modatlas}, we may use it to determine eigenspace dimensions for semisimple elements. We may also seek to restrict $V$ to a subgroup of $G$ and apply Proposition \ref{compfactors}.
If we are not able to use these methods to successfully complete the proof with an application of Proposition \ref{tools}, we resort to computational methods. We use the generators of some irreducible modules in the online ATLAS \cite{onlineATLAS} and construct them in GAP, or compute generators explictly in Magma \cite{Magma}. Once we have constructed $V$, we then either determine eigenspace dimensions of projective prime order elements to enable us to apply Proposition \ref{tools}, or determine computationally whether $G$ has a regular orbit on $V$.
We are grateful to Eamonn O'Brien  and J\"urgen M\"uller for their assistance with certain computationally intensive calculations for groups $G$ with $E(G)/Z(E(G))$ equal to one of: $L_2(q)$ with $11\leq q\leq 81$, $L_3(4)$, $L_4(3)$, $U_3(3)$, $U_3(7)$, $U_4(3)$, $U_5(2)$, $U_5(3)$, $U_6(2)$, $\mathrm{PSp}_4(9)$, $\mathrm{PSp}_6(2)$, $\mathrm{PSp}_6(3)$, $\pom^+_8(2)$, $G_2(4)$ or ${}^3D_4(2)$.

\section{Proof of Theorem \ref{mainthm}: Linear groups}
In this section, we set out to prove the following theorem.
\begin{theorem}
\label{linearprop}
Suppose $G$ is an almost quasisimple group with $E(G)/Z(E(G)) \cong L_n(q)$, with $n\geq 2$.
%, with $(n,q) \neq (2,4),(2,5),(2,9), (4,2)$. 
Let $V=V_d(r)$, $(r,q)=1$ be a module for $G$, with absolutely irreducible restriction to $E(G)$. Also suppose that $(r, |G|)>1$. Then either:
\begin{enumerate}
\item $b(G) = \lceil \log |G|/\log |V| \rceil$, or 
\item $b(G) = \lceil \log |G|/\log |V| \rceil+1$ and $(G,V)$ appears in Table \ref{allbad}.
\end{enumerate}
\end{theorem}

%The excluded values of $(n,q)$ in Theorem \ref{linearprop} are due to exceptional isomorphisms with alternating groups, which have already been treated in \cite{MR3500766}.
%\begin{table}[h!]
%\begin{tabular}{@{}ccc@{}}
%\toprule
%Group $G$ & $V$ & $b(G)$ \\ \midrule
%$L_3(2)$ & $V_8(2)$&  2\\
%$(c\circ 2).L_3(2)$, $c\in \{4,8,12\}$, $(4\circ 2).L_3(2).2$ & $V_2(49)$ & 2\\
%$L_3(2)$, $L_3(2).2$, $2\circ L_3(2)$ & $V_3(7)$& 2\\
%$L_4(3)\leq G \leq L_4(3).2_2$ & $V_{26}(2)$ & 2 \\
%$4_1.L_3(4).2_3$ & $V_{8}(5)$ & 2 \\
%$2.L_3(4)\leq G\leq 2.L_3(4).2^2$ & $V_6(3)$ & 3 \\
%$2.L_3(4) \leq G\leq (\mathrm{F}_9^\times \circ 2).L_3(4).2^2$ & $V_6(9)$ & 2 \\
%$c\times L_3(2)$, $c\in \{2,4\}$ & $V_3(9)$ & 2 \\
%$L_3(2).2 \leq G\leq \mathbb{F}_3^\times \otimes L_3(2).2$ & $V_6(3)$ & 2 \\
%$L_2(31)$ & $V_{15}(2)$ & 2 \\
%$c \times L_2(13)$ $c\in \{1,2\}$ & $V_7(3)$ & 2 \\
%$3\times L_2(13)$ & $V_6(4)$ & 2 \\
%$L_2(13).c$,  $c\in \{1,2\}$ & $V_{14}(2)$ & 2 \\
%$L_2(11).c$,  $c\in \{1,2\}$ & $V_{10}(2)$ & 2 \\
%$2.L_2(11)$ & $V_6(3)$ & 2 \\
%$L_2(11)$ & $V_5(5)$ & 2 \\
%$L_2(11)$ & $V_5(4)$ & 2 \\
%$2\times L_2(8)$ & $V_7(3)$ & 2 \\
%$L_2(8).3$ & $V_7(3)$ & 2 \\ \bottomrule
%\end{tabular}
%\caption{Pairs $(G,V)$ with $G$ an almost quasisimple linear group where $b(G)  > \lceil \log |G|/\log |V| \rceil$.}
%\label{linearbad}
%\end{table}

Let $G$ be an almost quasisimple group with $E(G)/Z(E(G)) \cong L_n(q)$. 

\begin{proposition}
	Theorem \ref{linearprop} holds if $E(G)/Z(E(G)) \cong L_2(q)$ with $q> 7$ and $q\neq 9$.
\end{proposition}
\begin{proof}
Note that $q \neq 2,3$ since these groups are soluble,  and $q\neq 4,5,9$, since these groups are isomorphic to alternating groups and dealt with in Proposition \ref{alt_grp_prop}. We also assume that $q\neq 7$, since we deal with $L_2(7) \cong L_3(2)$ in Proposition \ref{L3_prop}. Therefore, from now on, assume $q > 7$ and $q\neq 9$.
	
	 In addition to cross-characteristic representations, we must also consider representations of $L_2(8)$ in characteristic 2, due to the exceptional isomorphism $L_2(8) \cong  {}^2G_2(3)'$.

	 Now \cite[Table 2]{HM} provided a list of the dimensions of cross-characteristic absolutely irreducible representations of $E(G)$, along with conditions on $r$ and any field irrationalities that occur in the corresponding (Brauer) character. 
	 %In particular,  the minimal degrees of the relevant representations of $G$ are $\frac{q-1}{2}$, $\frac{q+1}{2}$ and $q-1$ if $q$ is odd, and $q-1$, $q$ and $q+1$ if $q$ is even. 
	 The number of involutions in $G/F(G)$ is at most $2(q^2+q)$ by Proposition \ref{invols}, and the number of field automorphisms of order 2 is at most $q^{1/2}(q+1)$ by \cite[Table B.3]{bg}.
	Moreover if $x\in G$ is of projective prime order with image $\overline{x}\in G/F(G)$, then by Proposition \ref{L2lemma}, $\alpha(x)=2$ if $\bar{x}$ has odd order, $\alpha(x)\leq 4$ if $\overline{x}$ is an involutory field automorphism, and $\alpha(x)\leq 3$ otherwise.

	In order to find the base size of $G$ on $V$, we will use a variant of  Proposition \ref{tools}\ref{crude}. Namely, note that if Proposition \ref{tools}\ref{crude} holds, then, using the notation of Proposition \ref{tools}, 
	\[
	|V|< 2\sum_{x\in \mathcal{P} } |\bar{x}^H| r^{\lfloor(1-1/\alpha(x))d\rfloor}
	\]
	also holds. Therefore, if $G$ has no regular orbit on $V=V_d(r)$ then we have
	\[
	r^{d} \leq  4(q^2+q)r^{\floor{2d/3}}+2q^{1/2}(q+1)r^{\floor{3d/4}} +2 |\mathrm{P}\Gamma L_2(q)|r^{\floor{d/2}},
	\]
	since the number of odd prime order elements in $G/F(G) \leq \mathrm{P}\Gamma \mathrm{L}_2(q)$ is bounded above by $| \mathrm{P}\Gamma  \mathrm{L}_2(q)|$. By Proposition \ref{mindegprop}, $d_1(E(G)) = (q-1)/(2,q-1)$ and when we substitute this in to the inequality, we see that it is false for all $q$ except $q\leq 113$ for $q$ odd and $q\leq 32$ for $q$ even.

	If $83 \leq q\leq 113$, then $q$ is prime and examining \cite[Table 2]{HM}, we observe that we only need to consider $(d,r) = (\frac{q\pm 1}{2},2)$.
	 Note that if $q$ is prime, then there are no involutory field automorphisms in $G/F(G)$. Let $i_P(H)$ denote the number of prime order elements in a finite group $H$, and $r_0=\mathrm{char}(\F_r)$. If $G$ has no regular orbit on $V$ then by Proposition \ref{tools}\ref{crude},
	\begin{equation}
	\label{L2crude}
	r^d \leq i_2(\mathrm{Aut}( L_{2}(q)))r^{\floor{2d/3}}+ 2(i_P(\mathrm{Aut}( L_2(q))-i_2(\mathrm{Aut}( L_{2}(q))))r^{\floor{d/2}},
	\end{equation}
	where we calculate $i_{2}(\mathrm{Aut}( L_2(q)))$ and $i_P(\mathrm{Aut}( L_2(q)))$ in GAP.
	This is false for $q$ with $83 \leq q\leq 113$ and $r=2$.
	
	So now we consider $q\leq 81$. Since we want to consider $r$ in all cross-characteristics dividing the order of the group, we use a modified version of Proposition \ref{tools}\ref{crude} to deduce that if $G$ has no regular orbit on $V$, then 
	\[
	r^d \leq 2i_{r_0}(\mathrm{Aut}( L_{2}(q)))r^{\floor{c}}+ 2i_P(\mathrm{Aut}( L_2(q)))r^{\floor{d/2}},
	\]
	where $c=3d/4$ if $q$ is a square, and $c=2d/3$ otherwise. For each $q\leq 81$, this inequality is false for all but a finite number of $d$ (obtained from \cite[Table 2]{HM}) and $r$.
	%However, there are several pairs of $(d,r)$ where no irreducible module $V_d(r)$ for $G$ exists. For example, if $q$ is odd and $d = (q\pm 1)/2$, the Brauer character corresponding to $V=V_d(r)$ contains $b_q$ irrationalities, so $V$ exists if and only if $b_q$ lies in $\F_r$.
%	Now the minimal polynomial of $b_q$ is $f(x)=x^2+x-(eq-1)/4$. Therefore, if $\F_r$ does not contain a root of $f(x)$, then $b_q \notin \F_r$. For example, if $r=2$, then $f(x)$ is reducible if and only if $q \equiv \pm 1 \, (8)$.

	With this in mind, we give a list of the remaining groups with $E(G)/Z(E(G)) \cong L_2(q)$ we must consider in Table \ref{l2qtable}. Here we also list the largest value of $r$ we need to check for each $d$-dimensional module $V$, taking into account that $(r,|G|)>1$ and any irrationalities in the corresponding Brauer characters. The entries given as ``--'' in Table \ref{l2qtable} mean that we have no further modules to check of this dimension.
	%, usually because \eqref{L2crude} is false for all but a small collection of $r$ (if any), and there is no irreducible module of dimension $d$ over $\F_r$ for each remaining $r$.
	
	 We also apply \eqref{L2crude} to show that if $E(G)/Z(E(G)) \cong L_2(8)$ and $V$ is an absolutely irreducible $\mathbb{F}_{2^k}E(G)$-module for some $k$, then $G$ has a regular orbit on $V$ unless possibly $d=2$ and $r=8,84,512$, or $(d,r)$ is one of $(4,8)$ or $(8,2)$. 
	
	\begin{table}[h!]
		\centering \begin{tabular}{cccc}
			\toprule
			& \multicolumn{3}{c}{Upper bound for $r$}\\ 
			\cmidrule{2-4}
			$q$ & $d = (q-1)/2$ & $d = (q+1)/2$ & $d \geq q-1$ \\ 
			\midrule
			%7 & 100 & 17 & 10\\ 
			11 & 16 & 16 & 4 \\  
			13 & 16 & 8 & 4 \\
			17 & 9 & 4 & 3 \\  
			19 & 9 & 5 & 3 \\  
			23 & 4 & 3 & 2 \\  
			25 & 9 & 3 & 3 \\  
			27 & 4 & -- & 2 \\  
			29 & 4 & 3 & 2 \\  
			31 & 4 & 3 & 2 \\  
			37 & 3 & 3 & -- \\  
			41 & 3 & 3 & -- \\  
			43 & 3 & -- & -- \\ 
			47 & 2 & -- & -- \\  
			49 & 4 & 3 & 2 \\ 
			71&2& -- &--\\
			73&2& -- &--\\
			81 & 2 & -- & -- \\ 
			\bottomrule 
		\end{tabular} 
		\centering \begin{tabular}{cccc}
			\toprule
			& \multicolumn{3}{c}{Upper bound for $r$} \\ 
			\cmidrule{2-4}
			$q$ & $d = q-1$ & $d = q$ & $d =q+1$ \\ 
			\midrule
			8 & 5 & 5 & 3\\
			16 & 3 & 3 & 3\\
			\bottomrule 
		\end{tabular}
		\caption{Remaining cases for $L_2(q)$ in cross-characteristic. \label{l2qtable}}
	\end{table}

	%We obtain the set of absolutely irreducible modules to consider from \cite{HM} and also the relevant Brauer character tables. We do not include the modules that fail \eqref{L2crude}. 
	We include all of the information needed to analyse the remaining modules in Table \ref{l2q-case-analysis}. The first column of Table \ref{l2q-case-analysis} gives the layer $E(G)$ of $G$. The second column gives the list of primes (apart from 2 and 3) which divide the order of $E(G)$, given in ascending order. The third column gives the dimension $d$ and the field size $r$ of the irreducible module $V=V_d(r)$ being considered. The remaining columns give the eigenspace dimensions on $V$ of elements of projective prime order in $G$. 
	
	The eigenspace dimensions for the semisimple elements are sourced from the corresponding Brauer character, found either in \cite{modatlas} or constructed in GAP. We write the eigenspace dimensions using index notation so that, for example, we denote a set of eigenspaces with dimensions 16,12,12 as $(16, 12^2)$. We will usually write eigenspace dimensions in descending order. 
	
	We now explain our notation for the dimensions of fixed point spaces of elements $g$ of projective prime order dividing $r$ in Table \ref{l2q-case-analysis}. The italicised entries denote upper bounds for $\dim C_V(g)$, obtained using Proposition \ref{L2lemma}. The remaining entries give the true value(s) of $\dim C_V(g)$, which we derive from a construction of the module in GAP or Magma. 
	%The asterisked entry for projective involutions of $G$ with $E(G)=L_2(49)$ requires some extra explanation. 
	%The involutions in projective class 2D of size 175 have a fixed point space of dimension 15, while those in projective classes 2A, 2C of sizes 1225 and 175 respectively have fixed point spaces of dimension 12. Note that class 2B lies in $\mathrm{PGL}_2(49) \setminus L_2(49)$, and we do not consider it because $V$ does not extend to a module for $\pgl_2(49)$.

	We now apply \ref{eigsp1} from Proposition \ref{tools} to each module in Table \ref{l2q-case-analysis}. For example, suppose $E(G)=L_2(81)$ and $V=V_{40}(2)$, and let $H=G/F(G)$. We have $H \leq \mathrm{Aut}(L_2(81))=A$, and 
	\[
	i_2(A) = 7299, \quad i_3(A)=6560, \quad i_5(A)=13284, \quad i_{41}(A)=129600.
	\]
	The group $A$ does not itself have a 40-dimensional irreducible module over $\F_2$, but $L_2(81).4$, the group generated by $L_2(81)$ and its field automorphisms, does. Since we are only interested in prime order elements, we can work with $L_2(81).2$, which contains involutory field automorphisms. 
	There are three classes of involutions in $L_2(81).2$: one class of size 3321 lying in $L_2(81)$, and two classes of size 369 outside $L_2(81)$, which both consist of field automorphisms. The latter two classes are fused under the action of $\mathrm{PGL}_2(81)$. We construct $V$ in Magma \cite{Magma} and find that as described in Table \ref{l2q-case-analysis}, the inner involutions (those in class 2A) have a fixed point space of dimension 20, while the involutory field automorphisms have fixed point spaces of dimension 24.
	
	If $G$ has no regular orbit on $V$, then by Proposition \ref{tools} \ref{eigsp1}, together with the eigenspace dimensions computed in Table \ref{l2q-case-analysis}, we have
	\begin{align*}
	r^{40} &\leq 2\times 369r^{24}+|2A|r^{20}+\frac{1}{2}|3A|(2r^{15}+r^{10})+\frac{1}{2}|3B|(r^{16}+2r^{12})+\frac{1}{4}i_5(H)(5r^8)+\frac{1}{40}i_{41}(H)(40r) \\
	& \leq 2\times 369r^{24}+3321r^{20} +\frac{1}{2}6560(r^{16}+2r^{12})+\frac{1}{4}13284(5r^8)+\frac{1}{40}129600(40r) .
	\end{align*}
	But this inequality does not hold when $r=2$, so $G$ has a regular orbit on $V$.
	
Continuing in this way and applying \ref{eigsp1} from Proposition \ref{tools} to each module in Table \ref{l2q-case-analysis}, we see that in all cases, $G$ has a regular orbit on $V$.
	
	There are some modules for which this technique fails i.e., we cannot successfully apply any of the inequalities in Proposition \ref{tools} to show that $G$ has a regular orbit on $V$. These remaining cases are set out in Table \ref{l2q-ros}, where we compute the base size directly in Magma. Note that there is some intricacy involved with the computation of $\lceil \log |G|/ \log |V| \rceil$ for each row in Table \ref{l2q-ros}. Namely, the given values of $\lceil \log |G|/ \log |V| \rceil$  are computed only for the groups $G$ for which each module  $V$ exists, which we determine in each case from the Brauer character tables in \cite{modatlas}. For example, if $E(G)/Z(E(G)) \cong L_2(13)$ and $V=V_7(3)$, we have $\lceil \log |G|/ \log |V| \rceil=1$ for the two possible groups $G$, namely $L_2(13)$ and $\F_3^\times \times L_2(13)$. On the other hand,  $\lceil \log |2\times \pgl_2(13)|/\log|V|\rceil =2$, but $V$ is not a module for $\F_3^\times \times \pgl_2(13)$, so we do not record this in Table \ref{l2q-ros}.
	{\small
		\begin{table}[!htbp]
			\centering \begin{tabular}{lllllllll}
				\toprule
				&                 &                        & \multicolumn{5}{c}{Projective prime order}                                  \\
				\cmidrule{4-8}
				$E(G)$    & $(p_3,p_4,p_5)$ & $(d,r)$                & 2                 & 3            & $p_3$        & $p_4$        & $p_5$      \\
				\midrule
				$L_2(81)$   & $(5,41,-)$     & $(40,2)$               & 20 \tiny{(2A)} & $(15^2, 10)$ \tiny{(3A)} & $(8^5)$      & $(1^{40})$   & --         \\
				&&& 24 \tiny{(o/w)}&$(16, 12^2)$ \tiny{(3B)}&&&\\
				\midrule
				$L_2(73)$   & $(37,73,-)$      & $(36,2)$               & 18               & $(12^3)$     & $(1^{36})$   &$(1^{36})$    & --         \\				
				\midrule
				$L_2(71)$   & $(5,7,71)$      & $(35,2)$               & 18               & $(12^2,11)$     & $(7^5)$   &$(5^7)$    & $(1^{35})$       \\				
\midrule				
				$L_2(49)$   & $(5,7,-)$      & $(48,2)$               & \textit{36 }               & $(16^3)$     & $(10^4,8)$   & $(7^6,6)$    & --         \\
				&      & $(48,2)$               & \textit{36}                & $(16^3)$     & $(10^3, 9^2)$   & $(7^6,6)$    & --         \\
				&                 & $(25,3)$               & $(16,9)$ \tiny{(2C)}          & \textit{12 }          & $(5^5)$      & $(4^6,1)$ 	\tiny{(7A)}   & --         \\		
				&                 &                &  $(13,12)$ \tiny{(o/w) } &    &   & $(7,3^6)$ \tiny{(7B)}    &         \\			
				&                 & $(24,4)$         &     $15$ \tiny{(2D)}    & $(8^3)$         & $(5^4,4)$    & $(6,3^6)$ \tiny{(7A)}, & --         \\
								&                 &               &   $12$ \tiny{(o/w) }   & &  &  $(4^6)$ \tiny{(7B)}   &       \\
%				&                 & $r=2,4$               &      & &  &  $(4^6)$ \tiny{(7B)}   &       \\
				$2.L_2(49)$ &                 & $(24,3)$               & $(12^2)$          & \textit{12 }          & $(5^4,4)$    & $(6,3^6)$ \tiny{(7A)}, & --         \\
				&                 &  &      & &  &  $(4^6)$ \tiny{(7B)}   &       \\
				\midrule
				$L_2(47)$   & $(23,47,-)$    & $(23,2)$               & 12                & $(8^2,7)$    & $(1^{23})$   & $(1^{23})$   & --         \\
				\midrule
				$L_2(43)$   & $(7,11,43)$     & $(21,9)$               & $(11,10)$         & \textit{10  }       & $(3^7)$      & $(2^{10},1)$ & $(1^{21})$ \\
				\midrule
				$L_2(41)$   & $(5,7,41)$      & $(21,9)$               & $(11,10)$         & \textit{10 }          & $(5,4^4)$    & $(3^7)$      & $(1^{21})$ \\
				&                 & $(20,2)$               & 10                & $(7^2,6)$            & $(4^5)$          & $(3^6,2)$            & $(1^{20})$         \\
				$2.L_2(41)$   &                 & $(20,3)$               & $(10^2)$          &\textit{ 10}           & $(4^5)$      & $(3^6,2)$    & $(1^{20})$ \\
				\midrule
				$L_2(37)$   & $(19,37,-)$    & $(19,3)$               & $(10,9)$          & \textit{9}            & $(1^{19})$   & $(1^{19})$   & --         \\
				$2.L_2(37)$   &                 & $(18,3)$               & $(9^2)$           & \textit{9}            & $(1^{18})$   & $(1^{18})$   & --         \\
				\midrule
				$L_2(31)$   & $(5,31,-)$     & $(32,2)$               & \textit{21}                & $(11^2,10)$           & $(8,6^4)$          & $(2,1^{30})$           & --         \\
				&                & $(15,4)$               & 8                 & $(5^3)$      & $(3^5)$      & $(1^{15})$   & --         \\
				%$2.L_2(31)$   &                 & $(15,3)$               & $(8,7)$           & 7            & $(3^5)$      & $(1^{15})$   & --         \\ No 15D rep of 2.L_2(31) (only 16D rep and 15D rep of L_2(31) and contains b31 irr, in F9)
				\midrule
				$L_2(29)$   & $(5,7,29)$      & $(28,2)$               & \textit{18}                & $(10,9^2)$           & $(6^3,5^2)$           & $(4^7)$           & $(1^{28})$         \\
				&                 & $(15,9)$               & $(8,7)$           &\textit{ 7 }           & $(3^5)$      & $(3,2^6)$    & $(1^{15})$ \\
				&                 & $(14,4)$               & 7                 & $(5^2,4)$    & $(3^4,2)$    & $(2^7)$      & $(1^{14})$ \\
				$2.L_2(29)$   &                 & $(14,3)$               & $(7^2)$           & \textit{7  }          & $(3^4,2)$    & $(2^7)$      & $(1^{14})$ \\
				\midrule
				$L_2(27)$   & $(7,13,-)$     & $(13,4)$               & 7                 & $(5,4^2)$ \tiny{(3C/D)}  & $(2^6,1)$    & $(1^{13})$   & --         \\
				&   &            &           &  $(6,4,3)$ \tiny{(o/w)}    &    &   &     \\
				%   &                 & $(12,9)$ $(\chi_{12})$ & $(6^2)$           & 4             & $(3^4)$      & $(1^{12})$ & -- \\ Can incorporate into case above
								\midrule
$L_2(25)$   &    $(5,13,-)$             

% & $(13,3)$             & $(9,4)$ (2C), $(7,6)$ (o/w)& 5        &    $(5,2^4)$ (5A), $(3^4,1)$ (5B) &       $(1^{13})$  &            --\\ %Added 8/6/20 - actually addressed in table above

& $(12,r)$, & 8 \tiny{(2D)}   &   $(4^3)$       &    $(3^4)$ \tiny{(5A)}, &       $(1^{12})$  &            --\\
&           & $r=4,8$       &  6 \tiny{(o/w)}     &    $(3^4)$ &    $(4,2^4)$ \tiny{(5B)}     &            --\\
					$2.L_2(25)$   &                 & $(12,r)$, & $(6^2)$           & 4               &  $(3^4)$, \tiny{(5A)} &       $(1^{12})$  &            --\\
				&                 & $r=3,9$ &  &  &  $(4,2^4)$ \tiny{(5B)} &       &            --\\
				\midrule
			$L_2(23)$   & $(11,23,-)$    & $(22,2)$               & 12 \tiny{(2A)}                & $(8,7^2)$    & $(2^{11})$   & $(1^{22})$   & --         \\
&   &   & 11 \tiny{(2B)}             &   & &   &   \\
			&                 & $(11,4)$               & 6                 & $(4^2,3)$    & $(1^{11})$   & $(1^{11})$   & --         \\
				\midrule
$L_2(19)$   & $(5,19,-)$     & $(20,2)$               & 10 \tiny{(2A)}                & $(7^2,6)$    & $(4^5)$      & $(2,1^{17})$ & --         \\
			&&  & 11  \tiny{(2B)}    &   & &  &        \\
			&                 & $(19,3)$               & $(10,9)$                & \textit{9  }          & $(4^4,3)$           & $(1^{19})$            & --         \\
%		$L_2(17)$   & $(17,-,-)$    & $(18,3)$               & $(10,8)$ (2A)         & 6            & $(2,1^{16})$ & --           & --         \\
%			& &   & $(9^2)$ (2B)         & &  &  &   \\
%			&                 & $(16,3)$               & $(8^2)$ (2A),        & 6            & $(1^{16})$   & --           & --         \\
%			&                 & &  $(9,7)$ (2B)          & &   &          &        \\
%			&                 & $(9,9)$                & $(5,4)$           & 3            & $(1^9)$      & --           & --         \\
%			 &                 & $(8,4)$                & 4                 & $(3^2,2)$    & $(1^8)$      & --           & --         \\
				\bottomrule
			\end{tabular}
			\caption{Upper bounds for $\dim C_V(g)$ for projective prime order elements $g \in G$ with $E(G)/Z(E(G)) =L_2(q)$, Part I.}
			\label{l2q-case-analysis}
		\end{table}
	}
\clearpage
	{\small
	\begin{table}[h!]
		\ContinuedFloat
		\centering \begin{tabular}{llllllll}
			\toprule
			&                 &                        & \multicolumn{5}{c}{Projective prime order}                                  \\
			\cmidrule{4-8}
			$E(G)$    & $(p_3,p_4,p_5)$ & $(d,r)$                & 2                 & 3            & $p_3$        & $p_4$        & $p_5$      \\
			\midrule

%			\midrule$L_2(19)$		
$L_2(19)$   & $(5,19,-)$          & $(9,4)$                & 5                 & $(3^3)$      & $(2^4,1)$    & $(1^9)$      & --         \\
	
				&                 & $(9,5)$                & $(5,4)$           & $(3^3)$      & \textit{4}            & $(1^9)$      & --         \\
			&                 & $(9, 9)$               & $(5,4)$           & \textit{4  }          & $(2^4,1)$    & $(1^9)$      & --         \\
			\midrule
										$2.L_2(19)$   && $(18,3)$ & $(9^2)$ & 6 & $(4^4,2)$ & $(1^{18})$ & -- \\
			&                 & $(10,5)$               & $(5^2)$           & $(4, 3^2)$   & \textit{5 }           & $(1^{10})$   & --         \\
			&                 & $(10, 9)$              & $(5^2)$           & \textit{5}            & $(2^5)$      & $(1^{10})$   & --         \\ 
	\midrule
					$L_2(17)$   & $(17,-,-)$    & $(18,3)$               & $(10,8)$ (2A)         & 6            & $(2,1^{16})$ & --           & --         \\
			& &   & $(9^2)$ (2B)         & &  &  &   \\
			&                 & $(16,3)$               & $(8^2)$ (2A),        & 6            & $(1^{16})$   & --           & --         \\
			&                 & &  $(9,7)$ (2B)          & &   &          &        \\
			&                 & $(9,9)$                & $(5,4)$           & 3            & $(1^9)$      & --           & --         \\
			 &                 & $(8,4)$                & 4                 & $(3^2,2)$    & $(1^8)$      & --           & --         \\
			 \midrule
			$2.L_2(17)$  	 
			&                 & $(8,9)$                & $(4^2)$           & 3            & $(1^8)$      & --           & --         \\
			\midrule
							$L_2(16)$   & $(5,17,-)$     & $(16,3)$               & $(8^2)$ (2A),           &\textit{ 8}            & $(4,3^4)$    & $(1^{16})$   & --         \\
							&  & & $(10,6)$ (2B)          &    &   &&       \\
							\midrule
							$L_2(13)$   &    $(7,13,-)$      & $(13,3)$               & (7,6)                 & 5            & $(2^6,1)$            & $(1^{13})$            & --         \\
							$2.L_2(13)$   &    & $(14,3)$               & $(7^2)$               & \textit{7 }           & $(2^7)$           & $(2,1^{12})$            & --         \\
							\midrule
							$L_2(11)$   & $(5,11,-)$     & $(12,4)$               & \textit{8 }                & $(4^3)$      & $(3^2,2^3)$  & $(2,1^{10})$ & --         \\
							%					 &                 & $(5,5)$                & $(3,2)$           & $(2^2,1)$    & 2            & $(1^5)$      & --         \\
						   &                 & $(10,4)$               & 6 (2A),    & $(4, 3^2)$   & $(2^5)$      & $(1^{10})$   & --        \\
							
							&       &   & 5 (2B)    &  &  &   &   \\
							&                 & $(5,9)$                & $(3,2)$           & \textit{2 }           & $(1^5)$    & $(1^5)$      & --         \\
							&                 & $(5,16)$               & 3                 & $(2^2,1)$    & $(1^5)$      & $(1^5)$      & --         \\
							
							$2.L_2(11)$ &                 & $(10,3)$               & $(5^2)$           & \textit{5}            & $(2^5)$      & $(1^{10})$   & --         \\
							&                 & $(6,9)$                & $(3,3)$           & 2            & $(2,1^4)$      & $(1^6)$      & --         \\
			\bottomrule
		\end{tabular}
		\caption{Upper bounds for $\dim C_V(g)$ for projective prime order elements $g \in G$ with $E(G)/Z(E(G)) = L_2(q)$, Part II.}
	\end{table}
}

% except where $(q,d,r)=(11,5,5)$, where we find in GAP that $b(G)=2$.
The asterisked cases in Table \ref{l2q-ros} require some extra explanation.
\begin{enumerate}[(a)]
	\item  If $E(G)/Z(E(G)) \cong L_2(13)$ and $V=V_6(4)$, then $G$ has a regular orbit on $V$ unless $G = \mathbb{F}_4^\times \times L_2(13)$ where $b(G)=2$. Moreover, $L_2(13)$ has a regular orbit on $V_{14}(2)$, while $L_2(13).2$ does not.
	\item If $E(G)/Z(E(G))\cong L_2(11)$, and $V=V_{10}(2)$ or $V_5(4)$, then $\lceil \log|G|/\log|V|\rceil=1$ if and and only if $G=L_2(11)$ and is 2 otherwise.
	\item If instead $E(G)/Z(E(G)) \cong L_2(8)$ and $V=V_7(3)$ or $V_4(8)$,
	then $b(G)=2$ unless $G=L_2(8)$, where $b(G)=1$. Moreover, in the former case, $\lceil \log|G|/\log|V|\rceil\leq 2$ with equality if and only if $G=\F_3^\times \times L_2(8).3$.
	Finally, if $V=V_2(64)$, so, examining the Brauer character table, we deduce that  $G = c\times L_2(8)$ for some $c\mid 63$, then $\lceil \log|G|/\log|V| \rceil>1$ if and only if $c \geq 9$, and moreover $b(G)=2$ if $3\mid c$ and $b(G)=1$ otherwise.\qedhere
\end{enumerate}
\end{proof}
		\begin{table}[h!]
	\centering \begin{tabular}{@{}llll@{}}
		\toprule
		$E(G)/Z(E(G))$ & $(d,r)$ & $\lceil \frac{\log|G|}{\log|V|} \rceil$ & $b(G)$ \\ \midrule
		$L_2(49)$ &$(24,2)$ &1&1\\
		$L_2(31)$ & $(15,2)$ & 1 & 2   \\
		$L_2(25)$ & $(13,3)$ & 1 & 1     \\
		& $(12,2)$ & 2 & 2   \\
		$L_2(23)$ & $(11,3)$ & 1 & 1    \\
		& $(12,3)$ & 1 & 1    \\
		& $(11,2)$ & 2 & 2   \\
		$L_2(17)$ & $(16,2)$ & 1 & 1  \\
		& $(8,2)$ & 2 & 2   \\
		$L_2(13)$ & $(7,3)$ & 1 & $2^*$  \\
		& $(6,3^k)$, $k\geq 2$ & 1 & 1  \\
		& $(6,3)$ & 2 & 2   \\
		& $(14,2)$ & 1 & 2  \\
		& $(6,16)$ & 1 & 1   \\
		\bottomrule
	\end{tabular}
	\quad 
	\centering \begin{tabular}{@{}llll@{}}
		\toprule
		$E(G)/Z(E(G))$ & $(d,r)$ & $\lceil \frac{\log|G|}{\log|V|} \rceil$ & $b(G)$ \\ \midrule
		$L_2(13)$
		& $(6,4)$ & 1 & 2* \\
		$L_2(11)$ & $(10,3)$ & 1 & 1 \\
		& $(10,2)$ & $1^*$ & 2  \\
		& $(6,3)$ & 1 & 2   \\
		& $(5,5)$ & 1 & 2   \\
		& $(5,4)$ & $1^*$ & 2   \\
		& $(5,3)$ & 2 & 2  \\
		& $(6,5)$ & 1 & 1  \\
		$L_2(8)$  & $(7,3)$ & $1^*$ & 2*   \\
		& $(8,2)$ & 2 & 2\\ 
		& $(4,8)$ & 1 & 2*\\
		& $(2,512)$ & 1& 1\\ 
		& $(2,64)$ & 1 & 2*\\ 
		& $(2,8)$ & 2 & 2\\ 
		\bottomrule
	\end{tabular}
	\caption{Base sizes for some irreducible modules for $G$ with $E(G)/Z(E(G))\cong \mathrm{L}_2(q)$.}
	\label{l2q-ros}
\end{table}

%\subsection{$n=3$}
%DONEx2
\begin{proposition}
	\label{L3_prop}
	Theorem \ref{linearprop} holds for $G$ with $E(G)/Z(E(G)) \cong L_3(q)$.
\end{proposition}
\begin{proof}
	To begin with, let us assume $q\neq 2,4$.
By Proposition \ref{alphas}, if $x \in G$ has projective prime order, then $\alpha(x) \leq 3$ unless $x$ is an involutory graph-field automorphism where $\alpha(x)\leq 4$. Therefore, if $G$ has no regular orbit on $V = V_d(r)$, then by Proposition \ref{tools}\ref{crude} and Lemma \ref{graph}\ref{lgraph},
\[
	r^d\leq 2(|\mathrm{P}\Gamma\mathrm{L}_3(q).2|+2q^4)r^{\floor{2d/3}}+ 4q^4r^{\floor{3d/4}},
	\]
	since the number of non-graph prime order elements in $G/F(G)$ is bounded above by $|\mathrm{P}\Gamma\mathrm{L}_3(q)|$. If we substitute in $d_1(E(G))=q^2+q-1$ from Proposition \ref{mindegprop}, then the inequality is false except for when $(q,r) = (7,2)$, $q=5$ and $r\leq 3$, and $q=3$ with $r\leq 11$, so $G$ has a regular orbit on all choices of $V$ in all other cases.
	
	We now examine each of the remaining cases, along with the cases with $q=2,4$ in detail.

%\subsubsection{$q=7$}
%DONEx2
Suppose that $E(G)/Z(E(G)) \cong L_3(7)$ and $r=2$. From \cite{HM}, the smallest non-trivial absolutely irreducible representation of $G$ when $r=2$ has degree 56. The number of prime order elements in $\mathrm{Aut}(L_3(7))$ is 853481. %No graph field auts since 7 is prime
 Therefore, by Proposition \ref{tools}\ref{crude}, if $G$ has no regular orbit on $V$, then
\[
2^{d} \leq 2 \times 853481\times 2^{\floor{2d/3}}.
\]
This is false for all $V$ except for those with $56\leq d\leq 60$ so by \cite{HM}, $V=V_{56}(2)$ or $V_{57}(2)$.
	Therefore, $G$ has a regular orbit on $V$ except perhaps when $d=56,57$.
In the latter cases, examining the corresponding Brauer character in \cite{modatlas}, we determine that no element of $G$ of odd projective prime order has eigenspace of dimension greater than 24, i.e., $\emax(g_{2'}) \leq 24$. We also compute that $i_2(\mathrm{Aut}(L_3(7))) = 19551$. So if $G$ has no regular orbit on $V=V_{d}(2)$ with $d\in \{56,57\}$, then by Proposition \ref{tools}\ref{qsgood},
\[
2^{d} \leq 2\times 853481\times 2^{24}+ 19551\times 2^{\floor{2d/3}}
\]
This is false for both values of $d$, so $G$ has a regular orbit on $V$ in both cases.
%\subsubsection{$q=5$}
\par %DONEx2
Now let $E(G)/Z(E(G)) \cong L_3(5)$ and $r\leq 3$. Since $G$ has no graph-field automorphisms, $\alpha(x)\leq 3$ for all projective prime order $x\in G$. We also compute that the number of prime order elements in $G/F(G)$ is at most 154999.  If $G$ has no regular orbit on $V = V_d(r)$, then $r^d\leq 2\times 154999 r^{\floor{2d/3}}$ which implies that $d \leq 54,33$ for $r=2$ and 3 respectively. From \cite{HM} we deduce that we only need to examine one  representation of degree 30 for $r=2$, while if $r=3$ there is one representation of degree 30 and another of degree 31 to consider. First suppose $r=3$. The Brauer character for the 30-dimensional representation of $L_3(5).2$ contains the irrationality $r_5 \notin \mathbb{F}_3$ so we only need to consider this representation for $L_3(5)$.
There are at most 15500 elements of order 3 in $G/F(G)$, and at most 3875 involutions.
Examining the Brauer characters of the 30- and 31-dimensional representations, we find that $\emax(g_{\{2,3\}'})\leq 10$ and 11 respectively. In addition,  $\emax(g_2) \leq 18$ and 19 respectively. Therefore, if $G$ has no regular orbit on the 30-dimensional module, then by Proposition \ref{tools}\ref{qsgood},
\[
r^{30} \leq 2 \times 154999 r^{10}+2\times 3875r^{18}+15500r^{\floor{2d/3}},
\]
which is false for $r=3$. A similar inequality gives the result for $d=31$.

Now suppose that $r=2$ and $V = V_{30}(2)$ is an absolutely irreducible $\mathbb{F}_2E(G)$-module. We determine in Magma that $\emax(g_2) \leq 18$. Examining the Brauer character in \cite{modatlas}, we see that elements in classes 3A and 5A have eigenspaces of dimension at most 10, and no other element of prime order has an eigenspace of dimension greater than 6. Therefore, if $G$ has no regular orbit on $V$, by Proposition \ref{tools}\ref{qsgood},
\[
r^{30} \leq 2\times 154999r^{6}+ 2(|3A|+|5A|)r^{10}+3875r^{18}.
\]
This is false for $r =2$, so $G$ has a regular orbit on $V$.
%\subsubsection{$q=4$}
\par%DONEx2 
Suppose that $E(G)/Z(E(G)) \cong L_3(4)$.
By \cite[Table B.3]{bg}, the number of graph-field involutory automorphisms in $L_3(4).D_{12}$ is $|\mathrm{PGL}(3,4)|/|\mathrm{PGU}(3,2)| = 280$, and we compute in GAP that the number of prime order elements in $G/F(G)$ is at most $20619$.
Therefore, if $G$ has no regular orbit on the irreducible module $V$ then by Proposition \ref{tools}\ref{crude},
\[
r^d \leq 2\times 20619 r^{\floor{2d/3}}+ 2\times 280r^{\floor{3d/4}}
\]
Therefore, either $G$ has a regular orbit on $V$, or $V$ appears in Table \ref{L34list}.
\begin{table}[h!]
\centering
\begin{tabular}{ccccc}
\toprule
$d$ & $E(G)$ & Characteristic & Irrationalities & Notes \\ 
\midrule
4 & $4_2.L_3(4)$ & 3 & $i_1,r_7$ & $r=9^k$ \\ 
6 & $6.L_3(4)$ & $\neq 2,3$ & $z_3$ & $r=25^k$ or $7^k$ \\ 
6 & $2.L_3(4)$ & 3 &  &  $r=3^k$\\ 
8 & $4_1.L_3(4)$ & $\neq 2,5$ & $i_1,b_5$ & $r=9$ \\ 
8 & $4_1.L_3(4)$ & 5 & $i_1$ &$r=5,25$  \\ 
10 & $2.L_3(4)$ & $\neq2,7$ & $b_7$ & $r=9$ \\ 
10 & $2.L_3(4)$ & 7 &  & $r=7$ \\ 
15 & $L_3(4)$ & 3 &  & $r=3$ \\ 
15 & $3.L_3(4)$ & $\neq 2,3$ & $z_3$ & $r=7$ \\ 
19 & $L_3(4)$ & 3,7 &  & $r=3$ \\ 
\bottomrule
\end{tabular} 
\caption{List of remaining cases for $E(G)/Z(E(G)) \cong L_3(4)$.\label{L34list}}
\end{table}

We will investigate the remaining modules in decreasing order of dimension. 

We summarise our analysis across Tables \ref{l34-case-analysis} and \ref{l34-ros}. In Table \ref{l34-case-analysis}, the eigenspaces for elements of projective prime order coprime to $r$ are derived from the Brauer character table, while those of projective prime order dividing $r$ are derived from Proposition \ref{tools}\ref{alphabound}, unless they are italicised, in which case the eigenspaces are computed using GAP. Using Proposition \ref{tools} \ref{eigsp1}, we see that if $V$ lies in Table \ref{L34list}, either $V$ also appears in Table \ref{l34-ros}, where the base size $b(G)$ is determined computationally from a construction in GAP, or $G$ appears in Table \ref{l34-case-analysis} and $G$ has a regular orbit on $V$ by Proposition \ref{tools}\ref{eigsp1}.In the latter case, if $(d,r)=(19,3)$, we must also consider elements in classes 3B and 3C, which have sizes 672 and 1920 respectively, and fixed point space dimensions 7 and 9 respectively. 
	
	{\footnotesize
		
		\begin{table}[]
 \begin{tabular}{@{}lcccccccccc@{}}
	\toprule
	$E(G)$ & $(d,r)$ & 2A & 3A & 5A & 5B & 7A & 7B & 2B & 2C & 2D \\ \midrule
	$L_3(4)$ & $(19,3)$ & $(11,8)$ & \textit{7} & $(4^4, 3)$ & $(4^4, 3)$ & $(3^6,1)$ & $(3^6,1)$ & $(10,9)$ & $(12,7)$ & $(10,9)$ \\
	$3.L_3(4)$ & $(15,7)$ & $(8,7)$ & $(5^3)$ & $(3^5)$ & $(3^5)$ & 10 & 10 & $(9,6)$ & -- & -- \\
	$L_3(4)$ & $(15,3)$ & $(8,7)$ & \textit{5} & $(3^5)$ & $(3^5)$ & $(3,2^6)$ & $(3,2^6)$ & $(9,6)$ & $(9,6)$ & $(9,6)$ \\
	$2.L_3(4)$ & $(10,7)$ & $(6,4)$ & $(4,3^2)$ & $(2^5)$ & $(2^5)$ & \textit{2} & \textit{2} & $(5^2)$ & $(7,3)$ & $(6,4)$ \\
	& $(10,9)$ & $(6,4)$ & \textit{4} & $(2^5)$ & $(2^5)$ & $(2^3,1^4)$ & $(2^3,1^4)$ & -- & $(7,3)$ & -- \\
	$4_1.L_3(4)$ & $(8,9)$ & $(4^2)$ & \textit{3} & $(2^3,1^2)$ & $(2^3,1^2)$ & $(2,1^6)$ & $(2,1^6)$ & -- & -- & (5,3) \\
	& $(8,25)$ & $(4^2)$ & $(3^2,2)$ & 5 & 5 & $(2,1^6)$ & $(2,1^6)$ & -- & -- & (5,3) \\
	$2.L_3(4)$ & $(6,3^k)$ & $(4,2)$ & 2 & $(2,1^4)$ & $(2,1^4)$ & $(1^6)$ & $(1^6)$ & $(3^2)$ & $(3^2)$ & $(3^2)$ \\
	$6.L_3(4)$ & $(6,5^k)$ & $(4,2)$ & $(3^2)$ & 2 & 2 & $(1^6)$ & $(1^6)$ & $(3^2)$ & -- & -- \\
	$6.L_3(4)$ & $(6,7^k)$ & $(4,2)$ & $(3^2)$ & $(2,1^4)$ & $(2,1^4)$ & 1 & 1 & $(3^2)$ & -- & -- \\
	$4_2.L_3(4)$ & $(4,9^k)$ & $(2^2)$ & 2 & $(1^4)$ & $(1^4)$ & $(1^4)$ & $(1^4)$ & -- & $(3,1)$ & -- \\
	\midrule
	& \#  & 315 & 2240 & 4032 & 4032 & 2880 & 2880 & 280 & 360 & 1008 \\ \bottomrule
\end{tabular}
			\caption{Eigenspace dimensions of projectively prime order elements in $G$ with $E(G)/Z(E(G)) \cong L_3(4)$. \label{l34-case-analysis}}
		\end{table}
	}

\begin{table}[h!]
\begin{tabular}{@{}cccc@{}}
\toprule
$G$                                & $(d,r)$  & $\lceil \log |G| /\log|V| \rceil$ & $b(G)$                                        \\ \midrule
$2\times L_3(4).2^2$               & $(15,3)$ & 1                                 & 1       \\
$4_1.L_3(4)$                   & $(8,5)$  & 1                                 & 1       \\
$4_1.L_3(4).2_3$                   & $(8,5)$  & 1                                 & 2       \\
$2.L_3(4)\leq G \leq F_r^\times \circ (2.L_3(4).2^2)$                     & $(6,3)$  & 2                                 & 3       \\
                                   & $(6,9)$  & 1                                 & 2      \\
$6.L_3(4) \leq G\leq 6.L_3(4).2_1$                     & $(6,7)$  & 2                                 & 2       \\
$4_2.L_3(4) \leq G \leq \mathbb{F}_9^\times \circ (4_2.L_3(4).2_2)$  & $(4,9)$  & 2                                 & 2   \\
$4_2.L_3(4) \leq G \leq 20\times \mathbb{F}_{81}^\times \circ (4_2L_3(4).2_2)$ & $(4,81)$ & 1                                 & 1  \\ \bottomrule
\end{tabular}
\caption{Some base size results for groups $G$ with $E(G)/Z(E(G)) \cong L_3(4)$. \label{l34-ros}}
\end{table}

%\subsubsection{$q=3$}
\par%DONEx2 
Now let $E(G)/Z(E(G)) \cong L_3(3)$ and $r\in \{2,4,8\}$. We compute that the number of prime order elements in $\mathrm{Aut}(L_3(3))$ is 2807.  Therefore, if $G$ has no regular orbit on $V$ then by Proposition \ref{tools}\ref{crude}, we have $r^d \leq 2\times 2807 r^{\floor{2d/3}}$. Substituting in the degrees of the absolutely irreducible representations of $E(G)$ from \cite{modatlas}, we see that $G$ has a regular orbit on all $V$ except possibly when $d\leq 12,18,36$ for $r=8,4$ and $2$ respectively. From  \cite{HM}
	 we deduce that the modules we need to consider either have $d=12$, or $(d,r)=(26,2)$. If $(d,r)=(26,2)$, then using GAP, we find that there is a regular orbit.  
We may also construct the 12-dimensional module in GAP. Note that $G$ has no regular orbit on the 12-dimensional module when $r=2$, since $|V|<|L_3(3)|$. When $r=2$, we find that there is a base of size two and when $r=4,8$, there is a regular orbit.
%\subsubsection{$q=2$}
%DONE

Finally, let $E(G)/Z(E(G)) \cong L_3(2)$.
 Since $L_3(2) \cong L_2(7)$, we will consider representations in characteristics 2, 3, 7. By Proposition \ref{mindegprop}, we have $\alpha(x) \leq 3$ for non-trivial $x \in G/F(G)$ by Proposition \ref{alphas}, and we compute that there are 153  prime order elements in $\mathrm{Aut}(L_3(2))$. If $G$ has no regular orbit on $V=V_d(r)$ then by Proposition \ref{tools}\ref{crude},
\[
r^d \leq 2\times 153 r^{\floor{2d/3}}.
\]
	Substituting in $d=2$, we see that $G$ has a regular orbit on $V$ unless possibly when $r\leq 297$. 

We construct all of the modules $V$ that are left to consider in GAP and determine their base sizes computationally. We summarise our findings in Table \ref{l32-case-analysis}. The first column of Table \ref{l32-case-analysis} gives the layer $E(G)$ and the second column gives the dimension of the module $V$ being considered. The third column gives information about the finite fields where $V$ is realised. Here 
\[
\nu = \begin{cases}
1 & \textrm{if } G/F(G) \textrm{ is simple},\\
2 & \textrm{otherwise.}
\end{cases}
\] 
The remaining columns give the base sizes of groups $G$ with the specified $E(G)$ on $V=V_d(r)$. If the entry is asterisked, then $G$ has a regular orbit on $V=V_d(r)$ if $G=L_3(2)$, and $b(G)=2$ otherwise. The entry marked with a dagger represents that $G$ has a regular orbit on $V=V_2(49)$ if $F(G)$ has order 2 or 6, and $b(G)=2$ otherwise. There are some cases included in the table where $b(G) = \lceil \log|G|/\log |V| \rceil +1$, namely $G=L_3(2), L_3(2).2$ or $2 \times L_3(2)$ and $V=V_3(7)$; $G=L_3(2)$ and $V=V_8(2)$;$G=c\times L_3(2)$ with $c\in \{2,4\}$ and $V=V_3(9)$; $L_3(2)\leq G\leq \F_3^\times \times L_3(2).2$ with $V=V_6(3)$;$G$ quasisimple with $|F(G)|=4,8,12$, or $G/F(G) = L_3(2).2$, $|F(G)|=4$, and $V = V_2(49)$. All of these are included in Table \ref{allbad}.
\end{proof}

\begin{table}[h!]
\begin{tabular}{@{}cccccc@{}}
\toprule
$E(G)$ & $d$ & Field & $b(G)=1$ & $b(G)=2$ & $b(G)=3$ \\ \midrule
$L_3(2)$ & 3 & $r=2^k$ & $r=8^*,r\geq 16$ & $r=4$ & $r=2$ \\
 &  & $r=9^k$ & $r\geq 9^*$ &  &  \\
 &  & $r=7^k$ & $r\geq 49$ & $r=7$ &  \\
 & 5 & $r=7^k$ & $r\geq 7$ &  &  \\
 & 6 & $r=3^k$ & $r\geq 9$ & $r=3$ &  \\
 & 7 & $r=3^k$ & $r\geq 3$ &  &  \\
 &  & $r=7^k$ & $r\geq 7$ &  &  \\
 & 8 & $r=2^k$ & $r\geq 4$ & $r=2$ &  \\
$2.L_3(2)$ & 2 & $r=7^{\nu k}$ & $r=49^\dag$ & $r=7$ &  \\
 & 4 & $r=9^k$ & $r\geq 9$ &  &  \\
 &  & $r=7^{\nu k}$ & $r\geq 7$ &  &  \\
 & 6 & $r=9^{\nu k}$ & $r\geq 9$ &  &  \\
 &  & $r=7^{\nu k}$ & $r\geq 7$ &  &  \\
\bottomrule
\end{tabular}
		\caption{Base sizes of groups $G$ with $E(G)/Z(E(G)) \cong L_3(2)$ on $V=V_d(r)$ with $k\geq 1$.}
\label{l32-case-analysis}
\end{table}
\begin{proposition}
	Theorem \ref{linearprop} holds if $E(G)/Z(E(G)) \cong L_4(q)$ with $q\geq 3$.
\end{proposition}
\begin{proof}
By Proposition \ref{alphas}, for $1\neq x \in G$ we have $\alpha(x)\leq 4$ unless $x$ is an involutory graph automorphism where $\alpha(x) \leq 6$. If $G$ has no regular orbit on $V=V_d(r)$ then by Proposition \ref{tools}\ref{crude} and Lemma \ref{graph}\ref{lgraph},
\[
r^{d} \leq 2|\mathrm{P}\Gamma \mathrm{L}_4(q)|r^{\floor{3d/4}}+4q^{9}r^{\floor{5d/6}}
\]
Substituting in $d=d_1(E(G)) = \frac{q^4-q}{q-1}-1$ for $q>3$ and $d=26$ for $q=3$ (cf. Proposition \ref{mindegprop}), we see that $G$ has a regular orbit on all choices of $V$ unless possibly when $q=3$ and $r\leq 11$.
%\subsubsection{$q=3$}
\par%DONEx2 
So suppose $(n,q)=(4,3)$. 
We first determine which modules we must consider. We compute that the number of prime order elements in $\mathrm{Aut}(L_4(3))$ is 2279303, and the number of involutions in $G/F(G)$ is at most 27639. If $G$ has no regular orbit on the $d$-dimensional module $V$, then by Proposition \ref{tools}\ref{crude} we have $r^d \leq 2\times 2279303r^{\floor{3d/4}}+ 2\times 27639 r^{\floor{5d/6}}$, implying that $d \leq 96,48,40,30$ for $r=2,4,5$ and 8 respectively. Note that since $7,11$ do not divide $2|\Gamma\mathrm{L}_4(3)|$, we exclude these possibilities.
	From \cite{HM}, we deduce that the only modules left to consider are those in Table \ref{l43-case-analysis}, and also $V=V_{26}(2)$ or $V_{38}(2)$, where we use Magma to show that there is no regular orbit of $G$ on $V$ if $d=26$, but there is for $d=38$.
In Table \ref{l43-case-analysis}, we also give sharp upper bounds on the eigenspace dimensions of projective prime order elements. If the bound is italicised, then the bound was determined using a construction of the relevant matrix group in GAP,  and otherwise the bounds were determined from the Brauer character tables in \cite[p. 164--171]{modatlas}. We have asterisked an entry in the row for $V=V_{26}(2^k)$ in Table \ref{l43-case-analysis} because involutions in classes 2D, 2E have fixed point spaces of dimension 20 on $V$, while involutions in classes 2B, 2F have fixed point spaces of dimensions 14, 16 respectively.
	%The entry of 24 in the row for $V=V_{38}(2^k)$ is asterisked because involutions in classes 2D and 2E have a fixed point space of dimension 24, while involutions in classes 2B, 2F and 2G have fixed point spaces of dimension 22, 20 and 19 respectively.
	%were obtained from the Brauer characters, constructing the modules in GAP and applying Technique \CEDII{} and Proposition \ref{alphabound}. The asterisked entry has a bound of 31 for involutions in classes 2D and 2E, and a bound of 25 otherwise.
	Therefore, applying \ref{eigsp1} from Proposition \ref{tools}, we find that $G$ has a regular orbit on $V$ for all cases in Table \ref{l43-case-analysis}.
	{
		\small
	\begin{table}[h!]
		\centering \begin{tabular}{@{}llllp{1.7cm}llp{1.7cm}ll@{}}
\toprule
$E(G)$ & $(d,r)$ & 2A & 2C & Other $g_2$ & 3A & 3B & Other $g_3$ & $o(x)=5$ & $o(x)=13$ \\ \midrule
$2.L_4(3)$ & $(40,5)$ & 20 & 26 & 20 & 22 & 16 & 16 & \textit{8} & 4 \\
$L_4(3)$ & $(38,5)$ & 20 & 25 & 23 & 20 & 14 & 14 & \textit{8} & 3 \\
$L_4(3)$ & $(26,5)$ & 16 & -- & 20 & 9 & 14 & 9 & \textit{6} & 2 \\
$L_4(3)$ & $(26,2^k)$ & \textit{16} & -- & \textit{20}$^*$ & 9 & 14 & 9 & 6 & 2 \\
& $(38,2^k)$ & 20 & \textit{25} & \textit{24} & 20 & 14 & 14 & 8 & 3 \\
\midrule
\#elements &  & 2106 & 1080 & 24453 & 1040 & 3120 & 78000 & 303264 & 1866240 \\ \bottomrule
		\end{tabular}
		\caption{Upper bounds on $\emax$ for element of projective prime order in $G$ where  $E(G)/Z(E(G)) \cong L_4(3)$.}
		\label{l43-case-analysis}
	\end{table}
}
\end{proof}
%\subsection{$n\geq 5$}
%DONEx2
\begin{proposition}
	\label{l_n>=5}
	Theorem \ref{linearprop} holds if $E(G)/Z(E(G)) \cong L_n(q)$ with $n\geq 5$.
\end{proposition}
\begin{proof}
We begin by applying Proposition \ref{tools}\ref{crude}, using the bounds $\alpha(x)\leq n$ from Proposition \ref{alphas} and $d_1(G) \geq \frac{q^n-q}{q-1} -1$ from Proposition \ref{mindegprop}. 
Therefore, if $G$ has no regular orbit on $V$ then 
\[
r^d \leq 2\times |\mathrm{P}\Gamma \mathrm{L}_n(q).2|r^{\floor{\frac{n-1}{n}d}}.
\]
We find that $G$ has a regular orbit on $V$ except possibly when $(n,q,d,r)$ lies in Table \ref{n>5lcases}.
\begin{table}[h!]
\label{sl5}

\begin{tabular}{cccc}
\toprule
$n$ & $q$ & $d\leq$ & $r\leq$ \\ \midrule
8   & 2   & 320     & $3$     \\
7   & 2   & 217     & $5$     \\
6   & 2   & 132     & 9       \\
5   & 3   & 195     & $2$     \\
5   & 2   & 75      & 17      \\ \bottomrule
\end{tabular}
\caption{Remaining cases to consider when $n\geq 5$. \label{n>5lcases}}
\end{table}

We consult \cite{HM} to compile the list of modules that require further attention, also making use of Proposition \ref{rep2sl} for $(n,q)=(8,2)$. We summarise our analysis of these remaining cases in Table \ref{n5l-case-analysis}. The bold entries denote those upper bounds obtained from the corresponding Brauer characters, while the italicised entries are computed from an explicit construction in GAP. The remaining entries are obtained by applying Propositions \ref{tools}\ref{alphabound} and \ref{alphas}. Applying Proposition \ref{tools}\ref{eigsp2} to each module in Table \ref{n5l-case-analysis} shows that $G$ has a regular orbit on $V$ in all cases. The only cases not addressed in Table \ref{n5l-case-analysis}, are where $G=L_5(2)$ or $\mathbb{F}_3^\times \times L_5(2)$ and $V=V_{30}(3)$. These cases were investigated computationally using a construction in Magma and there is a regular orbit in both cases.

\begin{table}[h!]
\begin{tabular}{@{}cccccc@{}}
\toprule
$E(G)/Z(E(G))$            & $V$                  & $o(x)=2$    & $o(x)=3$    & $o(x) \mid r$ & Remaining prime order elements \\ \midrule
$L_8(2)$ & $V_{253}(3)$         & 221         & 221         & 221           & \textbf{61}                    \\
$L_7(2)$ & $V_{126}(3)$         & \textbf{94} & \textit{62}         & \textit{62}           & \textbf{42}                    \\
$L_7(2)$ & $V_{126}(5)$         & \textbf{94} & \textbf{62} & 108           &\textbf{42}                 \\
$L_6(2)$ & $V_{61}(r)$, $r=3,9$ & \textbf{46} & \textit{29 }        & \textit{29}            & \textbf{13}                    \\
$L_6(2)$ & $V_{61}(7)$          & \textbf{46} & \textbf{30} & 50            & \textbf{14}                    \\
$L_6(2)$ & $V_{62}(5)$          & \textbf{46} & \textbf{30} & 51            & \textbf{14}                    \\
$L_5(3)$ & $V_{120}(2)$         & 96       & \textbf{66} & 96   & \textbf{66}                    \\
$L_5(2)$ & $V_{30}(5)$          & \textbf{22} & \textbf{14} & \textit{6}    & \textbf{6}                     \\
$L_5(2)$ & $V_{30}(7)$          & \textbf{22} & \textbf{14} & \textit{6}    & \textbf{6}                     \\
$L_5(2)$ & $V_{30}(9)$          & \textbf{22} & \textit{14} & \textit{14}   & \textbf{6}                     \\ \bottomrule
\end{tabular}
\caption{Upper bounds on $\emax$ for elements of projective prime order elements in linear groups, $n\geq 5$.}
\label{n5l-case-analysis}
\end{table}
\end{proof}

In order to complete the proof of Theorem \ref{linearprop}, we need to consider the cases where $E(G)/Z(E(G))$ is also isomorphic to an alternating group. 
The classification of pairs $(G,V)$ where $G$ has a regular orbit on $V$ has been completed in \cite{MR3500766} for the groups $G$ with $\mathrm{soc}(G/F(G))$ an alternating group $Alt_n$ with $n\neq 6$. This also covers the groups $G$ such that $\mathrm{soc}(G/F(G)) \cong Alt_6$, and $G/Z(G)$ is almost simple. The other almost quasisimple groups with $\mathrm{soc}(G/F(G)) \cong Alt_6$ are not considered because the authors of \cite{MR3500766} adopt a more restrictive definition of almost quasisimple than is given here. Note that the authors of \cite{MR3500766} also do not determine the base size for groups where there is no regular orbit.  For completeness, we also consider representations in defining characteristic for these groups.

\begin{proposition}
	\label{alt_grp_prop}
	Theorem \ref{linearprop} holds for almost quasisimple $G$ with $E(G)/Z(E(G)) \cong L_n(q)$ with $(n,q)=(2,4), (2,5), (2,9)$ or $(4,2)$.
\end{proposition}
\begin{proof}
	Recall that we have $L_2(4) \cong L_2(5) \cong Alt_5$, $L_2(9)\cong Alt_6$ and $L_4(2)\cong Alt_8$.
	By \cite[Theorem 1.1]{MR3500766}, the modules (up to algebraic conjugacy) where there may be no regular orbit of $G$ on $V$ are given in Tables \ref{A5_tab} and \ref{A8_tab} for $G$ with $E(G)/Z(E(G)) \cong L_2(4)$, and $L_4(2)$ respectively. Also included in Tables  \ref{A5_tab} and  \ref{A8_tab} are the eigenspace dimensions of elements $x\in G$ of projective prime order. These values are computed from the relevant Brauer characters in \cite{modatlas} for $x$ of projective prime order coprime to $r$, and from a construction of the relevant matrix group in GAP for $x$ of  projective order $r$. Applying Proposition \ref{tools}\ref{eigsp1}, we see that in each case $G$ has a regular orbit on $V=V_d(r_0^k)$, when $k\geq f$ as given in the tables. In the remaining cases where we have not shown that there is a regular orbit using this technique, we construct the modules in GAP \cite{GAP4} and find a minimal base. The groups that were found to have a regular orbit using this method are listed in Table \ref{RO_GAP_alt}. The groups $G$ with $b(G) \geq 2$ appear in Table \ref{noros}.
	
	We now turn our attention to groups $G$ with $E(G)/Z(E(G)) \cong L_2(9)$. There are further modules to check, in addition to those listed in \cite[Table 1]{MR3500766}. This is due to our definition of almost quasisimple requiring $G/F(G)$ to be almost simple, while \cite{MR3500766} instead requires $G/Z(G)$ to be almost simple.
	The modules that we need to check are listed in Table \ref{A6_tab}, along with the eigenspace dimensions of elements of projective prime order, computed in the same manner as in Tables  \ref{A5_tab} and  \ref{A8_tab}. Again applying Proposition \ref{tools}\ref{eigsp1}, we find that in each case $G$ has a regular orbit on $V=V_d(r_0^k)$, when $k\geq c$ as given in the table. As with the other groups, we construct the remaining groups $G$ and modules $V$ in GAP and determine that those listed in Table \ref{RO_GAP_alt} have a regular orbit of $G$ on $V$, while those listed in Table \ref{noros} do not. The only exception to this is $G=  L_2(9).2^{2}$ acting on $\F_3^9$, which is listed in both tables. Here there are four inequivalent modules for $G$. Now, via construction in GAP, we compute that $G$ has a regular orbit on three of these modules, while there is no regular orbit of $G$ on the remaining module. This module is characterised by its corresponding Brauer character values of 3,1 on conjugacy classes 2B and 2D respectively. 
	This completes the proof.

		\begin{table}[h!]
		\centering \begin{tabular}{@{}ll@{}}
			\toprule
			$G$                                                  & $(d,r)$  \\ \midrule
			$L_2(4)$                                             & $(4,4)$  \\
			%$c\times L_2(4)$, $c\in \{1,3\}$                     & $(2,16)$ \\
			$L_2(4)$                                             & $(3,9)$  \\
			$L_2(4)\leq G \leq \F_9^\times \times L_2(4).2$      & $(4,9)$  \\
			$c\circ(2.L_2(4))$, $c\mid 80$                       & $(2,81)$ \\
			$L_2(4)\leq G\leq L_2(4).2$                       & $(6,2)$  \\
			$L_2(4)\leq G\leq  \F_{25}^\times \times L_2(4).2$   & $(3,25)$ \\
			$c \circ (2.L_2(4))$, $c\in \{1,2,4,8\}$                & $(2,25)$ \\
			$2.L_2(4)$                                           & $(4,5)$  \\
			%			$2.L_2(4).c$, $c\in \{1,2\}$  & $(6,3)$  \\ Not in F et al
			\midrule
			$c\circ (2. L_2(9))$, $c\in \{1,2,4,8,16\}$              & $(2,81)$ \\
	%	$c\circ(2. L_2(9).2_2)$, $c\mid 80$             & $(2,81)$ \\
			$L_2(9) \leq G\leq \F_{81}^\times \times L_2(9).2_2$ & $(3,81)$ \\
			$3. L_2(9)$                                      & $(3,16)$ \\
			\bottomrule
		\end{tabular}
		\quad
		\centering \begin{tabular}{@{}ll@{}}
			\toprule
			$G$                                                  & $(d,r)$  \\ \midrule
			$c \times L_2(9)$, $c\in \{1,7\}$                    & $(4,8)$  \\
			$L_2(9) \leq G\leq \F_{16}^\times \times L_2(9).2_1$ & $(4,16)$ \\
			$L_2(9)$, $L_2(9).2_2$, $L_2(9).2_3$                 & $(4,9)$  \\
			$L_2(9) \leq G<\F_3^\times \times L_2(9).2^{2*}$ & $(9,3)$  \\
			$L_2(9) \leq G\leq \F_5^\times \times L_2(9).2^{2}$  & $(8,5)$  \\
			$3.L_2(9)$, $3.L_2(9).2_3$                           & $(3,25)$ \\
			$c \times L_4(2)$, $c\mid 31$                        & $(4,32)$ \\
			$c \times L_4(2)$, $c\mid 63$                        & $(4,64)$ \\
			$L_4(2)$                                             & $(4,16)$ \\
			$c\times L_4(2)$, $c\mid 15$                         & $(6,16)$ \\
			$c\times L_4(2)$, $c\mid 8$                          & $(7,9)$  \\
			$L_4(2)\leq G\leq \F_{27}^\times \times L_4(2).2$    & $(7,27)$ \\
		\bottomrule
		\end{tabular}
		\caption{Groups $G$ where the existence of a regular orbit on $V_d(r)$ was confirmed by construction in GAP \cite{GAP4}. \label{RO_GAP_alt}}
	\end{table}

	\begin{table}[h!]
		\centering \begin{tabular}{@{}llllllll@{}}
			\toprule
			&  & \multicolumn{5}{c}{Eigenspace dimensions} &  \\
			\cmidrule{3-7}
			$E(G)$ & $(d,r_0^k)$ & 2A & 2B & 3A & 5A & 5B & $f$ \\
			\midrule
			\# &  & 15 & 10 & 20 & 12 & 12 &  \\
			\midrule
			$L_2(4)$ & $(2,4^k)$ & 1 & 1 & $(1^2)$ & $(1^2)$ & $(1^2)$ & 3 \\
			& $(3,5^k)$ & $(2,1)$ & $(2,1)$ & $(1^3)$ & 1 & 1 & 3 \\
			& $(3,9^k)$ & $(2,1)$ & -- & 1 & $(1^3)$ & $(1^3)$ & 2 \\
			& $(4,2^k)$ & 2 & 2 & $(2,1^2)$ & $(1^4)$ & $(1^4)$ & 3 \\
			& $(4,3^k)$ & $(2^2)$ & $(3,1)$ & 2 & $(1^4)$ & $(1^4)$ & 3 \\
			$2.L_2(4)$ & $(2,5^k)$ & $(1^2)$ & $(1^2)$ & $(1^2)$ & 1 & 1 & 3 \\
			& $(2,9^k)$ & $(1^2)$ & -- & 1 & $(1^2)$ & $(1^2)$ & 2\\
			& $(4,5^k)$ & $(2^2)$ & $(2^2)$ & $(2,1^2)$ & 1 & 1 & 2 \\
			\bottomrule
		\end{tabular}
		\caption{The eigenspace dimensions of elements of projective prime order in groups $G$ with $E(G)/Z(E(G)) \cong L_2(4)\cong L_2(5) \cong Alt_5$.\label{A5_tab}}
	\end{table}

	\begin{table}[h!]
		\begin{tabular}{@{}llllllllllll@{}}
			\toprule
			&  & \multicolumn{9}{c}{Eigenspace dimensions} &  \\
			\cmidrule{3-11}
			$E(G)$ & $(d,r)$ & 2A & 2B & 3A & 3B & 5A & 7A & 7B & 2C & 2D & $f$ \\
			\midrule
			\# &  & 105 & 210 & 112 & 1120 & 1344 & 2880 & 2880 & 28 & 420 &  \\
			\midrule
			$L_4(2)$ 	& $(4,2^k)$ & 3 & 2 & $(2^2)$ & $(2,1^2)$ & $(1^4)$ & $(1^4)$ & $(1^4)$ & -- & -- & 7\\
			& $(14,2^k)$ & 8 & 8 & $(6,4^2)$ & $(4,5^2)$ & $(2,3^4)$ & $(2^7)$ & $(2^7)$ & 10 & 8 & 2 \\
			$2.L_4(2)$ & $(8,3^k)$ & $(4^2)$ & $(4^2)$ & 4 & 4 & $(2^4)$ & $(2,1^6)$ & $(2,1^6)$ & $(4^2)$ & $(4^2)$ & 2 \\
			& $(8,5^k)$ & $(4^2)$ & $(4^2)$ & $(4^2)$ & $(4,2^2)$ & 2 & $(2,1^6)$ & $(2,1^6)$ & $(4^2)$ & $(4^2)$ & 2 \\
			\bottomrule
		\end{tabular}
		\caption{The eigenspace dimensions of elements of projective prime order in groups $G$ with $E(G)/Z(E(G)) \cong L_4(2) \cong Alt_8$.\label{A8_tab}}
	\end{table}

	\begin{table}[h!]
		\begin{tabular}{@{}lllllllllll@{}}
			\toprule
			&  & \multicolumn{8}{c}{Eigenspace dimensions} &  \\
			\cmidrule{3-10}
			$E(G)$ & $(d,r)$ & 2A & 3A & 3B & 5A & 5B & 2B & 2C & 2D & $f$ \\
			\midrule
			\# & & 45&40&40 &72& 72 & 15&15&36\\
			\midrule
			$L_2(9)$
			& $(3,9^k)$ & (1,2) & 1 & 1 & $(1^3)$ & $(1^3)$ & -- & -- & $(1,2)$ & 3 \\
			& $(4,2^k)$ & 2 & $(2,1^2)$ & $(2^2)$ & $(1^4)$ & $(1^4)$ & 2 & 3 & -- & 5 \\
			& $(4,3^k)$ & $(2^2)$ & 2 & 2 & $(1^4)$ & $(1^4)$ & $(3,1)$ & $(3,1)$ & $(2^2)$ & 4 \\
			& $(5,5^k)$ & $(3,2)$ & $(3,1^2)$ & $(1,2^2)$ & 1 & 1 & (4,1) & (2,3) & -- & 2 \\
			& $(8,4^k)$ & 4 & $(3^2,2)$ & $(3^2,2)$ & $(2^3,1^2)$ & $(2^3,1^2)$ & -- & -- & 5 & 1 \\
			& $(9,3^k)$ & $(5,4)$ & 3 & 3 & $(2^4,1)$ & $(2^4,1)$ & $(6,3)$ & $(6,3)$ & $(5,4)$ & 2 \\
			$2.L_2(9)$ & $(2,9^k)$ & $(1^2)$ & 1 & 1 & $(1^2)$ & $(1^2)$ & -- & -- & $(1^2)$ & 3 \\
			& $(4,5^k)$ & $(2^2)$ & $(2^2)$ & $(2,1^2)$ & 1 & 1 & $(2^2)$ & $(2^2)$ & -- & 2 \\
			
			$3.L_2(9)$ & $(3,4^k)$ & 2 & $(1^3)$ & $(1^3)$ & $(1^3)$ & $(1^3)$ & -- & -- & -- & 3 \\
			& $(3,25^k)$ & $(1,2)$ & $(1^3)$ & $(1^3)$ & 1 & 1 & -- & -- & -- & 2\\
			& $(9,2^k)$ & 5 & $(3^3)$ & $(3^3)$ & $(2^4,1)$ & $(2^4,1)$ & -- & -- & -- & 1 \\
			\bottomrule
		\end{tabular}
		\caption{The eigenspace dimensions of elements of projective prime order in groups $G$ with $E(G)/Z(E(G)) \cong L_2(9) \cong Alt_6$.\label{A6_tab}}
	\end{table}
\end{proof}
\section{Proof of Theorem \ref{mainthm}: Unitary groups}
%DONEx2
\label{u_grps}
In this section we set out to prove the following result.

\begin{theorem}
	\label{umain}
	Suppose $G$ is an almost quasisimple group with $E(G)/Z(E(G)) \cong U_n(q)$ where $n\geq 3$ and $(n,q) \neq (3,2)$. Let $V=V_d(r)$, $(r,q)=1$ be a module for $G$, with absolutely irreducible restriction to $E(G)$. Also suppose that $(r, |G|)>1$. Then one of the following holds.
	\begin{enumerate}
		\item $b(G) = \lceil \log |G|/\log |V| \rceil$.
		\item $b(G)= \lceil \log |G|/\log |V| \rceil+1$ and $(G,V)$ lies in Table \ref{allbad}.
\end{enumerate}
\end{theorem}
%\begin{table}[h!]
%\begin{tabular}{@{}ccc@{}}
%\toprule
%Group $G$ & $V$ & $b(G)$ \\ \midrule
%$U_6(2)$ & $V_{21}(3)$ & 2 \\
%$2\circ U_5(2).2$ & $V_{10}(3)$ & 3 \\
%$c \circ U_5(2).2$, $c\in \{1,2,4,8\}$ & $V_{10}(9)$ & 2 \\
%$3_1.U_4(3)$ & $V_6(4)$ & 3 \\
%$(c\circ 3_1).U_4(3)$, $c \in \{1,5\}$ & $V_6(16)$ & 2 \\
%$U_4(3).D_8$ & $V_{20}(2)$ & 2 \\
%$(c\circ 6_1).U_4(3).2_2$ $c \in \{1, 2,4\}$ & $V_6(25)$ & 2 \\
%$(c\circ 6_1).U_4(3).2_2$ $c \in \{1,2,4,8\}$ & $V_6(49)$ & 2 \\
%$U_3(5).c$,  $c\in \{2,3,6\}$ & $V_{20}(2)$ & 2 \\
%$2\circ U_3(4)$ & $V_{12}(3)$ & 2 \\
%$U_3(3).c$, $c\in \{1,2\}$ & $V_{14}(2)$ & 2 \\
%\bottomrule
%\end{tabular}
%\caption{Pairs $(G,V)$ with $b(G) > \lceil \log|G|/\log|V| \rceil$ for $G$ a unitary group. }
%\label{unitary-bad}
%\end{table}
We exclude the cases where $E(G)/Z(E(G)) \cong U_3(2)$ since it is soluble.
%\subsection{$n=3$}
%donex2
\begin{proposition}
Theorem  \ref{umain} holds when $E(G)/Z(E(G)) \cong U_3(q)$ with $q\geq 3$.
\end{proposition}
\begin{proof}
We begin by determining the values of $d,q,r$ we must consider. By Proposition \ref{alphas}, $\alpha(x)\leq 3$, unless $q=3$ and $x$ is an inner involution with $\alpha(x)=4$. So by Proposition \ref{tools}\ref{crude}, if $G$ has no regular orbit on $V$, and $q\neq 3$, then
\[
r^d \leq 2|\mathrm{P}\Gamma \mathrm{U}_3(q)| r^{\floor{2d/3}}.
\]
Substituting in $d_1(E(G)) = \floor{(q^3-1)/(q+1)}$, we see that this is false for $q\geq 11$ and $r\geq 2$, so there is a regular orbit of $G$ on all irreducible modules $V$ when $q\geq 11$ by Proposition \ref{tools}\ref{crude}. When we replace $|\mathrm{P}\Gamma \mathrm{U}_3(q)|$ by the number of prime order elements in $\mathrm{P}\Gamma \mathrm{U}_3(q)$, we also eliminate $q=8,9$, and the remaining cases to consider are given in Table \ref{u3q-cases}, obtained by examining \cite{HM}.
 Note that for $q=3$, we use $\alpha(x) \leq 4$, as in Proposition \ref{alphas}.
\begin{table}[h!]
\begin{tabular}{@{}llll@{}}
\toprule
$q$ & $d$ & $r$ & $\mathrm{char}(\mathbb{F}_r)$ \\ \midrule
7 & 42 & 2 & 2 \\
5 & 20 & $\leq 4$ & 2,3 \\
 & 21 & 3 & 3 \\
 & 28 & $\leq 3$ & 2,3 \\
4 & 12 & $\leq 13$ & 3,5,13 \\
3 & 3 & $r \leq 729$ & 3\\
 & 6 & $\leq 64$ & 2,3,7 \\
 & 7 & $\leq 49$ & 3,7 \\
 & 14 & $\leq 8$ & 2,7 \\
 & 27 & 3 & 3 \\
 & 32 & 2 & 2 \\ \bottomrule
\end{tabular}
\caption{Remaining cases for $E(G)/Z(E(G)) \cong U_3(q)$. \label{u3q-cases}}
\end{table}
%\subsubsection{$q=7$}
\par%DONEx2 
Suppose $E(G)/Z(E(G)) \cong U_3(7)$.
From Table \ref{u3q-cases}, we are left to consider one 42-dimensional module $V$, and from the Brauer character in \cite{modatlas}, $\emax(g_{2'}) \leq 14$ on the corresponding module. From a construction of $V$ in GAP, we determine that involutions in classes 2A and 2B have eigenspaces of dimension 24 and 21 respectively. The number of prime order elements in $\mathrm{Aut}(U_3(7))$ is 2098571. Therefore, if $G$ has no regular orbit on $V$ then by Proposition \ref{tools}\ref{qsgood},
$r^{42} \leq 2\times 2098571r^{14}+ |2A|r^{24}+|2B|r^{21} $, which is false for $r=2$.
%\subsubsection{$q=5$}
\par%DONEx2 
Now let $E(G)/Z(E(G)) \cong U_3(5)$.
The number of prime order elements in $\mathrm{Aut}(U_3(5))$ is 71849, and we proceed by examining each of the relevant cases in Table \ref{u3q-cases}. We begin with the modules in characteristic 3.
First suppose $d=28$. There are at most 16550 elements of order 3 in $G/F(G)$, and we infer from the Brauer character in \cite{modatlas} that $\emax(g_{3'}) \leq 16$ on both of the 28-dimensional modules. Therefore, if $G$ has no regular orbit on $V$,
	\[
	r^{28} \leq 2(71849-16550)r^{16} + \frac{1}{2}16550r^{\floor{2\times 28/3}},
	\]
which is false for $r=3$.
If instead $V=V_{20}(3)$, then we compute that $\emax(g_2) \leq 12$ and $\emax(g_{\{2,3\}'}) \leq 5$. The number of involutions in $\mathrm{Aut}(U_3(5))$ is 3675. In addition,  we use Magma to compute that $C_V(g_3)\leq 8$. Therefore, if $G$ has no regular orbit on $V$,
	\[
	r^{20} \leq  2\times 71849 r^5 + 3675 (r^{12} + r^8) + (|3A|+|3B|+|3C|) r^{8} 
	\]
	which is false for $r=3$.
	
Now suppose $V$ is a 21-dimensional Weil module. Again, we find from the Brauer character in \cite{modatlas} that $\emax(g_2) \leq 13$ and $\emax(g_{\{2,3\}'}) \leq 5$.
Therefore, if $G$ has no regular orbit on $V$,
\[
r^{21} \leq  2\times 71849 r^5 + 3674 (r^{13} + r^8) + (|3A|+|3C|) r^{10} + |3B| r^{14}
\]
This is also false for $r=3$. We now consider the modules in characteristic 2.
So from Table \ref{u3q-cases}, $d=20$ or 28. Let $V=V_{20}(r)$. From the Brauer character we compute that $\emax(g_3) \leq 8$ and $\emax(g_{\{2,3\}'}) \leq 5$. Therefore, if $G$ has no regular orbit on $V$,
\[
r^{20} \leq 2 \times (71849-3674-16550)r^{5}+ 2\times 16550r^8 + 3674r^{\floor{2\times 20/3}}.
\]
This is false for $r=4$. When $r=2$, we find in GAP that there is a single regular orbit under $U_3(5)$ and there is no regular orbit for any $G$ with $U_3(5)<G \leq U_3(5).S_3$ and $b(G)=2$ here by Lemma \ref{fieldext}. 
Now suppose $V = V_{28}(2)$. In GAP we compute that $\emax(g_2)\leq  16$. Let $H=G/F(G)$. We deduce bounds on eigenspace dimensions from the Brauer character and find that if $G$ has no regular orbit on $V$,
\[
r^{28} \leq 2i_7(H)r^{4} + 2i_5(H)r^8+2|3A|r^{10}+i_2(H)r^{16}
\]
and this is false for $r=2$, so $G$ has a regular orbit on $V$.
%\subsubsection{$q=4$}
\par%DONEx2 
Suppose now that $q=4$.
	From Table \ref{u3q-cases}, we see that we only need to consider 12-dimensional modules for $r\leq 13$. When $r=3$, then $G/F(G) \cong U_3(4)$ since the Brauer character for $U_3(4).2$ contains the irrationality $i\notin \F_3$.  Constructing the module in GAP shows there is a regular orbit under $U_3(4)$, however, there is no regular orbit under $2\times U_3(4)$ and $b(G)=2$ here.
For $r=5,9,13$, we summarise the relevant bounds on eigenspace dimensions of projective prime order elements in Table \ref{u34-case-analysis}. Applying Proposition \ref{tools}\ref{qsgood}, we see that $G$ has a regular orbit on $V$ in all cases.
\begin{table}[h!]
\begin{tabular}{cccccc}
\toprule
 & \multicolumn{5}{c}{Bounds on eigenspace dimensions} \\ 
\cmidrule{2-6}
$r$ & $2A$ & $2B$ & $o(x)=3$ & $o(x)=5$ & $o(x)=13$ \\ 
\midrule
5 & 8 & 6 & 4 & 5 & 1 \\ 
9 & 8 & 6 & 4 & 4 & 1 \\ 
13 & 8 & 6 & 4 & 4 & 1 \\ 
\bottomrule
\end{tabular} 
\caption{Upper bounds on $\emax$ for projective prime order elements acting on $V_{12}(r)$. \label{u34-case-analysis}}
\end{table}
%\subsubsection{$q=3$}
\par%DONEx2 
Finally, suppose $q=3$. Since $U_3(3) \cong G_2(2)'$, we will consider representations in characteristics 2, 3, 7.
The number of prime order elements in $U_3(3).2$ is 2771. We determine in GAP that if $x$ is of order 7 or in class 3A, then $\alpha(x)=2$. Otherwise, $\alpha(x) \leq 3$ unless $x$ is an inner involution and $\alpha(x) \leq 4$. Therefore, if $G$ has no regular orbit on an absolutely irreducible  $\mathbb{F}_rE(G)$-module $V$,
\begin{align*}
r^d & \leq 2(i_7(G) + |3A|)r^{\floor{d/2}}+ 2(|3B| + |2B|)r^{\floor{2d/3}}+ 2 |2A| r^{\floor{3d/4}}\\
& \leq 2(1728+ 672)r^{\floor{d/2}}+ 2(56+ 252)r^{\floor{2d/3}}+ 2 \times 63 r^{\floor{3d/4}}
\end{align*}
	Substituting in the cases listed in Table \ref{u3q-cases}, we see that we only now need to consider $d=3,6,7,14$ for $r=9,81,729$, $r\leq 16$, $r=3,7,9$ and $r=2,4$ respectively. We resolve all of the remaining cases by constructing the modules in GAP and computing the base size.We summarise our findings in Table \ref{u33_tab}.

	\begin{table}[h!]
		\centering
		\begin{tabular}{@{}lllll@{}}
			\toprule
			$d$ & Field & $b(G)=1$ & $b(G)=2$ & $b(G)=3$ \\ \midrule
			3 & $r=9^k$ & $r\geq 729$ & $r=81^*$ & $r=9^*$ \\
			6 & $r=2^k$ & $r\geq 8$ & $r=4$ & $r=2$ \\
			& $r=9^k$ & $r\geq 9$ &  &  \\
			& $r=7^k$ & $r\geq 49$ & $r=7^*$ &  \\
			7 & $r=3^k$ & $r\geq 9$ & $r=3$ &  \\
			& $r=7^k$ & $r\geq 7$ &  &  \\
			14 & $r=2^k$ & $r\geq 4$ & $r=2$ &  \\ \bottomrule
		\end{tabular}
	\caption{The base sizes of some irreducible modules for $G$ with $G/F(G) \cong U_3(3)$. \label{u33_tab}}
	\end{table}
The entries of Table \ref{u33_tab} with asterisks require some further explanation.
%	If $V = V_{14}(r)$ for $r=2,4$, we find there is no regular orbit on $V$ when $r=2$ and $b(G)=2$ here. If $r=4$, then there is a regular orbit.
%	If $V=V_{7}(r)$ for $r=3,7,9$, then there is a regular orbit for $r=7,9$. If $r=3$ then $|V|<|G|$ and $b(G) = 2$ by Lemma \ref{fieldext}. Now let $V=V_6(r)$ for $r\leq 16$. Note that in characteristic $3$, $r=9$ due to an $i$ irrationality in the corresponding character.
%	When $r=2$, $|V|^2< |G|$ so $b(G)\geq 3$, and we compute that indeed $b(G) = 3$. If $r=4$, we have $|V|< |G|$ and $b(G) = 2$. 
%	If instead $r=8,9$ or 16, we compute that $G$ has a regular orbit on $V$. 
	When $(d,r)=(6,7)$, we compute that $\langle \rho \rangle \times U_3(3)$, $\rho \in \mathbb{F}_7^\times$  has a regular orbit on $V$, except when $\langle \rho \rangle \cong \mathbb{F}_7^\times$. In this case, $b(G)=2$ by Lemma \ref{fieldext}. Moreover, $U_3(3).2$ and $3\times U_3(3).2$ have a regular orbit on $V$, while  $2\times U_3(3).2$ and $6\times U_3(3).2$ do not. Again, these have $b(G)=2$  by Lemma \ref{fieldext}.
	
	%Finally, let $V=V_3(r)$ for $r=9,81,729$. 
	If instead $(d,r)=(3,9)$ then $|V|<|G|$ and $b(G) =2$ if $G=U_3(3)$, and $b(G) =3$ if $c\times U_3(3)$ with $c \in\{2,4,8\}$.  If now $(d,r)=(3,81)$, then $G$ has a regular orbit on $V$ if $G = U_3(3)$ or $5\times U_3(3)$, and $b(G) =2$ if $G=c\times U_3(3)$ for even $c$ dividing 80. 
	
%If $V = V_{14}(r)$ for $r=2,4$, we find there is no regular orbit on $V$ when $r=2$ and $b(G)=2$ here. If $r=4$, then there is a regular orbit.
%If $V=V_{7}(r)$ for $r=3,7,9$, then there is a regular orbit for $r=7,9$. If $r=3$ then $|V|<|G|$ and $b(G) = 2$ by Lemma \ref{fieldext}. Now let $V=V_6(r)$ for $r\leq 16$. Note that in characteristic $3$, $r=9$ due to an $i$ irrationality in the corresponding character.
% When $r=2$, $|V|^2< |G|$ so $b(G)\geq 3$, and we compute that indeed $b(G) = 3$. If $r=4$, we have $|V|< |G|$ and $b(G) = 2$. 
%If instead $r=8,9$ or 16, we compute that $G$ has a regular orbit on $V$. 
%When $r=7$, we compute that $\langle \rho \rangle. U_3(3)$, $\rho \in \mathbb{F}_7^\times$  has a regular orbit on $V$, except when $\langle \rho \rangle \cong \mathbb{F}_7^\times$. In this case, $b(G)=2$ by Lemma \ref{fieldext}. Moreover, $U_3(3).2$ and $3\times U_3(3).2$ have a regular orbit on $V$, while  $2\times U_3(3).2$ and $6\times U_3(3).2$ do not. Again, these have $b(G)=2$  by Lemma \ref{fieldext}.
%Finally, let $V=V_3(r)$ for $r=9,81,729$. If $r=9$ then $|V|<|G|$ and $b(G) =2$ if $G=U_3(3)$, and $b(G) =3$ if $c\times U_3(3)$ with $c \in\{2,4,8\}$.  If instead $r=81$, then $G$ has a regular orbit on $V$ if $G = U_3(3)$ or $5\times U_3(3)$, and $b(G) =2$ if $G=c\times U_3(3)$ for even $c$ dividing 80. Finally, if $r=729$, then $G$ always has a regular orbit on $V$.
\end{proof}
%\subsection{$n=4$}
\begin{proposition}
\label{u_4_q}
Theorem  \ref{umain} holds when $E(G)/Z(E(G)) \cong U_4(q)$ with $q\geq 2$.
\end{proposition}
\begin{proof}
Applying Proposition \ref{tools} \ref{crude} with $d_1(G) \geq (q^4-1)/(q+1)$ from Proposition \ref{mindegprop} and $\alpha(G) \leq 6$ from Proposition \ref{alphas}, we deduce that $G$ has a regular orbit on all $V$ if $q \geq 7$. For the remainder we may assume that $q\leq 5$.
%\subsubsection{$q=5$}
\par%DONEx2 
Suppose $q=5$. The number of prime order elements in $\mathrm{Aut}(U_4(5))$ is 1572224649, while the number of involutions is 2391775. By Proposition \ref{alphas}, $\alpha(x) \leq 4$ for all non-trivial $x \in G/F(G)$, unless $x$ is an involutory graph automorphism with $\alpha(x) \leq 6$. 
Therefore, if $G$ has no regular orbit on $V$,
	\[
	r^d \leq 2\times 1572224649r^{\floor{\frac{3d}{4}}} + 2\times 1971900r^{\floor{\frac{5d}{6}}},
	\]
so by examining \cite{HM} we see that $G$ has a regular orbit on $V$ unless possibly $r=2$ and $V=V_{104}(2)$. 
We compute in GAP that this module is a composition factor of the reduction modulo 2 of the 105-dimensional irreducible module of $G$ over $\mathbb{Z}$, so we can deduce the Brauer character of $V$ from the ordinary character. From the Brauer character we find that $\emax(g_{2'}) \leq 42$. We also find that the number of graph automorphisms in $G/F(G)$ is at most 3150. Therefore, if $G$ has no regular orbit on $V$ then by Proposition \ref{tools}\ref{qsgood},
\[
r^{104} \leq2\times 1572224649r^{42} +  2391775r^{\floor{\frac{3d}{4}}}+ 3150r^{\floor{\frac{5d}{6}}}.
\]
This inequality is false for $r=2$, so $G$ has a regular orbit on $V$.
%\subsubsection{$q=4$}
\par%DONEx2 
Now suppose $q=4$.
 The number of prime order elements in $U_4(4).4$ is 156437999, and 335855 of these are involutions. If $G$ has no regular orbit on an irreducible module $V$,
\[
r^d \leq 2\times 156437999r^{\floor{\frac{3d}{4}}}+2\times 335855r^{\floor{\frac{5d}{6}}}
\]
This is false for all $r,d$ except $r=3$ and $d \leq 72$. By \cite{HM}, this leaves a 52-dimensional module $V$ to consider. This module can be constructed by reducing a 52-dimensional module for $G$ over $\mathbb{Z}$, so we can deduce the Brauer character corresponding to $V$. The number of elements of order 3 in $G$ is 1214720, and examining the Brauer character we see that $\emax(g_{3'}) \leq 28$, except if the element $x$ is in class 2A or 2C, where $\dim C_V(x) \leq 32$ or 36 respectively. Therefore, if $G$ has no regular orbit on $V$,
	\[
	r^{52} \leq 2\times 156437999r^{28}+ 2\times |2A|r^{32}+2\times |2C|r^{36}+1214720r^{\floor{\frac{3\times 52}{4}}}.
	\]
This is false for $r=3$, so $G$ has a regular orbit on $V$.
%\subsubsection{$q=3$}
\par
Now suppose $q=3$.
 We summarise the numbers of prime order elements in $H=U_4(3).D_8$ computed using GAP in Table \ref{u43-counts}.
\begin{table}[h!]
\begin{tabular}{cc}
\toprule
$r_0$ & $i_{r_0}(H)$ \\ 
\midrule
2 (graph auts) & 20664 \\ 
2 (not graph auts) & 7911 \\ 
3 & 47600 \\ 
5 & 653184 \\ 
7 & 933120 \\ 
\midrule
Total & 1662479 \\ 
\bottomrule
\end{tabular} 
\caption{Elements of prime order in $U_4(3).D_8$. \label{u43-counts}}
\end{table}
%Let $G$ be an almost quasisimple group with $\mathrm{soc}(G/F(G)) \cong U_4(3)$. Then $d\geq 6$ by Proposition \ref{mindegprop}, and the number of prime order elements in $G/F(G)$ is at most 1662479, with at most of 20664 of these being graph automorphisms. 
Therefore, if  $G$ has no regular orbit on $V$,
\[
r^d \leq 2(1662479-20664)r^{\floor{\frac{3d}{4}}} + 2\times 20664r^{\floor{\frac{5d}{6}}}
\]
This is false except for the cases in Table \ref{u43-cases}, which is derived from examining \cite{HM} and also takes Brauer character irrationalities into account. For example,  when $r=2$, the inequality is false for all $d\geq 93$.  An examination of \cite{HM} shows that the Brauer character of every absolutely irreducible module for $E(G)$ in characteristic 2 of dimension $36\leq d\leq 93$ contains a $z_3$ irrationality, so is not realised over $\F_2$, so we do not have to consider any such modules over $\F_2$.
\begin{table}[h!]
\begin{tabular}{ccccc}
\toprule
$d$ & $E(G)$ & Characteristic & Irrationalities & $r\leq $ \\ 
\midrule
6 & $3_1.U_4(3)$ & 2 & $z_3$ & $2^{15}$ \\ 
6 & $6_1.U_4(3)$ & $\neq 2,3$ & $z_3$ & 41407  \\ 
15 & $3_1.U_4(3)$ & $\neq 3$ & $z_3$ & 32 \\ 
20 & $U_4(3)$ & 2 &  & 16 \\ 
20 & $2.U_4(3)$ & $\neq 2,3$ &  & 7 \\ %Smallest power of 5,7 under 21
20 & $4.U_4(3)$ & $\neq 2,3$ & $i_1$ & 5 \\ 
21 & $U_4(3)$ & $\neq 2,3$ &  & 7 \\ 
21 & $3_1.U_4(3)$ & $\neq 2,3$ & $z_3$ & 7 \\ 
34 & $U_4(3)$ & 2 &  & 4 \\ 
35 & $U_4(3)$ & $\neq 2,3$ &  & 5 \\ 
36 & $3_2.U_4(3)$ & $\neq 3$ & $z_3$ & 4\\
\bottomrule
\end{tabular} 
\caption{Remaining modules to consider for $G$ with $\mathrm{soc}(G/F(G)) \cong U_4(3)$.\label{u43-cases}}
\end{table}
We summarise our analysis of these remaining cases in Tables \ref{u43-ros} and \ref{u43-case-analysis}. The italicised entries in Table \ref{u43-case-analysis} were obtained by a construction in GAP and the underlined entries using Magma. The bolded entries were found by constructing the group $m.U_4(3).2_2$ as a subgroup of $m.U_6(2)$, and then computing eigenspace dimensions. The remaining entries were obtained from the Brauer characters and Proposition \ref{tools}\ref{alphabound}.
An application of  Proposition \ref{tools}\ref{eigsp1} implies that $G$ has a regular orbit on $V$ in all of the cases listed in Table \ref{u43-case-analysis}, except for those listed in Table \ref{u43-ros}, where the value of $b(G)$ is computed explicitly in Magma \cite{Magma}.

\begin{table}[h]
\begin{tabular}{cccc}
\toprule
$G$              & $(d,r)$   & $\lceil \frac{\log|G|}{\log |V|} \rceil$ & $b(G)$                \\
\midrule
$U_4(3)\leq G\leq U_4(3).D_8$     & $(20,2)$  & 2                                        & 2            \\
$3_1.U_4(3).2_2$ & $(15,4)$  & 1                                        & 1      \\
$c\circ(6_1.U_4(3))$, $c\mid 24$ & $(6,25)$  & 1                                        & 2        \\
$c\circ(6_1.U_4(3).2_2)$, $c\mid 24$ & $(6,25)$  & 1                                        & 2        \\
$6_1.U_4(3).2_2$ & $(6,7)$   & 2                                        & 2            \\
$c\circ(6_1.U_4(3).2_2)$, $c\mid 48$ & $(6,49)$  & 1                                        & 2             \\
$3_1.U_4(3)$     & $(6,4)$   & 2                                        & 3      \\
$3_1.U_4(3).2_2$ & $(6,4)$   & 3                                        & 3       \\
$3_1.U_4(3)$     & $(6,16)$  & 1                                        & 2      \\
$15\circ(3_1.U_4(3))$ & $(6,16)$  & 2                                        & 2       \\
$c\circ(3_1.U_4(3).2_2)$, $c\mid 15$ & $(6,16)$  & 2                                        & 2       \\
$3_1.U_4(3).2_2$ & $(6,64)$  & 1                                        & 1          \\
$3_1.U_4(3).2_2$ & $(6,256)$ & 1                                        & 1        \\
\bottomrule   
\end{tabular}
\caption{Base sizes of some modules for $G$ with $\mathrm{soc}(G/F(G)) \cong U_4(3)$. \label{u43-ros}}
\end{table}

	{\footnotesize
		\begin{table}[h!]
		\centering
	\begin{tabular}{@{}lccccccc@{}}
		\toprule
		$E(G)$ & $(d,r)$ & 2A & 2B & 2C & 2D & 2E & 2F \\ \midrule
		$U_4(3)$ & $(35,5)$ & $(19,16)$ & $(21,14)$ & $(20,15)$ & $(25,10)$ & $(18,17)$ & -- \\
		$U_4(3)$ & $(34,2^k)$ & \textit{18} & {\ul 20} & {\ul 19} & {\ul 20} & {\ul 19} & -- \\
		$3_1.U_4(3)$ & $(21,7)$ & $(13,8)$ & -- & -- & $(16,5)$ & $(12,9)$ & -- \\
		$U_4(3)$ & $(21,7)$ & $(13,8)$ & $(14,7)$ & $(11,10)$ & $(15,6)$ & $(11,10)$ & $(11,10)$ \\
		& $(21,5)$ & $(13,8)$ & $(14,7)$ & $(11,10)$ & $(15,6)$ & $(11,10)$ & $(11,10)$ \\
		$4.U_4(3)$ & $(20,5)$ & $(10^2)$ & $(13,7)$ & $(10^2)$ & -- & -- & -- \\
		$2.U_4(3)$ & $(20,7)$ & $(12,8)$ & $(14,6)$ & $(10^2)$ & $(10^2)$ & $(10^2)$ & $(10^2)$ \\
		& $(20,5)$ & $(12,8)$ & $(14,6)$ & $(10^2)$ & $(10^2)$ & $(10^2)$ & $(10^2)$ \\
		$U_4(3)$ & $(20,2^k)$ & {\ul 12} & {\ul 14} & {\ul 10} & {\ul 14} & {\ul 10} & {\ul 11} \\
		$3_1.U_4(3)$ & $(15,4^k)$ & \textbf{9} & -- & -- & \textbf{11} & \textbf{9} & -- \\
		& $(15,25)$ & $(8,7)$ & -- & -- & $(10,5)$ & $(9,6)$ & -- \\
		& $(15,7)$ & $(8,7)$ & -- & -- & $(10,5)$ & $(9,6)$ & -- \\
		$6_1.U_4(3)$ & $(6,25^k)$ & $(4,2)$ & -- & -- & $(5,1)$ & $(3^2)$ & -- \\
		& $(6,7^k)$ & $(4,2)$ & -- & -- & $(5,1)$ & $(3^2)$ & -- \\
		$3_1.U_4(3)$ & $(6,4^k)$ & \textbf{4} & -- & -- & \textbf{5} & \textbf{3} & -- \\
		\midrule
		\#elements &  & 2835 & 540 & 4536 & 126 & 5670 & 4536 \\ \bottomrule
	\end{tabular}
			\caption{Eigenspace dimensions of elements of projective prime order, Part I.  \label{u43-case-analysis}}
	\end{table}
\begin{table}[h!]
	\ContinuedFloat
	\centering
		\begin{tabular}{@{}lccc>{\centering\arraybackslash}p{2cm}cc@{}}
			\toprule
			$E(G)$ & $(d,r)$ & 3A & 3B & 3C/3D & 5A & 7A/7B \\ \midrule
			$U_4(3)$ & $(35,5)$ & $(17,9^2)$ & $(17,9^2)$ & $(12^2,11)$ & \textit{7} & $(5^7)$ \\
			$U_4(3)$ & $(34,2^k)$ & $(16,9^2)$ & $(16,9^2)$ & $(12^2,10)$ & $(7^4,6)$ & $(5^6,4)$ \\
			$3_1.U_4(3)$ & $(21,7)$ & $(9,6^2)$ & $(11,5^2)$ & $(7^3)$ & $(5,4^4)$ & \textbf{3} \\
			$U_4(3)$ & $(21,7)$ & $(9^2,3)$ & $(9,6^2)$ & $(9,6^2)$ & $(5,4^4)$ & \textit{3} \\
			& $(21,5)$ & $(9^2,3)$ & $(9,6^2)$ & $(9,6^2)$ & \textit{5} & $(3^7)$ \\
			$4.U_4(3)$ & $(20,5)$ & $(9^2,2)$ & $(8,6^2)$ & $(8,6^2)$ & 10 & $(3^6,2)$ \\
			$2.U_4(3)$ & $(20,7)$ & $(9^2,2)$ & $(8,6^2)$ & $(8,6^2)$ & $(4^5)$ & \textbf{3} \\
			& $(20,5)$ & $(9^2,2)$ & $(8,6^2)$ & $(8,6^2)$ & 10 & $(3^6,2)$ \\
			$U_4(3)$ & $(20,2^k)$ & $(9^2,2)$ & $(8,6^2)$ & $(8,6^2)$ & $(4^5)$ & $(3^6,2)$ \\
			$3_1.U_4(3)$ & $(15,4^k)$ & $(9,3^2)$ & $(7,4^2)$ & $(5^3)$ & $(3^5)$ & $(3,2^6)$ \\
			& $(15,25)$ & $(9,3^2)$ & $(7,4^2)$ & $(5^3)$ & \textbf{3} & $(3,2^6)$ \\
			& $(15,7)$ & $(9,3^2)$ & $(7,4^2)$ & $(5^3)$ & $(3^5)$ & \textbf{3} \\
			$6_1.U_4(3)$ & $(6,25^k)$ & $(3^2)$ & $(4,1^2)$ & $(2^3)$ & \textbf{2} & $(1^6)$ \\
			& $(6,7^k)$ & $(3^2)$ & $(4,1^2)$ & $(2^3)$ & $(2,1^4)$ & \textbf{1} \\
			$3_1.U_4(3)$ & $(6,4^k)$ & $(3^2)$ & $(4,1^2)$ & $(2^3)$ & $(2,1^4)$ & $(1^6)$ \\
			\midrule
			\#elements &  & 560 & 3360 & 3360/40320 (2 classes) & 653184 & 466560 (each) \\ \bottomrule
		\end{tabular}
			\caption{Eigenspace dimensions of elements of projective prime order, Part II}
		\end{table}
	}	
	Finally, suppose $q=2$. Recall that $ U_4(2) \cong \mathrm{PSp}_4(3)$, so we investigate modules in all characteristics $r_0$ such that $r_0 \mid |U_4(2)|$.
	We summarise some information about elements of prime order in $H= U_4(2) .2$ in Table \ref{s43-counts}. This information was obtained using GAP.
	\begin{table}[h!]
		\centering \begin{tabular}{ccc}
			\toprule
			Prime $r_0$ &  $i_{r_0}(H)$ & $\alpha(x) \leq$ \\ 
			\midrule
			2 & 891 & 3 (2B/2D), 5 (2A), 6 (2D) \\ 
			3 & 800 & 4 (3A), 3 (o/w) \\ 
			5 & 5184& 2 \\ 
			\bottomrule
		\end{tabular} 
		\caption{Elements of prime order in $H= U_4(2) .2$. \label{s43-counts}}
	\end{table}
	
	If $G$ has no regular orbit on $V$,
	\[
	\frac{1}{2} r^d \leq 5184r^{\floor{d/2}}+80r^{\floor{3d/4}}+ (240+480)r^{\floor{2d/3}}+ (270+540)r^{\floor{2d/3}}+45r^{\floor{4d/5}}+ 36r^{\floor{5d/6}}.
	\]
	
	%\begin{table}[h!]
	%\centering \begin{tabular}{cccc}
	%\toprule
	%Dimension & QS group & Irrationalities & Max $r$ to check \\ 
	%\midrule
	%4 & $2.\mathrm{PSp}_4(3)$ & $z_3$ & 125 \\ 
	%5 & $\mathrm{PSp}_4(3)$ & $z_3$ &125 \\ 
	%6 & $\mathrm{PSp}_4(3)$ & -- & 25 \\ 
	%15 & $\mathrm{PSp}_4(3)$ & -- & 5 \\ 
	%\bottomrule
	%\end{tabular} 
	%\caption{Remaining modules to consider for $G$ with $E(G)/Z(E(G)) \cong \mathrm{PSp}_4(3)$ \label{s43-cases}}
	%\end{table}
	
	This is false except for the modules set out in Table \ref{s43-case-analysis}, which were obtained from the Brauer character tables \cite[p. 60--62]{modatlas}. In Table \ref{s43-case-analysis} we also present our analysis of these cases. We give the eigenspace dimensions of elements of projective prime order, obtained by
	examining the Brauer characters of the modules and constructing the modules themselves using GAP.
	{\renewcommand{\arraystretch}{1.5}	\begin{table}[h!]
			\centering \begin{tabular}{ccccccccc}
				\toprule
				& \multicolumn{8}{c}{Class in $ U_4(2) .2$} \\ \cline{2-9} 
				& 2A & 2B & 3A/3B & 3C & 3D & 5A & 2C & 2D \\
				\midrule
				Size & 45 & 270 & 80 & 240 & 480 & 5184 & 36 & 540 \\
				\midrule
				$V_4(4^k)$ & 3 & 2 & $(3,1)$ & $(2^2)$ & $(2,1^2)$ & $(1^4)$ & -- & -- \\
				$V_6(2^k)$ & 4 & 4 & $(3^2)$ & $(4,1^2)$ & $(2^3)$ & $(2,1^4)$ & 5 & 3 \\
				$V_{14}(2^k)$ & 8 & 8 & $(8,3^2)$ & $(6,4^2)$ & $(5^2,4)$ & $(3^4,2)$ & 10 & 8 \\
				$V_5(3^k)$ & $(4,1)$ & $(3,2)$ & 3 & 3 & 3 & $(1^5)$ & $(4,1)$ & $(3,2)$ \\
				$V_{10}(3^k)$ & $(6,4)$ & $(6,4)$ & 6 & 4 & 4 & $(2^5)$ & $(6,4)$ & $(6,4)$ \\
				$V_{14}(3^k)$ & $(10,4)$ & $(8,6)$ & 6 & 6 & 6 & $(3^4,2)$ & $(10,4)$ & $(8,6)$ \\
				$V_{4}(3^k)$ & $(2^2)$ & $(2^2)$ & 3 & 2 & 2 & $(1^4)$ & $(2^2)$ & $(2^2)$ \\
				$V_{16}(3^k)$ & $(8^2)$ & $(8^2)$ & 8 & 6 & 6 & $(4,3^4)$ & $(8^2)$ & $(8^2)$ \\
				$V_4(25^k)$ & $(2^2)$ & $(2^2)$ & $(1,3)$ & $(2^2)$ & $(2,1^2)$ & 1 & -- & -- \\
				$V_5(25^k)$ & (4,1) & $(3,2)$ & $(2,3)$ & $(1,2^2)$ & $(3,1^2)$ & 1 & -- & -- \\
				$V_6(5^k)$ & $(4,2)$ & $(4,2)$ & $(3^2)$ & $(4,1^2)$ & $(2^3)$ & 2 & $(5,1)$ & $(3^2)$ \\
				$V_{10}(25^k)$ & $(6,4)$ & $(6,4)$ & $(1,6,3)$ & $(4,3^2)$ & $(4,3^2)$ & 2 & -- & -- \\
				$V_{15a}(5^k)$ & $(8,7)$ & $(8,7)$ & $(9,3^2)$ & $(7,4^2)$ & $(5^3)$ & 3 & $(10,5)$ & $(9,6)$ \\
				$V_{15b}(5^k)$ & (11,4) & $(9,6)$ & $(3,6^2)$ & $(5^3)$ & $(7,4^2)$ & 3 & $(10,5)$ & $(8,7)$ \\ \hline
			\end{tabular}
			\caption{Eigenspace dimensions of projective prime order elements in $G$. \label{s43-case-analysis}}
		\end{table}
	}

	Applying Proposition \ref{tools}\ref{eigsp1}, or alternatively constructing the module with GAP and then computing the base size, we see that $G$ has a regular orbit on $V$ unless $(G,V)$ appear in Table \ref{allbad}, where the base size is computed explicitly.
	%if $d=10,15$, or if $G/F(G) \cong \mathrm{PSp}(4,3)$ and $(d,r)=(6,25)$. Let $w$ be a generator of $\mathbb{F}_{25}^\times$.
	%A representative of a regular orbit of $V = V_6(25)$ for $G = \mathrm{PSp}(4,3).2$ is
	% $w^3e_1+w^{17}e_2+w^{19}e_4+3e_5+w^{23}e_6$.
	% When $(d,r) = (6,5)$, $|V|<|G|$ and a base of size 2 is given by $\{e_3, e_6\}$.
	%When $V=V_5(25)$, we see that $G$ has a regular orbit on $V$ with representative $w^9e_1+w^5e_2+w^{16}e_3+w^{14}e_4$ .  Finally, if $V = V_4(25)$ then there is a regular orbit of $c\times \mathrm{PSp}_4(3)$ for $c \in \{2,4,8\}$ with representative $e_2+e_3+Z(5)^3e_4$, while if $c\in \{6,12,24\}$, there is no regular orbit and $b(G)=2$.
\end{proof}

Now assume that $n\geq 5$. 
Applying Proposition \ref{tools}\ref{crude}, with $d_1(E(G)) \geq \lfloor (q^n-1)/(q+1)\rfloor$ and  $\alpha(G) \leq n$, we see that either $G$ has a regular orbit on $V$ or $G$ is a group listed in Table \ref{n>5ugrps}. 
	\begin{table}[h!]
		\centering \begin{tabular}{ccccccccc}
			\toprule 
			$n$ & 5 & 6 & 7 & 8 & 9 & 10 & 11 \\ 
			\midrule
			$q$ & $2,3$ & 2,3 & 2,3 & 2 & 2 & 2 & 2  \\ 
			\bottomrule 
		\end{tabular} 
		\caption{Remaining groups to consider when $n\geq 5$. \label{n>5ugrps}}
	\end{table}
	
	We now consider each of these cases in turn.
%%\subsection{$n\geq 5$}
%\begin{proposition}
%	\label{n>=5_unitary}
%Theorem \ref{umain} holds for $n\geq 5$.
%\end{proposition}
%\begin{proof}
%Suppose that $n\geq 5$. We begin by applying Proposition \ref{tools}\ref{crude} with $d_1(G) \geq \lfloor (q^n-1)/(q+1)\rfloor$ and  $\alpha(G) \leq n$. It follows that $G$ has a regular orbit on $V$ except possibly if $G$ is a group listed in Table \ref{n>5ugrps}.
%\begin{table}[h!]
%\begin{tabular}{ccccccccc}
%\toprule 
%$n$ & 5 & 6 & 7 & 8 & 9 & 10 & 11 \\ 
%\midrule
%$q$ & $2,3$ & 2,3 & 2,3 & 2 & 2 & 2 & 2  \\ 
%\bottomrule 
%\end{tabular} 
%\caption{Remaining groups to consider when $n\geq 5$. \label{n>5ugrps}}
%\end{table}

%\subsubsection{$(n,q)=(11,2)$}
\par%DONEx2 
	\begin{proposition}
	Theorem \ref{umain} holds when $E(G)/Z(E(G)) \cong U_{11}(2)$.
	\end{proposition}
\begin{proof}
By Propositions \ref{tools}\ref{crude} and \ref{d2unitary}, we only need consider the 682-dimensional Weil module $V$ for $r=3$.
In GAP we compute that $i_2(H)< 3^{42.02}$  and $i_3(H)<3^{53.58}$ for $H=\mathrm{Aut}(U_{11}(2))$. We will bound the eigenspace dimensions of elements of projective prime order at least 5 on the $2048$-dimensional Weil module $\overline{V}$ for $G$ using Propositions \ref{weilunitary} and \ref{intchar}. Since $V$ is a constituent of $\overline{V}$, this also gives upper bounds for eigenspace dimensions in $V$. Let $\chi$ be the Brauer character corresponding to $\overline{V}$ and let $W$ be the natural module for $U_{11}(2)$. Now, let $x\in G/F(G)$, and suppose $o(x)=5$. By Proposition \ref{itoeig}, $\dim C_W(x) \in \{3,7\}$ and therefore by Proposition \ref{weilunitary}, $\chi(x) = 8$ or 128, which, in the notation of Proposition \ref{intchar}, corresponds to eigenvalues $(\Omega^{408},1^8)$ or $(\Omega^{384},1^{128})$. Therefore, $\dim C_V(x) \leq \dim C_{\overline{V}}(x)\leq 512$. We summarise the analogous calculations for other prime order elements $x$ in Table \ref{u112-eigsp}.

\begin{table}[]
\begin{tabular}{@{}cccc@{}}
\toprule
$o(x)$ & $\dim C_W(x)$ & $\chi(x)$ & Eigenvalues\\ \midrule
7 & 5 & 32 & $(\Omega^{288},1^{32})$ \\
11 & 1 & 2 & $(\Omega^{186},1^2)$ \\
 & 6 & -64 & $(\overline{\Omega}^{64}, \Omega^{128})$ \\
17 & 3 & 8 & $(\Omega^{120},1^8)$ \\
19 & 2 & -4 & $(\overline{\Omega}^{4}, \Omega^{104})$ \\
31 & 1 & 2 & $(\Omega^{66},1^2)$ \\
43 & 4 & -16 & $(\overline{\Omega}^{16}, \Omega^{32})$ \\
683 & 0 & -1 & $(\overline{\Omega}, \Omega^{3})$ \\ \bottomrule
\end{tabular}
\caption{Eigenvalues of some prime order elements of $U_{11}(2)$ on $\overline{V} = V_{2048}(3)$.}
\label{u112-eigsp}
\end{table}

From Table \ref{u112-eigsp}, we therefore deduce that $\emax(g_{\{2,3\}'}) \leq 512$ on $\overline{V}$ and therefore $V$. Therefore, if $G$ has no regular orbit on $V$,
\[
3^{682} \leq 2(3^{42.02}+3^{53.58})3^{\floor{10\times 682/11}}+ 2|U_{11}(2).2|\times 3^{512}.
\]
This is false for $r=3$, so $G$ has a regular orbit on $V$.
%\subsubsection{$(n,q)=(10,2)$}
%DONEx2
\end{proof}

	\begin{proposition}
	Theorem \ref{umain} holds when $E(G)/Z(E(G)) \cong U_{10}(2)$.
	\label{prop_u102}
	\end{proposition}
\begin{proof}
By Propositions \ref{tools}\ref{crude} and \ref{d2unitary}, we must consider the 341- and 342-degree Weil modules for $r=3,5,7$.
We begin with $r=7$.
%Using \cite[Table B.4]{bg}, we deduce that $i_7(U_{10}(2).2)<7^{27.81}$. 
By Propositions \ref{weilunitary}, \ref{intchar} and \ref{itoeig}, $\emax(g_{\{2,3,7\}'}) \leq 256$ on a Weil module. Let $H=G/F(G)$. If $G$ has no regular orbit on the 342-dimensional Weil module $V$,
\[
r^{342} \leq 2(i_2(H)+i_3(H)+i_7(H))7^{\floor{9\times 342/10}}+ 2|U_{10}(2).2|7^{256},
\]
and this is false for $r=7$. An analogous inequality gives the result when $d=341$.
Now suppose $r=5$. 
%Using \cite[Table B.4]{bg}, we deduce that  $i_5(U_{10}(2).2)<5^{35.66}$. 
By Propositions \ref{weilunitary}, \ref{intchar} and \ref{itoeig}, $\emax(g_{\{2,3,5\}'}) \leq 160$ on a Weil module. Let $H=G/F(G)$. If $G$ has no regular orbit on the 342-dimensional Weil module $V$, then by Proposition \ref{tools}\ref{qsgood},
\[
r^{342} \leq \frac{1}{4}i_5(H)\times 5^{\floor{9\times 342/10}}+2(i_2(H)+i_3(H))5^{\floor{9 \times 342/10}}+ 2|U_{10}(2).2|5^{160},
\]
and this is false for $r=5$. An analogous inequality gives the result when $d=341$.
Finally,  let us assume $r=3$. Here we use Propositions \ref{weilunitary}, \ref{intchar} and \ref{itoeig} to deduce that $\emax(g_{\{2,3\}'}) \leq 256$ on any of the Weil modules. We observe that every class of elements of order 3 in $G$ has a representative in the subgroup $K = \mathrm{SU}_6(2) \times \mathrm{SU}_4(2)$.
By Proposition \ref{su_weil_tensor},  the $2^{10}$-dimensional Weil module $V'$ for $\mathrm{SU}_{10}(2)$ restricts to the subgroup $\mathrm{SU}_6(2) \times \mathrm{SU}_4(2)$
	as the tensor product of a $2^6$-dimensional module for $\mathrm{SU}_6(2)$, and a $2^4$-dimensional module for $\mathrm{SU}_4(2)$.  As described in Section \ref{weilbackground}, the $2^6$-dimensional module restricts as the sum of three irreducible Weil modules of dimensions 21, and one copy of the trivial module for $\mathrm{SU}_6(2)$. On the other hand, the $2^4$-dimensional module restricts as the sum of three irreducible Weil modules of dimensions 5, and one copy of the trivial module for $\mathrm{SU}_4(2)$. Therefore,
	\[
V' \downarrow K = (21+21+21+1)\otimes (5+5+5+1) =(21\otimes 5)^9 +(21\otimes 1)^3+(1\otimes 5)^3+ 1\otimes 1.
	\]
Since $V'$ restricts to $\mathrm{SU}_{10}(2)$ as a sum of two Weil modules of dimension 341 and one of dimension 342, we deduce that $V_{341}(3) \downarrow K$ has composition factors $ (21\otimes 5)^3/21\otimes 1/1\otimes 5$, while $V_{342}(3) \downarrow K$ has composition factors $ (21\otimes 5)^3/21\otimes1 /1\otimes 5/1\otimes 1$. So by Proposition \ref{compfactors}, any element of order 3 has a fixed point space of dimension at most $3\times(\floor{\frac{5\times 21}{6}} \times 5)+\lfloor\frac{5\times 21}{6}\rfloor+ 1\times 5 = 273$ on $V_{341}(3)$, and $274$ on $V_{342}(3)$. By inspecting \cite[Table B.4]{bg}, we deduce that $i_2(G/F(G)) < 3^{34.29}$ and $i_3(G/F(G)) < 3^{44.54}$ by Proposition \ref{invols}. Therefore, if $G$ has no regular orbit on $V$,
	\[
	r^{d} \leq 2|U_{10}(2).2|r^{256} + 2i_2(G/F(G))r^{\floor{9d/10}}+ i_3(G/F(G))r^{d-68},
	\]
	where $d=341$ or 342. This is false for $r=3$, so $G$ has a regular orbit on all of the Weil modules $V$.
%
%
%
%
%The $2^{10}$-dimensional irreducible Weil module for $4 \circ 2^{1+20}.\mathrm{SU}_{10}(2)$ restricts to the subgroup $\mathrm{SU}_6(2)$ as the sum of composition factors of dimensions $(21+22+21)$ and to $\mathrm{SU}_4(2)$ as the sum of composition factors of dimensions $(5+6+5)$. 
%The restriction to $H = \mathrm{SU}_6(2) \times \mathrm{SU}_4(2)$ then is a tensor product $(21+22+21)\otimes (5+6+5)$ (cf. \cite[Lemma 9.3.1]{goodwin}). Thus $V_{341}(3) \downarrow H$ has composition factors $ (21\otimes 5)^3/21\otimes 1/1\otimes 5$, while $V_{341}(3) \downarrow H$ has composition factors $ (21\otimes 5)^3/21\otimes1 /1\otimes 5/1\otimes 1$. So any element of order 3 has a fixed point space of dimension at most $3\times(\floor{\frac{5\times 21}{6}} \times 5)+\lfloor\frac{5\times 21}{6}\rfloor+ 1\times 5 = 273$ on $V_{341}(3)$, and $274$ on $V_{342}(3)$. We have by \cite[Table B.4]{bg} that $i_2(G/F(G)) < 2^{45}$ and $i_3(G/F(G)) < 3^{44.54}$. Therefore, if $G$ has no regular orbit on $V$,
%\[
%r^{d} \leq 2|U_{10}(2).2|r^{256} + 2i_2(G)r^{\floor{9d/10}}+ i_3(G)r^{d-68}
%\]
%This is false for $r=3$, so $G$ has a regular orbit on all of the Weil modules $V$.
\end{proof}
%\subsubsection{$(n,q)=(9,2)$}
%DONEx2

\begin{proposition}
		Theorem \ref{umain} holds when $E(G)/Z(E(G)) \cong U_{9}(2)$.
	\end{proposition}
\begin{proof}
We need to consider $r\leq 17$, and by \cite{HM}, we only need consider the Weil modules of degree 170 and 171. The 171-dimensional Weil modules exist in characteristics $r_0 \neq 2,3$, and their Brauer characters contain the $z_3$ irrationality, so we only have $r=7$ here.
By Propositions \ref{weilunitary}, \ref{intchar} and \ref{itoeig},
$\emax(g_{\{2,3,5\}'}) \leq 80$ for semisimple $g_{\{2,3,5\}'}$, and by \cite[Lemma 8.3.10]{goodwin} elements $g \in E(G)$ of projective order 2,3 have $\emax(g) \leq 128$. 
	
	We compute that $\alpha(g_{r_0}) \leq 3$ for $r_0 \in\{5,7,11,17\}$, so $\emax(g_{r_0}) \leq 113+\delta$ for $V$ a Weil module of dimension $170+\delta$. Moreover, we find that  $\alpha(x) \leq 6$ for $x$ an outer element of order 3, except for two classes of order 43776 with $\alpha(x)\leq 8$. We also calculate from \cite[Table B.4]{bg} that the number of outer elements of order 3 in $G/F(G)$ is at most $7^{19.29}$, the number of elements of order 5 is less than $5^{27.48}$, and by Lemma \ref{graph}, the number of involutory graph automorphisms in $G/F(G)$ is at most $2^{45}$.
	
	If $d=171$, $r=7$ and $G$ has no regular orbit on $V$, then 
	\[
	r^{171} \leq 2 |U_9(2).S_3|r^{128}+ 2 \times 7^{19.29}r^{\floor{7\times 171/8}}+ 2\times 2^{45} r^{\floor{8\times 171/9}}.
	\]
	But this is false, implying the existence of a regular orbit.
	Let $H=G/F(G)$.
	If $d=170$ and $G$ has no regular orbit on $V$, then, making use of Proposition \ref{invols}, we have 
	\[
	r^{170} \leq 2 |U_9(2).S_3|r^{128}+ 2 \times i_3(H) r^{\floor{7\times 170/8}}+ 2\times i_2(H)r^{\floor{8\times 170/9}},
	\]
	which is false for $r=7,9,11,17$. 
	Now suppose $r=5$. By \cite[Proposition 3.3.17]{bg}, $U_9(2).S_3$ has a single $U_9(2).S_3$-class of involutory graph automorphisms, which splits into three $U_9(2)$-classes. Namely, there are three groups of shape $U_9(2).2$, each of which has a single class of outer automorphisms and these together comprise the graph automorphisms of $U_9(2).S_3$.
	We restrict $V$ to the subgroup $J_3.2< U_9(2).2<U_9(2).S_3$.
	By \cite[Theorem 4.10.8]{BHRD}, $J_3$ is a maximal subgroup of $U_9(2)$ which can be extended so that each of the three groups of shape $U_9(2).2$ has a maximal subgroup $J_3.2$. 
	It follows that the outer automorphisms in $J_3.2 < U_9(2).2$ are contained in the unique class of graph automorphisms of $U_9(2).2$
	The Brauer character tables for $J_3.2$ are available in \cite[p. 215--219]{modatlas}. The smallest non-trivial irreducible module for $J_3.2$ in characteristic 5 has dimension 170, so the restriction of $V$ to $J_3.2$ must be irreducible. 
%	By \cite[\S 3.3]{bg}, there is one outer class of involutions in $U_9(2).S_3$, and by 
%By \cite[Proposition 3.3.17]{bg}, the classes of graph automorphisms contained in the groups of shape $U_9(2).2$ fuse under the action of $U_9(2).S_3$ to give a single class of 
%	Every class of involutory graph automorphisms in $U_9(2).S_3$ has a representative $g$ amongst outer automorphisms of $J_3.2$, so 
We therefore deduce from the Brauer character tables of $J_3.2$ that $\dim C_V(g)\leq 85$ for a graph automorphism $g$.
Therefore, if $G$ has no regular orbit on $V$,
	\[
	r^{170} \leq  2 |U_9(2).S_3| r^{80} + i_5(H) r^{\floor{2\times 170/3}} + 2(i_2(H)+i_3(H))r^{141}+2(2\times 43776)r^{\floor{7\times 170/8}}.
	\]
	This is false for $r=5$.
	Finally, suppose $r=3$. Then $V$ is an irreducible (but not necessarily faithful) module for $K_1 =\mathrm{GU}_9(2)$. From \cite[Table B.4]{bg}, we deduce that every class of elements of order 3 in $K_1$ has a representative in $K_2=\mathrm{GU}_6(2)\times\mathrm{GU}_3(2)$. Let $\rho$ be the representation from $K_1$ into $\gl(V)$, so that $\rho(K_1)$ is faithful. Then every element of order 3 in $G$ lies in $\rho(K_1)$, so is conjugate to an element of $\rho(K_2)$. We now restrict $V$ to $\rho(K_2)$. The composition factors of $V \downarrow K_2$ are tensor products of Weil modules for $\mathrm{GU}_6(2)$ and $\mathrm{GU}_3(2)$ (cf. \cite[Corollary 3.4]{GERARDIN197754}). In particular, $V \downarrow K_2=(3\otimes 21)^2/2\otimes 21/ 2\otimes 1$. We find that all elements  of order 3 in $\mathrm{GU}_6(2)$ have fixed point spaces of dimension at most 11 on the 21-dimensional Weil module for $\mathrm{GU}_6(2)$. Therefore, by Propositions \ref{compfactors} and  \ref{tensorcodim}, for $g\in G$ of order 3, we have $\dim C_V(g) \leq 2(11\times 3)+ 2\times 11+2\times 1 = 90$. It follows that if $G$ has no regular orbit on $V$,
	\[
	r^{170} \leq  2 |U_9(2).S_3| r^{80} + 2i_5(H) r^{\floor{4\times 170/5}} + 
	i_3(H) r^{90} + 2i_2(H) r^{128}.
	\]
	This is false for $r=3$, so $G$ has a regular orbit on $V$.
\end{proof}

\begin{proposition}
Theorem \ref{umain} holds when $E(G)/Z(E(G)) \cong U_{8}(2)$.
	\end{proposition}
\begin{proof}
 The number of prime order elements in $\mathrm{P} \Gamma \mathrm{U}_8(2)$ is less than $2^{58.92}$. By Proposition \ref{tools}\ref{crude} and examining \cite{HM}, we see that we only need consider the Weil modules of $G$ of degrees 85 and 86 for $r\leq 61$. By Propositions \ref{weilunitary}, \ref{intchar} and \ref{itoeig}, we compute that semisimple elements of prime order at least 7 have $\emax(g_{\{2,3,5\}'}) \leq 40$. Now, constructing the character table of $U_8(2).2$ using Magma, we determine that $\emax(g_{\{2,3,5\}}) \leq 54$, except that elements in class $2A$ have $\emax(g) \leq 64$. We summarise upper bounds on the number of prime order elements in Table \ref{u82-prime}.
\begin{table}[h!]
\begin{tabular}{ccccccccc}
\toprule
Order $r_0$ & 2 (inner) & 2 (outer) & 3 & 5 & 7 & 11 & 17 & 43 \\ 
\midrule
$\log_{r_0}(\#r_0\textrm{-elts})<$ & 32.31 & 35.23 &  26.45& 20.87 & 19.29 & 14.72 & 14.02 & 10.74 \\ 
\bottomrule
\end{tabular} 
\caption{Bounds on the number of elements of prime order in $\mathrm{Aut}(U_8(2))$.\label{u82-prime}}
\end{table}
	We now find upper bounds for fixed point spaces of elements of order $r_0$ in modules in characteristic $r_0$. We compute that if $r_0\geq 5$, then $\alpha(g_{r_0})=2$. If instead $r_0 = 3$, then we find that $\alpha(g_3)\leq 3$, unless $g_3$ belongs to a class of size 21760 or one of two further classes, each of size 29941760, in which case $\alpha(g_3)\leq 8,4$ respectively.
%	\begin{table}[h!]
%		\centering \begin{tabular}{cccc}
%			\toprule
%			$r_0$ & Technique & Bound on 85-dim reps $\nu_{85}$ & Bound on 86-dim reps $\nu_{86}$ \\ 
%			\midrule
%		%	3 & Restrict to $\mathrm{SU}_4(2)\times \mathrm{SU}_4(2)$  & 53 & 54 \\ 
%			5 & Compute in GAP $\alpha(x)=2$  & 42 & 43 \\ 
%			7 & Compute in GAP $\alpha(x)\leq 3$  & 56 & 57 \\ 
%			11 & Compute in GAP $\alpha(x) =2$  & 42 & 43\\  
%			17 &  Compute in GAP $\alpha(x) \leq 3$  & 56 & 57 \\ 
%			43 & Compute in GAP $\alpha(x) =2$  & 42 & 43 \\ 
%			\hline 
%		\end{tabular} 
%		\caption{Bounds on unipotent elements in absolutely irreducible modules for $\mathrm{SU}_8(2)$. \label{u82-case-analysis}}
%	\end{table}
	Let $H=G/F(G)$. If $G$ has no regular orbit on one of the 86-dimensional modules $V = V_d(r)$ in characteristic $r_0$,
	\begin{align*}
	r^{86} \leq & 2\times 2^{58.92}r^{40} + 2i_2(H)r^{54}+2i_5(H)r^{43}+ 2\times |2A|r^{64} + i_{r_0}(H)r^{86/2}+2i_3(H)r^{\floor{2\times 86/3}}\\
	&+ 2\times 21670r^{\floor{7\times 86/8}}+ 4\times 29941760r^{\floor{3\times 86/4}}.
	\end{align*}
	This is false for $r\geq 3$. A similar calculation gives the result for the 85-dimensional module.
\end{proof}
%\subsubsection{$(n,q)=(7,3)$}
%DONEx2

		\begin{proposition}
	Theorem \ref{umain} holds when $E(G)/Z(E(G)) \cong U_{7}(3)$.
	\end{proposition}
\begin{proof}
By Propositions \ref{tools}\ref{crude} and \ref{d2unitary}, we need only consider the 546-dimensional Weil module for $r=2$. By Propositions \ref{weilunitary}, \ref{intchar} and \ref{itoeig}, $\emax(g_{\{2,3\}'}) \leq 459$. Therefore, if $G$ has no regular orbit on $V$, then 
\[
r^{546} \leq 2|U_7(3).2|r^{459}+2(i_2(G/F(G))+i_3(G/F(G)))r^{\floor{6\times 546/7}}
\]
which is false for $r=2$, so $G$ has a regular orbit on $V$.
\end{proof}
%\subsubsection{$(n,q)=(7,2)$}
%DONE

		\begin{proposition}
Theorem \ref{umain} holds when $E(G)/Z(E(G)) \cong U_{7}(2)$.
	\end{proposition}
\begin{proof}
	Let $H = U_7(2).2$. We begin by summarising some information about the prime order elements in $H$, computed using GAP, in Table \ref{u72-counts}. In particular, we observe that $H$ contains fewer than $2^{45.11}$ elements of prime order.
	\begin{table}[h!]
		\centering \begin{tabular}{ccc}
			\toprule
			$r_0$ & $\log_{r_0}i_{r_0}(H)<$  & $\alpha(x) \leq$ \\ 
			\midrule
			2 & $27.33$ & 3 (2D), 7 \\  
			3 & $20.16$ & 2 (3J/3K), 7 \\ 
			5 & $15.52$ & 2 \\  
			7 & $14.87$ & 2 \\ 
			11 & $11.88$ & 2 \\ 
			43 & $8.27$ & 2 \\ 
%			\midrule
%			Total & $2^{45.11}$ & \\
			\bottomrule
		\end{tabular} 
		\caption{Bounds on  the number of prime order elements in $U_7(2).2$. \label{u72-counts}}
	\end{table}
	
	If $G$ has no regular orbit on $V$, then by \ref{crude} from Proposition \ref{tools},
	\[
	\frac{1}{2}r^d \leq (i_{43}(H)+i_{11}(H)+ i_7(H)+ i_5(H))r^{\floor{d/2}}+ ( i_2(H)+ i_3(H))r^{\floor{6d/7}}.
	\]
	From \cite{HM}, we see that we only need to consider Weil modules of dimensions 42, 43 for $r\leq 43$ and $r=7,25$ respectively,  since the 43-dimensional representation does not exist in characteristics 2,3 and the Brauer character contains a $z_3$ irrationality.
	
	First let $V = V_{43}(r)$. We compute that $\alpha(g_{\{5,7\}}) = 2$. We can construct the ordinary character table of $U_7(2).2$ using Magma \cite{Magma}. From this, we can infer the values of the Brauer characters in characteristics $p=5,7$. We therefore deduce that if $G$ has no regular orbit on $V$,
		\begin{align*}
		\frac{1}{2} r^{43} \leq i_{43}(H) r+i_{11}(H) r^{4}+ i_7(H)r^{\floor{43/2}}+  i_5(H)r^{\floor{43/2}} &+ i_3(H)r^{26} +|2A|r^{32}\\
		& + (|2B|+|2C|)r^{26}+ |2D|r^{\floor{2\times 43/3}}.
		\end{align*}
	This inequality is false for both $r=7,25$.
	
	Now suppose $V$ is the 42-dimensional Weil module for $U_7(2).2$. As before, we can deduce the Brauer character of $V$ from the ordinary character of a Weil module over $\mathbb{C}$ of the same dimension. Therefore, if $G$ has no regular orbit on $V$ in characteristics 5,7 then
	\begin{align*}
	\frac{1}{2} r^{42} \leq i_{43}(H) r+i_{11}(H) r^{4}+ i_7(H)r^{21}+  i_5(H)r^{21}+ i_3(H)r^{26} &+|2A|r^{32}\\
	&+ |2B|r^{26}+|2C|r^{24} + |2D|r^{21}.
	\end{align*}
	This inequality is false for $r=5,7,25$. Similar computations give the results in characteristics 11 and 43.

Finally let us assume $V = V_{42}(3^k)$. Here we bound $\emax(g_3)$ by restricting the module to $K=3.U_6(2).3$, which contains a Sylow 3-subgroup of $U_7(2)$. By comparing the character tables of these groups, we find that $V\downarrow K$ has two composition factors, each of dimension 21. Every element of order 3 in $K$ has a fixed point space of dimension at most 11 on each of these 21-dimensional modules, so, applying Proposition \ref{compfactors}, we have $\dim C_V(g_3)\leq 22$. We can also infer the values of the Brauer character from the ordinary character of a Weil module of the same dimension (cf. \cite[\S 6]{modatlas}). So if $G$ has no regular orbit on $V$,
	\begin{align*}
	r^{42} \leq & 2 i_{43}(H) r+ 2i_{11}(H) r^{4}+ 2 i_7(H)r^{6}+ 2  i_5(H)r^{21}+\frac{1}{2}(i_3(H)-|3J|-|3K|)r^{22}\\
	&+\frac{1}{2} (|3J|+|3K|)r^{21}+|2A|(r^{32}+r^{10}) + 2|2B|r^{26}+2|2C|r^{24} +2 |2D|r^{23}.
	\end{align*}
	%+(|3A|+|3B|)r^{22}+(|3C|+|3D|+|3E|)r^{26}+(|3F|+|3G|)r^{22}+(|3H|+|3I|)r^{18}
	This is false for $r\geq 3$.
	%\subsubsection{$(n,q)=(6,3)$}
	%DONEx2
	\end{proof}
%\subsubsection{$(n,q)=(6,3)$}
%DONEx2

		\begin{proposition}
Theorem \ref{umain} holds when $E(G)/Z(E(G)) \cong U_6(3)$.
	\end{proposition}
\begin{proof}
We have $d_1(G) \geq 182$ by Proposition \ref{mindegprop}, while $\alpha(G) \leq 6$ by Proposition \ref{alphas}. Therefore, applying Proposition \ref{tools}\ref{crude} and examining \cite{HM}, we see that we only need to consider $r=2$ and $V = V_{182}(2)$.

From the Brauer character (computed using Magma \cite{Magma}), we deduce that $\emax(g_{2'}) \leq 81$. Using GAP, we determine that $\alpha(g_2) \leq 5$, except for elements in class 2A where $\alpha(x)\leq 6$ by Proposition \ref{alphas}.
Therefore, if $G$ has no regular orbit on $V$ then
\[
r^{182} \leq 2|U_6(3).2^2|r^{81} + |2A|r^{\floor{5\times 182/6}} + i_2(G/F(G))r^{\floor{4\times 182/5}}
\]
This is false, so $G$ has a regular orbit on $V$.
\end{proof}
%\subsubsection{$(n,q)=(6,2)$}
%DONE

	\begin{proposition}
Theorem \ref{umain} holds when $E(G)/Z(E(G)) \cong U_{6}(2)$.
	\end{proposition}
\begin{proof}
Let $H = U_6(2).S_3$.
We summarise some information about elements of prime order in $H$ in Table \ref{u62-counts}.
\begin{table}[h!]
\begin{tabular}{ccc}
\toprule
$r_0$ &$ i_{r_0} (H)$ & $\alpha(g_{r_0}) \leq$ \\ 
\midrule
2 & 1529055 & 6 (2A/2D), 3 (o/w) \\ 
3 & 22061888 & 6 (3D), 3 (o/w) \\ 
5 & 306561024 & 2 \\ 
7 & 1313832960 & 2 \\ 
11 & 1672151040 & 2 \\ 
\bottomrule
\end{tabular} 
\caption{Elements of prime order in $U_6(2).S_3$. \label{u62-counts}}
\end{table}
Therefore, if $G$ has no regular orbit on $V$,
\[
\frac{1}{2} r^d \leq ( i_{11}(H)+ i_{7}(H)+ i_{5}(H))r^{\floor{d/2}} + ( i_{3}(H)+ i_{2}(H))r^{\floor{2d/3}}+ (|2A|+|2D|+|3D|)r^{\floor{5d/6}}
\]
Consulting \cite{HM}, we see that the only modules that remain are the Weil modules of dimension 21 and 22 for $r\leq 11$.
First suppose $V$ is a 21-dimensional Weil module for $U_6(2).S_3$ in characteristic 3. When $r=3$, we compute using Magma that there is no regular orbit of $G$ on $V$.
Now suppose $r=9$. Then from constructing the module in GAP, we find that  $\dim C_V(g_3) \leq 11$ and all involutions have eigenspaces of dimension less than 14, except elements in class 2A, which have an eigenspace of dimension 16. Therefore, if $G$ has no regular orbit on $V$,
\[
r^{21} \leq 2 ( i_{11}(H) +  i_{7}(H) +  i_{5}(H)) r^{10} + 
   i_{3}(H) r^{11} + 2  i_{2}(H) r^{14}+ 693(r^{16}+r^5)
\]
This is false for $r=9$, so $G$ has a regular orbit on $V$ and $b(G)=2$ for $r=3$ by Lemma \ref{fieldext}.
Now suppose that $V$ is a 21-dimensional Weil module for $3.U_6(2).3$ in characteristic $\neq 2,3$. 
Since the Brauer character contains a $z_3$ irrationality, we only consider $r=7$. Constructing the module in GAP, we find that $\dim C_V(g_7) = 3$. This along with observations from the Brauer character (constructed in Magma \cite{Magma}) show that if $G$ has no regular orbit on $V$,
\[
r^{21} \leq  2 ( i_{11}(H) +  i_{7}(H) +  i_{5}(H)) r^5 +2  i_{3}(H) r^{11} + 
  2  i_{2}(H) r^{12}
\]
This is false, so $G$ has a regular orbit on $V$.

\begin{table}[h!]
\begin{tabular}{p{3cm}cccccccc}
\toprule
 & \multicolumn{8}{c}{Class in $U_6(2)$}  \\ 
\cmidrule{2-9}
Characteristic $r_0$         & 2A       & 2B     & 2C      & 3A         & 3B        & 3C         & 5A        & 7A            \\ \midrule
5                            & $(16,6)$ & (14,8) & (12,10) & $(10,6^2)$ & $(9^2,4)$ & $(10,6^2)$ & 6         & $(4,3^6)$   \\
7                            & $(16,6)$ & (14,8) & (12,10) & $(10,6^2)$ & $(9^2,4)$ & $(10,6^2)$ & $(6,4^4)$ & 4           \\
11                           & $(16,6)$ & (14,8) & (12,10) & $(10,6^2)$ & $(9^2,4)$ & $(10,6^2)$ & $(6,4^4)$ & $(4,3^6)$          \\ \midrule
Size of class in $U_6(2).S_3$ & 693      & 62370  & 249480  & 118272     & 197120    & 4730880    & 306561024 & 1313832960 \\ \bottomrule
\end{tabular}
\caption{Eigenspace dimensions on $V= V_{10}(r)$, Part I. \label{u62-case-analysis}}
\end{table}
\begin{table}[h!]
\ContinuedFloat
\begin{tabular}{p{3cm}ccccccc}
\toprule
 & \multicolumn{7}{c}{Class in $U_6(2)$}  \\ 
\cmidrule{2-8}
Characteristic $r_0$          & 11A/B      & 2D       & 2E       & 3D       & 3E         & 3F        & 3G        \\ \midrule
5                             & $(2^{11})$ & $(15,7)$ & $(11^2)$ & $(11^2)$ & $(12,5^2)$ & $(8^2,6)$ & $(8,7^2)$ \\
7                             & $(2^{11})$ & $(15,7)$ & $(11^2)$ & $(11^2)$ & $(12,5^2)$ & $(8^2,6)$ & $(8,7^2)$ \\
11                            & 2          & $(15,7)$ & $(11^2)$ & $(11^2)$ & $(12,5^2)$ & $(8^2,6)$ & $(8,7^2)$ \\
\midrule
Size of class in $U_6(2).S_3$ & 1672151040 & 19008    & 1197504  & 1344     & 118272     & 4730880   & 12165120  \\ \bottomrule
\end{tabular} 
\caption{Eigenspace dimensions on $V= V_{10}(r)$, Part II. }
\end{table}

Now suppose that $V$ is a 22-dimensional Weil module for $U_6(2).S_3$. From \cite{HM} we observe that this module does not exist in characteristic 3, so we only need consider $V$ in characteristic 5,7 or 11.  We summarise our analysis of $V$ in Table \ref{u62-case-analysis}. The eigenspace dimensions in the Table were derived from the corresponding Brauer characters and from computing the eigenspaces of elements of the relevant matrix group in GAP.
Therefore, if $G$ has no regular orbit on $V$ in characteristic 5,
\begin{align*}
r^{22} & \leq   \frac{1}{10}i_{11}(H)(11r^2) + \frac{1}{6} i_7(H) (r^4+6r^3) + \frac{1}{4}i_5(H) r^6 + 
  \frac{1}{2}|3A| (r^{10} + 2 r^6) +  \frac{1}{2}|3B| (2 r^9 + r^4) \\
  &+ 
   \frac{1}{2}|3C|(r^{10} + 2 r^6) +  \frac{1}{2}|3D| (2 r^{11}) + \frac{1}{2}|3E| (r^{12} + 2 r^5) + 
   \frac{1}{2}|3F|(2 r^8 + r^6) +  \frac{1}{2}|3G|(r^8 + 2 r^7) \\
 &+ |2A| (r^{16} + r^6) + 
 |2B| (r^{14} + r^8) + |2C|(r^{12} + r^{10}) + |2D| (r^{15} + r^7) + |2E| (2 r^{11})  
%& = 2 \times  i_{11}(H) r^2 + 2\times  i_{7}(H) r^4 + 2\times i_{5}(H) r^6 + 
%  118272 (r^{10} + 2 r^6) + 197120 (2 r^9 + r^4) \\
%  & + 
%  4730880 (r^{10} + 2 r^6) + 1344 (2 r^{11}) + 118272 (r^{12} + 2 r^5) + 
%  4730880 (2 r^8 + r^6) + 12165120 (r^8 + 2 r^7)\\
%  & + 693 (r^{16} + r^6) + 
%  62370 (r^{14} + r^8) + 249480 (r^{12} + r^{10}) + 19008 (r^{15} + r^7) + 
%  1197504 (2 r^{11})
\end{align*}
This is false for $r=5$, so $G$ has a regular orbit on $V$. The calculations for $r=7,11$ are similar.
\end{proof}
%\subsubsection{$(n,q)=(5,3)$}
%DONE

\begin{proposition}
Theorem \ref{umain} holds when $E(G)/Z(E(G)) \cong U_{5}(3)$.
	\end{proposition}
\begin{proof}
By Proposition \ref{tools}\ref{crude} and \ref{d2unitary}, $r\leq 7$ and we must consider 60-  and 61-dimensional Weil modules.We calculate that $i_2(U_5(3).2)= 5430159$, and the total number of prime order elements in $U_5(3).2$ is 54826442783. 
We begin with $r=7$.
Constructing the group in GAP, we deduce that $i_7(U_5(3).2)< 7^{10.52}$. From the character table of $U_5(3).2$, we find that $\emax(g_2) \leq 40$, and $\emax(g_{\{2,7\}'})\leq 27$. Therefore, if $G$ has no regular orbit on the 60-dimensional Weil module $V$, then
\[
r^{60} \leq 2\times 54826442783r^{27}+2i_2(G/F(G))r^{40}+7^{10.52}r^{\floor{4\times 60/5}}.
\]
This is false for $r=7$, so $G$ has a regular orbit on $V$. Analogous calculations show that $G$ also has a regular orbit on the three 61-dimensional Weil modules in characteristic 7.
Suppose now that $r=5$, and let $V$ be the 60-dimensional Weil module. Again, we find that $\emax(g_2) \leq 40$ from the character, and $\emax(g_{\{2,5\}'}) \leq 27$. We also compute that $i_5(U_5(3).2)<5^{13.61}$, and these elements all have $\alpha(x)=2$. Therefore, if $G$ has no regular orbit on $V$,
\[
r^{60} \leq 2\times 54826442783 r^{27} + 2\times 5430159 r^{40} + 5^{13.61} r^{30}.
\]
This is false for $r=5$, and analogous calculations give the results for the 61-dimensional modules as well. 
	Finally, suppose that $V$ is a 60-dimensional Weil module in characteristic 2.
	Examining the Brauer character we see that $\emax(g_3)\leq 27$, and $\emax(g_{\{2,3\}'})\leq 12$. 
	Using Magma, we deduce that elements $g$ in the
	classes 2A, 2B and 2C satisfy the bounds $\dim C_V(g) \leq 40, 34$ and 30 respectively. Therefore, if $G$ has no regular orbit on $V$,
	\begin{align*}
	r^{60} \leq &  2\times 54826442783 r^{21}+ 2i_3(G/F(G)) r^{27}+ 
	|2A| r^{40}  +|2B|r^{34}+ |2C|r^{30}.
	%=  2 \times 54826442783 r^{21} + 2\times 33618320  r^{27} + 
	%   + (444690 + 4941) r^{39} + 4980528 r^{30}
	\end{align*}
	This is false for $r\geq 2$.
	\end{proof}
%\subsubsection{$(n,q)=(5,2)$}
%DONE

\begin{proposition}
Theorem \ref{umain} holds when $E(G)/Z(E(G)) \cong U_{5}(2)$.
	\end{proposition}
\begin{proof}
The number of prime order elements in $H = U_5(2).2$ is 3479519. We summarise some further information about prime order elements in $U_5(2).2$ in Table \ref{u52-counts}.
\begin{table}[h!]
\begin{tabular}{ccc}
\toprule
 $r_0$ & $i_{r_0}(H)$ & $\alpha(x) \leq$ \\ 
\midrule
2 (inner) & 3135 & 5 \\ 
2 (outer) & 19008 & 3 \\ 
3 & 56672 & 2 (3F) 5 (o/w) \\ 
5 & 912384 & 2 \\ 
11 & 2488320 & 2 \\ 
\hline 
\end{tabular} 
\caption{Elements of prime order in $H= U_5(2).2$. \label{u52-counts}}
\end{table}
Therefore, applying Proposition \ref{tools}\ref{alphabound} and \ref{qsgood}, 
it follows from \cite{HM} that we only need to further consider a 44-dimensional module when $r=3$, 11-dimensional Weil modules for $r = 25$ and 10-dimensional Weil modules for $r\leq 187$.
If $V$ is the 44-dimensional module and $r=3$, we infer from the Brauer character that $\emax(g_{3'}) \leq  28$.
Therefore, if $G$ has no regular orbit on $V$,
\[
r^{44} \leq 2\times 3479519r^{28} + |3F|r^{\floor{44/2}}+(56672-42240)r^{\floor{4\times 44/5}},
\]
which is false, so $G$ has a regular orbit on $V$.
Now suppose $V = V_{11}(25)$.  From the Brauer characters of the two conjugate modules of dimension 11, we see that if $G$ has no regular orbit on $V$,
\[
r^{11} \leq 912384r^{5} + 3135(r^8+r^3)+2\times 56672 r^6+ 2488320r
\]
This is false for $r=25$. 
Now suppose $V$ is a 10-dimensional irreducible Weil module. 
%If $r=3,5$, then $|V|<|G|$ and $b(G) \geq  2$ with representatives $\{e_{10},e_9\}$ and 
We summarise information about eigenspace dimensions in $U_5(2).2$ obtained from the Brauer tables and constructions in GAP in Table \ref{u52-case-analysis}.
\begin{table}[h!]
\begin{tabular}{cccccccccc}
\toprule
 & \multicolumn{9}{c}{Class in $U_5(2).2$}  \\ 
\cmidrule{2-10}
Characteristic $r_0$ & 2A & 2B & 3A/3B & 3C/3D & 3E & 3F & 5A & 11A/11B & 2C \\ 
\midrule
3 & (8,2) & (6,4) & 5 & 5 & 6 & 4 & $(2^5)$ & $(1^{10})$ & $(5^2)$ \\ 
5 & (8,2) & (6,4) & $(5^2)$ & $(4,3^2)$ & $(6,2^2)$ & $(4^2,2)$ & 2 & $(1^{10})$ & $(5^2)$ \\ 
11 & (8,2) & (6,4) & $(5^2)$ & $(4,3^2)$ & $(6,2^2)$ & $(4^2,2)$ & $(2^5)$ & 1 & $(5^2)$ \\ 
\midrule
Size of class in $U_5(2).2$ & 165 & 2970 & 352 & 7040 & 7040 & 42240 & 912384 & 2488320 & 19008 \\ 
\bottomrule
\end{tabular} 
\caption{Eigenspace dimensions on $V= V_{10}(r)$. \label{u52-case-analysis}}
\end{table}
Now, if $G$ has no regular orbit on $V$,
\begin{align*}
r^{10} \leq &165 (r^8 + r^2) + 2970 (r^6 + r^4) + 352 (2 r^5) + 
  7040 (r^4 + 3 r^2 + r^6 + 2 r^2) + 42240 (2 r^4 + r^2) \\
  &+ 912384 (5 r^2) + 2488320 (10 r) + 19008 (2 r^5).
\end{align*}
This is false for $r\geq 25$. If $r=3$ then $|V|<|G|$ and we compute using GAP that $b(U_5(2).2) = 2$ for both 10-dimensional modules. We also compute that $b(2\times U_5(2))= 2$ and $b(2\times U_5(2).2) = 3$. If $r=5$ then we have $|V|<|G|$ as well, and $b(G)=2$ by Lemma \ref{fieldext}. If $r=9$ and $\rho \in \mathbb{F}_9^\times $, then we compute that $\langle \rho \rangle\times U_5(2)$ has a regular orbit on $V$ unless  $\langle \rho \rangle \cong \mathbb{F}_9^\times$, in which case $b(G) = 2$ by Lemma \ref{fieldext}. Moreover, $U_5(2).2$ has a regular orbit on each of the two 10-dimensional modules $V$, but $\langle \rho \rangle \times U_5(2).2$ does not for non-identity $\rho$. Again, in these cases we have $b(G) = 2$ by Lemma \ref{fieldext}. If instead $r=11$, then we compute that both $\langle \rho \rangle  \times U_5(2)$ and $\langle \rho \rangle  \times U_5(2).2$ always have a regular orbit on $V$.
\end{proof}
We are now able to establish the following result.
\begin{proposition}
	%	\label{n>=5_unitary}
	Theorem
  \ref{umain} holds when $E(G)/Z(E(G)) \cong U_n(q)$ with $n\geq 5$ and $q\geq 2$.
\end{proposition}
%\begin{proof}
%Applying Proposition \ref{tools}\ref{crude}, with $d_1(E(G)) \geq \lfloor (q^n-1)/(q+1)\rfloor$ and  $\alpha(G) \leq n$, we see that either $G$ has a regular orbit on $V$ or $G$ is a group listed in Table \ref{n>5ugrps}. Since we have proved Theorem  \ref{umain} for each of these groups, the result follows.
%\end{proof}

\section{Proof of Theorem \ref{mainthm}: Symplectic groups}
The following is the main result of this section. Note that $\mathrm{PSp}_4(2) \cong L_2(9).2_2$ and $\mathrm{PSp}_4(3)\cong U_4(2)$,  and these groups are considered in Propositions \ref{alt_grp_prop} and \ref{u_4_q} respectively,  so we do not need to consider them here.

\begin{theorem}
\label{mainsymp}
Suppose $G$ is an almost quasisimple group with $E(G)/Z(E(G))\cong \mathrm{PSp}_n(q)$, $n\geq 4$, with $(n,q) \neq (4,2), (4,3)$. Let $V=V_d(r)$, $(r,q)=1$ be a module for $G$, with absolutely irreducible restriction to $E(G)$. Also suppose that $(r, |G|)>1$. Then one of the following holds.
\begin{enumerate}
\item $b(G) = \lceil \log |G| /\log |V| \rceil$,
\item $b(G) = \lceil \log |G| /\log |V| \rceil+1$ and $(G,V)$ lies in Table \ref{allbad}, or 
%\item $(G,V)$ is one of  $(2\times \mathrm{PSp}_6(2), V_7(3))$, $(\mathrm{PSp}_4(3).2, V_5(3))$, $(2\times \mathrm{PSp}_4(3), V_5(3))$, or $(\mathrm{PSp}_4(3).2,V_6(2))$ and $b(G)= \lceil \log |G| /\log |V| \rceil+2=4,4,4,5$ respectively.
\item $(G,V)=(2\times \mathrm{PSp}_6(2), V_7(3))$ and $b(G)= \lceil \log |G| /\log |V| \rceil+2=4$.
\end{enumerate}
\end{theorem}

\subsection{Odd characteristic}
\begin{proposition}
\label{oddcharsprop}
Theorem \ref{mainsymp} holds for $E(G)/Z(E(G)) \cong \mathrm{PSp}_n(q)$, $q$ odd.
\end{proposition}

%\subsubsection{$n>4$}
Suppose $(n,q) \neq (4,3)$ (this case will be dealt with separately), and let $m=n/2$. From \cite[Table B.7]{bg} we compute that the number of transvections in $G/F(G)$ is at most $q^n-1$.  If $G$ has no regular orbit on an absolutely irreducible $\F_rE(G)$-module $V=V_d(r)$, then by Propositions \ref{mindegprop}, \ref{sympalphas} and \ref{tools}\ref{crude},
\begin{equation}
\label{generalspeqn}
r^d \leq 2|\mathrm{P}\Gamma \mathrm{Sp}_n(q)|r^{\floor{\frac{m+2}{m+3} d}}+ 2(q^{n}-1) r^{\floor{\frac{n-1}{n} d}}
\end{equation}
This is false in all cases except those listed in Table \ref{oddcharsremain}.
\begin{table}[h!]
\begin{tabular}{ccccccc}
\toprule
$n$ &14 & 12 & 10 & 8 & 6 & 4 \\ 
\midrule
$q$ &3& 3 & 3 & 3,5 & $\leq 9$ & $\leq 19$ \\ 
\bottomrule
\hline 
\end{tabular} 
\caption{Remaining symplectic groups in odd characteristic. \label{oddcharsremain}}
\end{table}

%\subsubsection{$(n,q,r) = (14,3,2)$}
%DONEx2
\begin{proposition}
Proposition \ref{oddcharsprop} holds for $E(G)/Z(E(G)) \cong \mathrm{PSp}_{14}(3)$.
\end{proposition}
\begin{proof}
By Proposition \ref{d2symp} and \eqref{generalspeqn}, we only need consider Weil modules $V$ of dimension 1093 for $r=2$. Let $\overline{V}$ be the 2187-dimensional Weil module (as described in Section \ref{weilbackground}) for $G$ with corresponding character $\rho$.
For an element $x \in G$ of projective prime order at least 5, we use Propositions \ref{speqn}, \ref{itoeig}, and \ref{intchar} to find the eigenspace dimensions of $x$ on $\overline{V}$ over $\overline{\mathbb{F}}_2$, and summarise our findings in Table \ref{s143-analysis}. These eigenspace dimensions of $x$ on $V$ form upper bounds for the eigenspace dimensions of $x$ on $V$, and therefore  $\emax(g_{\{2,3\}'}) \leq  486$ on $V$. From the proof of \cite[Theorem 4.3]{gs}, $\alpha(g_3)\leq 9$, unless $g_3$ is a transvection, where $\alpha(g_3) \leq 14$ by Proposition \ref{sympalphas}. The number of transvections in $G/F(G)$ is at most $3^{14}-1$ by \cite[Table B.7]{bg}, and for $x$ an involution, $\alpha(x) \leq 10$ by Proposition \ref{sympalphas}. Therefore, if $G$ has no regular orbit on $V$,
\[
r^{d} \leq 2|\mathrm{PSp}_{14}(3).2| r^{486} + i_2(G/F(G))r^{\floor{9d/10}}+2(3^{14}-1)r^{\floor{13d/14}} + 2i_3(G/F(G))r^{\floor{8d/9}}
\]
From the upper bounds for $i_2(G/F(G))$ and $i_3(G/F(G))$ given in Proposition \ref{invols}, we see that this is false for $r=2$ and $d=1093$.
\end{proof}
\begin{table}[h!]
\begin{tabular}{@{}cccc@{}}
\toprule
Proj. prime order & $\dim C_W(x)$ & $|\rho(x)|$ & Eigenvalue mults \\ \midrule
5 & 2 & 3 & $(438^4,435)$ \\
 & 6 & 27 & $(459, 432^4)$ \\
 & 10 & 243 & $(486^4,243)$ \\
7 & 2 & 3 & $(315, 312^6)$ \\
 & 8 & 81 & $(324^6,243)$ \\
11 & 4 & 9 & $(207, 198^{10})$ \\
13 & 2 & 3 & $(171, 168^{12})$ \\
 & 8 & 81 & $(243, 162^{12})$ \\
41 & 6 & 27 & $(54^{40},27)$ \\
61 & 4 & 9 & $(36^{60},27)$ \\
73 & 2 & 3 & $(30^{72},27)$ \\
547 & 0 & 1 & $(4,3^{546})$ \\
1093 & 0 & 1 & $(3,2^{1092})$ \\ \bottomrule
\end{tabular}
\caption{Eigenspace dimensions of some prime order elements in $\mathrm{PSp}_{14}(3)$ on a 2187-dimensional Weil module.}
\label{s143-analysis}
\end{table}

Recall the notation for the composition factors of a reducible module introduced in Section \ref{techniques}. 
\begin{proposition}
	\label{goodwin_12_3}
	Let $V_{364}$ and $V_{365}$ be  364- and 365-dimensional Weil modules for $\mathrm{Sp}_{12}(3)$ respectively over $\F_r$. Let $H = \mathrm{Sp}_{8}(3)\times  \mathrm{Sp}_{4}(3)$.
	\begin{enumerate}
		\item If $r$ is odd, $V_{364} \downarrow H = 41\otimes 4/40\otimes 5$, and $V_{365} \downarrow H  = 41\otimes 5/40\otimes 4$.
		\item If $r$ is even, then $V_{365}$ does not exist and $V_{364} \downarrow H = (40\otimes4)^2/1\otimes 4/40\otimes 1$.
	\end{enumerate}
Moreover, if $g \in \mathrm{Sp}_{12}(3)$ is semisimple of projective prime order not 5 or 73, then $\emax(g)\leq 273$ on a Weil module.
\end{proposition}
\begin{proof}
	This follows from the proof of \cite[Lemma 9.3.13]{goodwin}, and the fact that we can deduce the Brauer character of $V$ from the ordinary character of a Weil module over $ \mathbb{C}$ (cf. \cite [\S 6]{modatlas}).
\end{proof}
%\subsubsection{$(n,q) = (12,3)$}
%DONEx2
\begin{proposition}
\label{s123-prop}
Proposition \ref{oddcharsprop} holds for $E(G)/Z(E(G)) \cong \mathrm{PSp}_{12}(3)$.
\end{proposition}
\begin{proof}
Here by \eqref{generalspeqn} and Propositions \ref{d2symp} and \ref{tools}\ref{crude}, we only need consider the Weil modules of dimension 364, 365 for $r=4,5$ and $r=5$ respectively.
First suppose $r=5$. We restrict the Weil modules to $H=\mathrm{Sp}_8(3) \times \mathrm{Sp}_4(3)$ and by Proposition \ref{goodwin_12_3}, we see that $
V_{364} \downarrow H = 41\otimes 4+40\otimes 5$, while $
V_{365} \downarrow H  = 41\otimes 5+40\otimes 4$, where each integer denotes an irreducible module of that dimension.
Now, $H$ contains a Sylow 5-subgroup of $G$, so all classes of elements of order 5 in $G$ have representatives in $H$. Therefore, by Propositions  \ref{sympalphas} and \ref{tensorcodim},
any element of order 5 in $G$ has a fixed point space of dimension at most $41\times 3+40\times 4 = 283$ on the 364-dimensional module, and at most $41\times 4+40\times 3 = 284$ on the 365-dimensional module. Examining the Brauer character and usingProposition \ref{goodwin_12_3}, we see that $\emax(g_{\{5,73\}'}) \leq 273$ on either of the Weil modules. Since 73 is a primitive prime divisor of $3^{12}-1$, $\alpha(x) \leq 3$ for any $x \in G$ of order 73 by the proof of \cite[Theorem 4.3]{gs}. Therefore, if $G$ has no regular orbit on $V$, then $r^d \leq 2|\mathrm{PSp}_{12}(3).2|r^{284}$, and this is false for $r=5$, $d=364$ or $365$. 

Now suppose $V=V_{364}(4)$. Since neither of the two Weil modules of this dimension extend to a module for $\mathrm{PSp}_{12}(3).2$, we have $G/F(G) \cong  \mathrm{PSp}_{12}(3)$. Let $\rho$ be the natural projection map from $\mathrm{Sp}_{12}(3)$ to $ \mathrm{PSp}_{12}(3)$. We observe that all classes of involutions in  $\mathrm{PSp}_{12}(3)$ have representatives in $H=\rho(\mathrm{Sp}_8(3) \times \mathrm{Sp}_4(3))$.  Now, $V_{364} \downarrow H$ has composition factors $(40\otimes 4)^2$, $1\otimes 4$ and $40\otimes 1$ by Proposition \ref{goodwin_12_3}. By Propositions \ref{compfactors} and \ref{tensorcodim}, every involution in $G$ has a fixed point space of dimension at most $2\times 40\times 3+3+40 = 283$. Using the same argument as the characteristic 5 case, we deduce that $\emax(g_{2'})\leq 273$. We also compute that the number of inner involutions of $G$ is less than $2^{68}$. Therefore, if $G$ has no regular orbit on $V$,
\[
r^{364} \leq 2|\mathrm{PSp}_{12}(3).2|r^{273} + 2^{68}r^{283}
\]
This is false for $r=4$, so $G$ has a regular orbit on $V$.
\end{proof}
%\subsubsection{$(n,q) = (10,3)$}
%DONEx2
\begin{proposition}
Proposition \ref{oddcharsprop} holds for $E(G)/Z(E(G)) \cong \mathrm{PSp}_{10}(3)$.
\end{proposition}
\begin{proof}
By \eqref{generalspeqn} and Proposition \ref{d2symp}, we only need to consider the Weil modules of $G$ of dimensions 121 and 122 with $r \in \{ 4,7,13,16,25,49\}$, due to $z_3$ irrationalities in the corresponding Brauer characters. 

 We summarise upper bounds for eigenspace dimensions of elements of projective prime order acting on $V=V_{121}(r)$ in Table \ref{s103-case-analysis}. We obtain the bounds for elements of projective order $r_0 \mid r$ by restricting $V$ to $\mathrm{Sp}_8(3) \times \mathrm{Sp}_2(3)$ and then applying Propositions \ref{compfactors} and \ref{tools}\ref{alphabound}.  For $r_0 \nmid r$, we achieve the bounds for $r_0 = 11,41,61$ by applying Propositions \ref{speqn} and \ref{itoeig}. For the remaining semisimple elements of projective prime order, \cite[Lemma 9.3.8]{goodwin} gives the eigenspace dimensions of these element on the Weil modules over $\mathbb{C}$, and from this we can deduce the eigenspace dimensions on $V$. Applying Proposition \ref{tools}\ref{qsgood} it follows that $G$ has a regular orbit on $V$ for all  valid $r$. If instead $V=V_{122}(r)$, we compute upper bounds on $\emax$ which are at most one greater than those listed in Table \ref{s103-case-analysis}, so another application of Proposition \ref{tools}\ref{qsgood} yields the result.

 \begin{table}[h!]
\begin{tabular}{@{}ccc@{}}
$r_0$ & $r_0 \nmid r$ & $r_0 \mid r$ \\ \midrule
2 & 82 (2A), 70 & 94 \\
3 & 81 (3A/3B) 68 (O/W) & -- \\
5 & 27 & 67 \\
7 & 18 & 76 \\
11 & 23 & -- \\
13 & 14 & 94 \\
41 & 8 & -- \\
61 & 4 & -- \\ \bottomrule
\end{tabular}
\caption{Bounds on eigenspace dimensions for projective prime order elements in $G$ with $\mathrm{soc}(G/F(G)) \cong \mathrm{PSp}_{10}(3)$ acting on $V_{121}(r)$. }
\label{s103-case-analysis}
\end{table}
\end{proof}
%\subsubsection{$(n,q) = (8,5)$}
%DONEx2
\begin{proposition}
Proposition \ref{oddcharsprop} holds for $E(G)/Z(E(G)) \cong \mathrm{PSp}_{8}(5)$.
\end{proposition}
\begin{proof}
From Proposition \ref{d2symp} we deduce that the non-trivial irreducible modules of $G$ of minimal dimension are Weil modules of dimensions 312 and 313, and the next smallest module has dimension at least 32240. By \eqref{generalspeqn}, we see that we only need to consider the Weil modules for $r=2,3$; however the Weil modules are not realised over either field by \cite[\S 5]{MR1947325}. %Check - looks like from Hiss and Malle that contains b5 irr.
\end{proof}
%\subsubsection{$(n,q) = (8,3)$}
%DONEx2
\begin{proposition}
Proposition \ref{oddcharsprop} holds for $E(G)/Z(E(G)) \cong \mathrm{PSp}_{8}(3)$.
\end{proposition}
\begin{proof}
The number of prime order elements in $H=\mathrm{PSp}_8(3).2$ is 8658865029997007, and of these, 8022531299328000 are elements of order 41 with $\alpha(x) = 2$ (computed in GAP), and 6560 are transvections. By Propositions \ref{sympalphas} and \ref{tools} if $G$ has no regular orbit on the irreducible module $V=V_d(r)$,
\[
\frac{1}{2}r^d \leq i_{41}(H)r^{\floor{d/2}}+ 6560r^{\floor{7d/8}}+ (8658865029997007-i_{41}(H)-6560)r^{\floor{6d/7}}
\]
Therefore, by Proposition \ref{d2symp}, we only need to check the 40- and 41- dimensional Weil modules for $r\leq 329$.
In characteristic 41, there is nothing to consider because the Weil characters contain $z_3$, and  $\mathbb{F}_{41}$ does not contain a third root of unity. Moreover, $\mathbb{F}_5$ also has no third roots of unity,  so we only need to consider $r=25$ in characteristic 5. 
The computed upper bounds for $\emax$ on $V=V_{40}(r)$ are summarised in Table \ref{s83-case-analysis}. The values in the third column were computed from the associated Brauer character inferred from \cite[Table XV]{goodwin2}. If $r_0=5,7,13$, the bounds in the final column were determined by an application of \ref{tools}\ref{alphabound}, with a more accurate upper bound for $\alpha(g_{r_0})$ found using GAP. If instead $r_0=2$, the upper bound is found by constructing the relevant matrix group in Magma \cite{Magma} and computing the eigenspaces directly.
\begin{table}[h!]
	\centering \begin{tabular}{@{}cccc@{}}
		\toprule
	 $r_0$ & $i_{r_0}(\mathrm{PSp}(8,3).2) <$ & $r_0 \nmid r$ & $r_0 \mid r$ \\ \midrule
		2 & $2^{31.91}$ & 28 & 27* \\
		3 & $3^{24.11}$ & 27 (3A/3B), 23 (o/w) & -- \\
		5 & $5^{18.99}$ & 9 & 20 \\
		7 & $7^{16.92}$ & 6 & 20 \\
		13 & $13^{13.13}$ & 4 & 20 \\
		41 & $41^{9.87}$ & 2 & -- \\ \bottomrule
	\end{tabular}
	\caption{Upper bounds on $\emax$ for elements of projective prime order acting on $V_{40}(r)$. }
	\label{s83-case-analysis}
\end{table}
We have asterisked an entry in Table \ref{s83-case-analysis} because $\dim C_V(g_2)=27$ on $V=V_{40}(4^k)$ if $g_2$ is in class 2A, and $\dim C_V(g_2)\leq 24$ otherwise.
Therefore, applying Proposition \ref{tools}\ref{eigsp1}, we see that $G$ has a regular orbit on $V$ in all cases where $r$ is odd.

If $V$ is a Weil module in characteristic 2, then $d=40$ and $r = 4^k$.
Examining the Brauer character, we see that if $G$ has no regular orbit on $V$ then 
	\begin{align*}
	r^{40} &\leq 2^{31.91}r^{24}+|2A|r^{27}+|5A|(4r^9+r^4)+|5B|(5r^8)+|7A|(6r^6+r^4)+i_{13}(G/F(G))(r^4+12r^3)\\
	&+ i_{41}(G/F(G))(20r^2)+( |3A|+|3B|)(r^{27}+r^{13})+|3C|(2r^{18}+r^4)+|3D|(r^{22}+2r^9)\\
	&+(|3E|+|3F|)(r^{18}+r^{13}+r^9)+ |3G|(r^{16}+2r^{12})+|3H|(2r^{15}+r^{10})\\
	&+|3I|(r^{16}+2r^{12})+ (|3J|+|3K|)(r^{15}+r^{13}+r^{12}).
	\end{align*}
	This inequality is false for $r\geq 4$. 
Finally, if $d=41$ and $r$ is odd, then we compute upper bounds for $\emax$ and find that they are at most one greater than those given in Table \ref{s83-case-analysis}. An application of Proposition \ref{tools}\ref{eigsp1} gives the result.
\end{proof}
%\subsubsection{$(n,q) = (6,9)$}
%DONEx2
\begin{proposition}
Proposition \ref{oddcharsprop} holds for $E(G)/Z(E(G)) \cong \mathrm{PSp}_{6}(9)$.
\end{proposition}
\begin{proof}
	
By  \eqref{generalspeqn} and Proposition \ref{d2symp}, we only need to consider the Weil modules of $G$ of dimensions 364 for $r=2$.
Now $\mathrm{PSp}_6(9).2_1 < \mathrm{PSp}_{12}(3)$, and since the Weil modules for $\mathrm{PSp}_{12}(3)$ over $\F_2$ have dimension 364 as well, the Weil module  $V_2$ of $\mathrm{PSp}_{12}(3)$ must restrict irreducibly to $\mathrm{PSp}_6(9).2_1$, since 364 is the minimal dimension of a non-trivial absolutely irreducible module for both groups. 
By the proof of Proposition \ref{s123-prop}, we see that for projective prime order $g \in G$, $\emax(g) \leq 283$. Therefore if $G$ has no regular orbit on $V$,
$r^{364} \leq 2|\mathrm{PSp}_6(9).2_1|r^{283}$ and this is false for $r=2$.
\end{proof}
%\subsubsection{$(n,q) = (6,7)$}
%DONEx2
\begin{proposition}
Proposition \ref{oddcharsprop} holds for $E(G)/Z(E(G)) \cong \mathrm{PSp}_{6}(7)$.
\end{proposition}
\begin{proof}
By \eqref{generalspeqn} and Proposition \ref{d2symp}, we only need consider the two 171-dimensional Weil modules for $\mathrm{PSp}_6(7)$ with $r=2$. Using Proposition \ref{speqn}, we deduce that $\emax(g_{\{2,3\}'})\leq 91$. 
In GAP we determine that $\alpha(g_{\{2,3\}}) \leq 4$, so $\dim C_V(g_{\{2,3\}}) \leq 128$ by Proposition \ref{tools}\ref{alphabound}.
%We obtain bounds on the eigenspace dimensions of involutions and 3-elements by restricting the module to $H = \mathrm{Sp}_4(7)\times \mathrm{Sp}_2(7)$, where $
%V \downarrow H $, has composition factors $(24\otimes 3)^2, 1\otimes 3$ and $24\otimes 1$.
%Every class of involutions and elements of order 3 has representatives in $H$. Therefore, by Proposition \ref{compfactors}, involutions have fixed point spaces of dimension at most $2\times \floor{4\times 24/5}\times 3 + 1\times 3+\floor{4\times 24/5}\times 1 = 136$, and elements of order 3 have eigenspaces of dimension at most $2\times \floor{3\times 24/4}\times 3 + 1\times 3+\floor{3\times 24/4}\times 1 = 129$.
 We compute that $i_2(G/F(G)) < 2^{33.72}$ and $i_3(G/F(G)) < 2^{40}$. So if $G$ has no regular orbit on $V$,
\[
r^{171}\leq 2|\mathrm{PSp}_6(7).2|r^{91}+(2^{33.72} + 2^{40})r^{128}.
\]
This is false for $r=2$.
\end{proof}
%\subsubsection{$(n,q) = (6,5)$}
%DONEx2
\begin{proposition}
Proposition \ref{oddcharsprop} holds for $E(G)/Z(E(G)) \cong \mathrm{PSp}_{6}(5)$.
\end{proposition}
\begin{proof}
By \eqref{generalspeqn} and inspection of \cite{HM}, we only need to consider Weil modules with $d=62, 63$ for  $r \in \{ 4, 9, 16 \}$ due to a $b_5$ irrationality.
From Propositions \ref{speqn} and \ref{itoeig}, we see that elements  $g \in G$ of projective order 7,13 or 31 have $\dim C_V(g) \leq $ 18, 10 and 5 respectively. Furthermore, for $g \in G$ of order 5,  $\dim C_V(g) \leq 30$ unless $g$ is one of the 15624 transvections in $G$, in which case $\dim C_V(g) \leq 50$. In GAP we compute that all elements of order 2, 3 have $\alpha(x) \leq 4, 3$ respectively. 
%We summarise the numbers of prime order elements in $G$ in the following table.
%\begin{table}[h!]
%\begin{tabular}{cc}
%\toprule
%Order $r_0$ & Number of elts of order $r_0$ in $\mathrm{PSp}(6,5)$ \\ 
%\midrule
%2 & 254719375 \\ 
%3 & 5499012500 \\ 
%5 & 158681265624 \\ 
%7 & 3627000000000 \\ 
%13 & 439425000000 \\ 
%31 & 18427500000000 \\ 
%\bottomrule
%\end{tabular} 
%\end{table}
Let $H=G/F(G)$. If $G$ has no regular orbit on a Weil module $V$ in either characteristic 2 or 3,
	\[
	\frac{1}{2} r^{62} \leq i_2(H) r^{\floor{3\times 63/4}}+ i_3(H) r^{\floor{2\times 63/3}}+ i_5(H)r^{30} + 15624r^{50}+i_7(H)r^{18}+i_{13}(H)r^{10} + i_{31}(H)r^5
	\]
	and this is false for $r \geq 4$.
\end{proof}
%\subsubsection{$(n,q) = (6,3)$}
%DONEx2
\begin{proposition}
Proposition \ref{oddcharsprop} holds for $E(G)/Z(E(G)) \cong \mathrm{PSp}_{6}(3)$.
\end{proposition}
\begin{proof}
We summarise information about elements of prime order in $H=\mathrm{PSp}_6(3).2$ in Table \ref{s63-counts}.
\begin{table}[h!]
\begin{tabular}{ccc}
\toprule
$r_0$ & $i_{r_0}(\mathrm{PSp}_6(3).2)$ & $\alpha(x) \leq$ \\ 
\midrule
2 & 605151 & 4 (2A), 3 (2B), 5 (2C)  \\ 
3 & 5307848 & 6 (3A/3B), 3 (o/w) \\ 
5 & 38211264 & 2 \\  
7 & 327525120 & 2 \\ 
13 & 705438720 & 2 \\ 
\midrule
Total & 1077088103& \\
\bottomrule
\end{tabular} 
\caption{Elements of prime order in $\mathrm{PSp}_6(3).2$. \label{s63-counts}}
\end{table}
Therefore, if $G$ has no regular orbit on the $d$-dimensional irreducible module $V$,
{\small
\[
\frac{1}{2} r^d\leq (i_5(H)+i_7(H)+i_{13}(H))r^{\floor{d/2}}+ i_3(H)r^{\floor{2d/3}}+ (|3A|+|3B|) r^{\floor{5d/6}}+ |2A|r^{\floor{3d/4}} +|2B|^{\floor{2d/3}}+ |2C|r^{\floor{4d/5}}
\]
}
Examining \cite{HM}, we see that this is false for all non-trivial absolutely irreducible modules for $E(G)$, except for $V$ with $d=78$ and $r=2$, or $d=13,14$ for $r\leq 93$.
If $V= V_{78}(2)$, then from a GAP construction and the Brauer character we deduce that every element $g$ of projective prime order has $\emax(g) \leq 46$ on $V$. Therefore if $G$ has no regular orbit on $V$, $r^{78} \leq 2\times 1077088103 r^{46}$, which is false.

Now suppose $V = V_{14}(r)$, so $E(G)=\mathrm{Sp}_6(3)$. The Brauer characters of these representations contain a $z_3$ irrationality, so we only need to consider $r= 7, 13,25, 49$.
	We bound the eigenspace dimensions of unipotent elements $g$ of order $r_0$ by restricting $V$ to a subgroup $K<E(G)$. When $p=7,13$, we restrict to the maximal subgroup $K=\mathrm{Sp}_2(27).3$. Comparing the Brauer character tables of $K$ and $E(G)$, we see that $V$ remains irreducible when restricted to $K$. Moreover, $K$ contains a Sylow $p$-subgroup of $E(G)$ in both cases. We are able to construct the module $V \downarrow K$ using Magma, and we find that $\dim C_V(g_p)=2$ for $p=7,13$. If instead $p=5$, we determine that $\alpha(g_5)=2$ so $\dim C_V(g_5)\leq 7$. 
	
%Now suppose $V = V_{14}(r)$. The Brauer characters of these modules contain a $z_3$ irrationality, so we only need to consider $r= 7, 13,25, 49$.
%We bound the eigenspace dimensions of unipotent elements $g$ of order $r_0$ by restricting $V$ to a subgroup $K$. We summarise our findings in Table \ref{s63-case-analysis}.
%\begin{table}[h!]
%\begin{tabular}{ccc}
%\toprule
%Characteristic $r_0$ & $K$ & $\dim C_V(g) \leq$ \\ 
%\midrule
%5 & $\mathrm{PSp}_2(3) \times \mathrm{PSp}_4(3)$ & 8 \\ 
%7 & $\mathrm{SL}_2(13)$ & 2 \\ 
%13 & $\mathrm{SL}_3(3).2$ & 2 \\ 
%\bottomrule
%\end{tabular} 
%\caption{Upper bounds for $\dim C_V(g)$ for $V=V_{14}(r)$ obtained by restricting to $K<G$. \label{s63-case-analysis}}
%\end{table}
	Therefore, if $G$ has no regular orbit on $V = V_{14}(13)$,

	\begin{align*}
	r^{14} \leq &|2A|(r^{10}+r^4)+ |2B|(2r^7)+(|3A|+|3B|)(r^9+r^5)+|3C|(2r^6+r^2)\\
	&+|3D|(r^8+2r^3)+(|3E|+|3F|)(r^6+r^5+r^3)+|3G|(r^6+2r^4)+ i_5(H)(4r^3+r^2)\\
	&+i_7(H)(7r^2)+i_{13}(H)r^2.
	%& 7371 (r^{10} + r^4) + 189540 (2 r^7) + (728) (r^9 + r^5) +  21840 (2 r^6 + r^2) +  43680 (r^8 + 2 r^3) \\
	%& + (524160) (r^6 + r^5 + r^3) +  4717440 (r^6 + 2 r^4) + 38211264 (4 r^3 + r^2) + 327525120 (7 r^2) +  352719360 r^2\\
	\end{align*}
This is false for $r=13$, and the calculations for $r=25,49$ are similar. If $r=7$, then using a construction in Magma, we show that there is no regular orbit and $b(G)=2$ by Proposition \ref{fieldext}.
Now suppose that $V$ is a 13-dimensional Weil module. Again, there is a $z_3$ irrationality, so $r \in \{4,7,13,16,25,49,64\}$. Table \ref{s63-case-analysis2} summarises some fixed point space bounds for unipotent elements. The values were obtained from constructions of the relevant modules in GAP.

\begin{table}[h!]
\begin{tabular}{@{}cc@{}}
\toprule
Characteristic $r_0$ & $\dim C_V(g) \leq$ \\ \midrule
2 & 9 (2A), 7 (2B) \\
5 & 3 \\
7 & 2 \\
13 & 1 \\ \bottomrule
\end{tabular}
\caption{Bounds on $\dim C_V(g)$ for unipotent elements $g$ acting on $V_{13}(r)$.}
\label{s63-case-analysis2}
\end{table}
 If $G$ has no regular orbit on $V=V_{13}(13)$, then 
\begin{align*}
r^{13} \leq& 705438720r+ |2A|(r^8+r^5)+|2B|(r^7+r^6)+|3A|(r^9+r^4)+|3B|(r^9+r^4)+|3C|(2r^6+r)\\
&+|3D|(r^7+2r^3)+(|3E|+|3F|)(r^6+r^4+r^3)+|3G|(r^5+2r^4)+|5A|(4r^3+r)+|7A|(6r^2+r).
\end{align*}
This is false for $r=13$. Similar calculations give the results for $r \geq 16$. If $V = V_{13}(4)$, then $|V|<|G|$ and by Proposition \ref{fieldext}, $b(G)=2$. If $r=7$, then we compute using Magma that $\mathrm{PSp}_6(3)$ has a regular orbit on $V$, however $c\times \mathrm{PSp}_6(3)$ does not for $c\in \{2,3,6\}$. In these cases, $b(G)=2$ by the result for $r=49$ and Proposition \ref{fieldext}.
\end{proof}
%\subsubsection{$(n,q) = (4,13), (4,17)$}
%Donex2
\begin{proposition}
Proposition \ref{oddcharsprop} holds for $E(G)/Z(E(G)) \cong \mathrm{PSp}_{4}(q)$, $q\in\{13,17,19\}$.
\end{proposition}
\begin{proof}
If $q=19$, then by \eqref{generalspeqn} and Proposition \ref{d2symp}, we only need to consider Weil modules for $r=2$, but these are not realised over $\mathbb{F}_2$ due to a $b_{19}$ irrationality in the Brauer character. Now let $q=17$.
There are 62466181283 elements of prime order in $\mathrm{PSp}_4(17).2$ and of these, 24304611 are involutions and 48492960768 are elements of order 29. In GAP we compute that  $\alpha(g_{29}) =2$. Therefore, if $G$ has no regular orbit on some irreducible module $V$ in cross-characteristic, 
\[
\frac{1}{2} r^d \leq (62466181283 - 48492960768) r^{\floor{3d/4}}+ 
  48492960768 r^{\floor{d/2}} + 24304611 r^{\floor{4 d/5}}.
\]
This is false for $d=d_1(E(G))$ by Proposition \ref{mindegprop}, and so for all $V$ by Proposition \ref{tools}\ref{crude}.
Similarly, the number of prime order elements in $\mathrm{PSp}_4(13).2$ is 6986011811, and of these, 4883931 are involutions and 3224422656 are elements of order 17. In GAP, we compute that $\alpha(g_{17})=2$. Therefore, if $G$ has no regular orbit on an irreducible module $V$,
\[
\frac{1}{2} r^d \leq (6986011811 - 3224422656) r^{\floor{3 d/4}}+ 
 3224422656 r^{\floor{d/2}} + 4883931 r^{\floor{4 d/5}}.
\]
By Propositions \ref{d2symp} and \ref{tools}\ref{crude}, we see that this is false for $d=d_1(E(G))$ by Proposition \ref{mindegprop}, and so for all $V$.
\end{proof}
%\subsubsection{$(n,q) = (4,11)$}
%DONEx2
\begin{proposition}
Proposition \ref{oddcharsprop} holds for $E(G)/Z(E(G)) \cong \mathrm{PSp}_{4}(11)$.
\end{proposition}
\begin{proof}
	The number of prime order elements in $H = \mathrm{PSp}_4(11).2$ is 3646870787 and of these, 3162456000  are elements of order 61 which we compute have $\alpha(x) = 2$. We also have $i_2(H)= 1800843$. Therefore, if $G$ has no regular orbit on  $V= V_d(r)$, then
\[
r^d \leq 2i_{61}(H) r^{\floor{d/2}} + 
2i_2(H) r^{\floor{4d/5}}+ (3646870787 - i_{61}(H) - 
i_2(H))( r^{\floor{3d/4}}+r^{\ceil{d/4}}).
\]
This is false for all $V$, except for the 60- and 61- dimensional Weil modules for $r=3$. There are 3247640 elements of order 3 in $\mathrm{PSp}_4(11)$. From the Brauer character, we deduce that $\emax(g_{3'})\leq 36$ on a Weil module. Therefore, if $G$ has no regular orbit on either of the Weil modules $V$,
\[
r^{60} \leq 2\times 3643623147r^{36} + 3247640r^{\floor{3\times 61/4}}.
\]
This is false for $r=3$.
 \end{proof}
%\subsubsection{$(n,q) = (4,9)$}
%DONEx2
\begin{proposition}
Proposition \ref{oddcharsprop} holds for $E(G)/Z(E(G)) \cong \mathrm{PSp}_{4}(9)$.
\end{proposition}
\begin{proof}
The number of prime order elements in $H=\mathrm{PSp}_4(9).2^2$ is 457397927 and of these 419904000 are elements of order 41 and have $\alpha(x) = 2$ from GAP computations. In addition, $i_2(H) = 610983$. Therefore, if $G$ has no regular orbit on $V=V_d(r)$,
\[
 r^d \leq 2 i_{41}(H)r^{\floor{d/2}}+ 2i_2(H)r^{\floor{4d/5}}+ 
  2 (457397927 - i_{41}(H) - i_{2}(H)) r^{\floor{3d/4}}.
\]
This is false for all $V$ except for the 40- and 41- dimensional Weil modules for $r\leq 5$. Note that $\mathrm{Sp}_4(9).2_1 < \mathrm{Sp}_8(3)$, and these groups have Weil representations of the same dimensions. Moreover, comparing the ordinary characters, we see that the Weil representations of $\mathrm{Sp}_8(3)$ remain irreducible when restricted to $\mathrm{Sp}_4(9).2_1$. 
% and use the relevant bounds on eigenspaces that we have already calculated. 
%First suppose $r=5$. If $G$ has no regular orbit on $V$ then 
%\begin{align*}
%r^{40} & \leq 2\times  457397927 r^9 + 2i_3(G) r^{27} + 2i_2(G) r^{28} +  i_5(G) r^{23} \\
%%&  =2\times  457397927 r^9 + 2\times 537920 r^{27} + 2\times 610983 r^{28} +  36345024 r^{23} 
%\end{align*}
%This is false for $r=5$.
If $r=4,5$, then we deduce that $G$ has a regular orbit on $V$ because the same is true for the Weil module of $ \mathrm{Sp}_8(3)$. 
This leaves the 40-dimensional Weil module for $r=2$. In this case, we have $G/F(G) \leq \mathrm{PSp}_4(9).2_1$. We compute that $i_2(\mathrm{PSp}_4(9).2_1)=368631$. Constructing the module in GAP, we find that elements $g$ in classes 2A or 2B have $\dim C_V(g) = 20$ and 24 respectively. Moreover, the two classes of field automorphisms in $\mathrm{PSp}_4(9).2_1$ also have $\dim C_V(g) = 20$ and 24. These observations, along with those from the Brauer character, imply that if $G$ has no regular orbit on $V$, then
	\begin{align*}
	r^{40} \leq &(|2A| + |2C|) r^{20} + (|2B| + |2D|) r^{24} + 
	|3A| (r^{22} + 2 r^9) + |3B| (r^4 + 2 r^{18}) + 
	|3C| (r^{10} + 2 r^{15}) \\
	&+ |3D| (r^{16} + 2 r^{12}) + 2i_5(H) r^9 + 
	i_{41}(H) (40 r).
	%  = &(298890 + 33210) r^{20} + (3321 + 33210) r^{24} + 
	%  3280 (r^{22} + 2 r^9) + 3280 (r^4 + 2 r^{18}) + 
	%  236160 (r^{10} + 2 r^{15}) \\
	%  &+ 295200 (r^{16} + 2 r^{12}) + 2\times 36345024 r^9 + 
	%  2\times 419904000 (r)
	\end{align*}
	This is false for $r=2$ and so $G$ has a regular orbit on $V$.
\end{proof}
%\subsubsection{$(n,q) = (4,7)$}
%DONEx2
\begin{proposition}
Proposition \ref{oddcharsprop} holds for $E(G)/Z(E(G)) \cong \mathrm{PSp}_{4}(7)$.
\end{proposition}
\begin{proof}
 We summarise some information about elements of prime order in $H = \mathrm{PSp}_4(7).2$ in Table \ref{s47-counts}.
\begin{table}[h!]
\begin{tabular}{ccc}
\toprule
$r_0$ & $i_{r_0}(H)$ & $\alpha(g_{r_0})\leq$ \\ 
\midrule 
2 & 122451 & 3 (2B/2D), 5 (2A/2C) \\ 
3 & 274400 & 3 \\ 
5 & 5531904 & 2 \\ 
7 & 5764800 & 2 (7E/7F), 4 (7A/7B), 3 (o/w)\\ 
\bottomrule 
\end{tabular} 
\caption{Information about elements of prime order in $\mathrm{PSp}_4(7).2$. \label{s47-counts}}
\end{table}
Therefore, if $G$ has no regular orbit on an irreducible module $V = V_d(r)$,
\begin{align*}
\frac{1}{2} r^d  &\leq (|2B|+|2D|)r^{\floor{2d/3}}+ (|2A|+|2C|)r^{\floor{4d/5}}+ i_3(H)r^{\floor{2d/3}}+ i_5(H)r^{\floor{d/2}}+(|7E|+|7F|)r^{\floor{d/2}}\\
&+ (|7A|+|7B|)r^{\floor{3d/4}}+  (|7C|+|7D|)r^{\floor{2d/3}}.
\end{align*}
From \cite{HM} we see that we need only consider the 24-dimensional Weil modules for $r=2,4$.  When $r=4$, from the Brauer character and constructing the module in GAP, we see that if $G$ has no regular orbit on $V$, 
\begin{align*}
r^{24} \leq & |3A| (3 r^8) +|3B|( r^{10} + 2 r^7) + |5A|(4 r^5+r^4) +   (|7A|+|7B|)(3 r^7 + r^3) + |7C|(6 r^4) \\
	& + |7D| (r^6 + 6 r^3) +   (|7E|+|7F|)(3 r^4 + 4 r^3) + |2A| r^{15} + |2B| r^{12}.
	%  &137200 (r^9 + 2 r^8 + r^{11} + 2 r^7) + 5531904 (5 r^5) +  2*1200 (3 r^7 + r^4) + 50400 (r + 6 r^4) + 67200 (r^7 + 6 r^3) +   2*2822400 (4 r^4 + 3 r^3) + 1225 r^{15} + 51450 r^{12}.\\
	 \end{align*}
 This is false for $r=4$. 
 When $r=2$, $|V|<|G|$ and by Proposition \ref{fieldext}, $b(G)=2$.
 \end{proof}
%\subsubsection{$(n,q) = (4,5)$}
%DONE
\begin{proposition}
Proposition \ref{oddcharsprop} holds for $E(G)/Z(E(G)) \cong \mathrm{PSp}_{4}(5)$.
\end{proposition}
\begin{proof}
We summarise information about elements of prime order in $H =\mathrm{PSp}_4(5).2$ in Table \ref{s45-counts}.
\begin{table}[h!]
\begin{tabular}{ccc}
\toprule
$r_0$ & $i_{r_0}(H)$ & $\alpha(x)\leq$ \\ 
\midrule
2 & 16875 & 5 (2A/2C), 3 (2B/2D) \\ 
3 & 26000 & 3 \\ 
5 & 390624 & 4 (5A/5B), 3 (5C/5D), 2 (5E/5F)  \\ 
13 & 1080000& 2 \\ 
\bottomrule 
\end{tabular} 
\caption{Elements of prime order in $H =\mathrm{PSp}_4(5).2$. \label{s45-counts}}
\end{table}
Therefore, if $G$ has no regular orbit on an irreducible module $V$,
\begin{align*}
\frac{1}{2}r^d \leq & (|2B|+|2D|)r^{\floor{2d/3}}+ (|2A|+|2C|)r^{\floor{4d/5}}+ i_3(H)r^{\floor{2d/3}}+ i_{13}(H)r^{\floor{d/2}}+(|5E|+|5F|)r^{\floor{d/2}}\\
&+ (|5A|+|5B|)r^{\floor{3d/4}}+  (|5C|+|5D|)r^{\floor{2d/3}}.
\end{align*}
It follows from \cite{HM} that we only need to consider 12- and 13-dimensional Weil modules for  $r = 4,9,16$, and $r=9$ respectively, as well as $d=40$ for $r=2$.
We begin by supposing $V= V_{40}(2)$. Examining the corresponding Brauer character, we see that $\emax(g_{2'}) \leq 16$, and constructing the module in GAP, we find that elements in class 2C have a 26-dimensional fixed point space, and all other involutions $g$ have $\dim C_V(g) \leq 24$. Therefore, if $G$ has no regular orbit on $V$,
\[
r^d \leq 2(1080000+390624+26000)r^{16}+ 16875r^{24}+ 300r^{26}.
\]
This is false for $r=2$, so $G$ has a regular orbit on $V$. 
Now suppose that $V = V_{13}(9)$. Examining the Brauer character and constructing the modules in GAP, we find that if $G$ has no regular orbit on one of the two 13-dimensional Weil modules,
\begin{align*}
r^{13} \leq & |2A| (r^9 + r^4) + |2B| (r^7 + r^6) + (|5A|+|5B|) (r^3 + 2 r^5) +   |5C|(4 r^3 + r) + |5D| (r^5 + 4 r^2) \\
& + (|5E|+|5F|)(3 r^3 + 2 r^2) +  (|13A|+|13B|) (13 r) + (|3A|+|3B|) r^5
%& 325 (r^9 + r^4) + 9750 (r^7 + r^6) + 2\times 312 (r^3 + 2 r^5) +  6240 (4 r^3 + r) + 9360 (r^5 + 4 r^2) + 2*187200 (3 r^3 + 2 r^2) + 1080000 (13 r) + 26000 r^5.\\
\end{align*}
This is false for $r=9$.

Now suppose that $V$ is one of the 12-dimensional Weil modules for $r=4,9,16$.
If $r=9$ and $G$ has no regular orbit on $V$, then from the Brauer character and constructions in GAP,
\begin{align*}
r^{12} \leq & |2A|(2 r^6) + |2B| (2 r^6) + (|5A|+|5B|) (r^2 + 2 r^5) + 
  |5C| (4 r^3) + |5D| (r^4 + 4 r^2)\\
  & +  (|5E|+|5F|)(2 r^3 + 3 r^2) +   (|13A|+|13B|)  (12 r) + (|3A|+|3B|) r^7.
\end{align*}
This is false for $r=9$. If $r=4$ or 16 then we construct the module in GAP and find a representative of a regular orbit.
\end{proof}
\subsection{Even characteristic}
We now consider symplectic groups over finite fields of even characteristic. In particular, we set out to prove the following.
%DONEx2 
%TODO G+S prop should go here, also prop about #transvections
\begin{proposition}
\label{evencharsprop}
Theorem \ref{mainsymp} holds for $E(G)/Z(E(G)) \cong \mathrm{PSp}_n(q)$, $q$ even.
\end{proposition}
By Proposition \ref{mindegprop} we have $d_1(E(G))\geq (q^{n/2}-1)(q^{n/2}-q)/2(q+1)$ and also $\alpha(G) \leq n+1$ by Proposition \ref{alphas}. Applying Proposition \ref{tools}\ref{crude}, we see that $G$ has a regular orbit on $V$ except possibly when $q=2$ and $6\leq n\leq 10$, or $(n,q)=(4,4)$. We consider each of these cases in turn.
%\subsubsection{$(n,q)=(10,2)$} 
%DONEx2

\begin{proposition}
Theorem \ref{mainsymp} holds for $E(G)/Z(E(G)) \cong \mathrm{PSp}_{10}(2)$.
\end{proposition}
\begin{proof}
First let $E(G)/Z(E(G))\cong \mathrm{PSp}_{10}(2)$. We include information about elements of prime order in $H=G/F(G)$ in Table \ref{s102-info}, where the upper bounds of $\alpha(x)$ are computed using GAP.

\begin{table}[h!]
\begin{tabular}{@{}ccc@{}}
\toprule
$r_0$ & $i_{r_0}(H)$ & $\alpha(x) \leq $ \\ \midrule
2 & 1424671743 & 11 (transvection), 6 (o/w) \\
3 & 107929774592 & 6 \\
5 & 13789672833024 & 3 \\
7 & 4923662008320 & 2 \\
11 & 751977470361600 & 2 \\
17 & 486573657292800 & 2 \\
31 & 2401476437606400 & 2 \\ \bottomrule
\end{tabular}
\caption{Information about elements of prime order in $H=\mathrm{PSp}_{10}(2)$.}
\label{s102-info}
\end{table}

Therefore, if $G$ has no regular orbit on $V$, then by Proposition \ref{tools}\ref{crude},
\[
\frac{1}{2} r^d \leq (i_{7}(H)+i_{11}(H)+i_{17}(H)+i_{31}(H))r^{\floor{d/2}}+i_5(H)r^{\floor{2d/3}}+i_3(H)r^{\floor{5d/6}}+ i_2(H)r^{\floor{5d/6}}+ 1023r^{\floor{10d/11}}
\]
This is false for all $V$, so $G$ has a regular orbit in all cases.
\end{proof}
%Examining the character table, we see that we only need consider $d=155,187,341$, and we deduce from \ref{crude} that we only need consider $r\leq 11$. The number of prime order elements in $\mathrm{PSp}_{10}(2)$ is less than $2^{51.71}$, $i_2(G)< 2^{30.41}$, $i_3(G)< 3^{23.13}$, and the $i_k(G)< k^{18.80}$ for $k=5,7,11$.Examining the Brauer characters of the $155, 187$ and 341 dimensional modules we find that $s = \emax(g_{\{2,3,r_0\}'}) \leq 43,43$ and $85$ respectively. Therefore, if $G$ has no regular orbit on $V$, 
%\[
%r^d \leq 2\times 2^{51.71} r^s+2(i_2(G)+i_3(G))r^{\frac{7d}{8}}+ i_r(G) r^{\frac{7d}{8}}+4q^{10}2r^{\frac{10d}{11}}
%\]
%This is false for all $(d,r)$ except for $(155,3)$. In this case, we construct the module in GAP and find that $\emax(g_3)\leq 85$. Substituting this into the equation is false, so $G$ has a regular orbit on $V$.
%\subsubsection{$(n,q)=(8,2)$} 
%DONEx2

\begin{proposition}
Theorem \ref{mainsymp} holds for $E(G)/Z(E(G)) \cong \mathrm{PSp}_{8}(2)$.
\end{proposition}
\begin{proof}
Suppose that $E(G)/Z(E(G))\cong \mathrm{PSp}_{8}(2)$. There are 255 transvections in $G/F(G)$.  
By Propositions \ref{sympalphas} and \ref{tools}\ref{crude}, we only need to further consider modules listed in Table \ref{s82-cases}, which was compiled by inspecting \cite{HM}.
\begin{table}[h!]
\begin{tabular}{ccc}
\toprule
Dimension & Characteristic & $r \leq $ \\ 
\midrule
35  & $\neq 2$ & 156 \\ 
50  & 3 & 9 \\ 
51 & $\neq 2,3$ & 17 \\ 
85  & $\neq 2,3$ & 5 \\ 
118  & 3,5 & 3 \\ 
135 & $\neq 2,17$ & 3 \\ 
\bottomrule
\end{tabular} 
\caption{Absolutely irreducible $\mathrm{PSp}_8(2)$-modules to consider. \label{s82-cases}}
\end{table}

We summarise our analysis of the relevant modules $V=V_d(r)$ with dimension greater than 35 in Table \ref{s82-case-analysis}. Entries where $o(x) \nmid r$ were derived from the Brauer table if the entry is bold, and from Proposition \ref{tools}\ref{alphabound} otherwise. Entries where $o(x) \mid r$ were derived from Proposition \ref{tools}\ref{alphabound} if $d \geq 85$, and from a construction in GAP otherwise. Applying Proposition \ref{tools}\ref{qsgood}, we see that $G$ has a regular orbit on $V$ in all cases.

\begin{table}[h!]
\begin{tabular}{cccccccc}
\toprule
 & $o(x)=2$ & Transvections & $o(x) = 3$ & $o(x)=5$ & $o(x)=7$ & $o(x)=17$ \\ 
\midrule
\# elements & 1371135 & 255 & 16461440 & 171085824 & 1128038400 & 5573836800 \\ 
$V_{135}(3)$ & 115 & 120 & 115 & \textbf{31} & \textbf{31} & \textbf{31} \\  
$V_{118}(3)$ & 101 & 104 & 101 & \textbf{30 }&\textbf{ 30 }& \textbf{30} \\  
$V_{85}(5)$ & \textbf{64} & \textbf{64} & \textbf{43 }& 72 & \textbf{43 }& \textbf{43} \\  
$V_{51}(r)$, $r\in \{5,7,17\}$ & \textbf{36} &\textbf{ 36} & \textbf{21} & $\bm{11}$ & $\bm{9}$ & $\bm{3}$ \\  
$V_{50}(r)$ $r=3,9$ & \textbf{36} &\textbf{ 36} & 22 & \textbf{10} &\textbf{8} & \textbf{3} \\ 
\bottomrule
\end{tabular} 
\caption{Upper bounds on $\emax(g)$ for prime order $g$ acting on some absolutely irreducible modules of $\mathrm{PSp}_8(2)$.\label{s82-case-analysis}}
\end{table}
Finally, suppose that $V =  V_{35}(r)$ for $r \leq 156$. We will give the details here for this module in characteristic 3, the other cases are similar. Constructing the module in GAP, we find that elements in classes $3A, 3B, 3C, 3D$ have fixed point spaces of dimensions $21,15,13$ and 13 respectively. This, along with the Brauer character of the module allows us to deduce that if $G$ has no regular orbit on $V$, then 
\begin{align*}
r^{35} \leq & |2A|(r^{28}+r^7)+(|2B|+|2F|)(r^{19}+r^{16})+|2C|(r^{23}+r^{12})+(|2D|+|2E|)(r^{20}+r^{15})+|7A|(7r^5)\\
& + |5A|(r^{11}+4r^6)+|5B|(5r^7)+(|17A|+|17B|)(r^3+16r^2)+|3A|r^{21}+|3B|r^{15}+(|3C|+|3D|)r^{13}
\end{align*}
and this is false for $r\geq 3$.
\end{proof}
%\subsubsection{$(n,q)=(6,2)$} 

\begin{proposition}
Theorem \ref{mainsymp} holds for $E(G)/Z(E(G)) \cong \mathrm{PSp}_{6}(2)$.
\end{proposition}
\begin{proof}
We first compute that the number of prime order elements in $G/F(G)$ is at most $277199$. Therefore, by Propositions \ref{sympalphas} and \ref{tools}\ref{crude}, if $G$ has no regular orbit on $V=V_d(r)$,
\[
r^d \leq 2\times 277199 r^{\floor{ \frac{5d}{6}}} + 63 r^{\floor{ \frac{6d}{7}}} 
\]
After inspecting \cite{HM}, we find that we only need to further consider the absolutely irreducible modules that lie in Table \ref{s62-case-analysis} for $d\geq 21$, and Table \ref{s62smallcase} for $d<21$.

In both tables we also present information about eigenspaces of elements of projective prime order on these remaining modules. In Table \ref{s62-case-analysis}, upper bounds on $\emax$ for elements of projective prime order are given. The boldface entries were obtained from the corresponding Brauer characters, while the italicised entries were obtained from explicit constructions in GAP. The remainder were obtained from Proposition \ref{tools} \ref{alphabound}. 
In Table \ref{s62smallcase}, we give explicit eigenspace dimensions computed from the corresponding Brauer characters, and for elements whose order divides $r$, we obtain the dimension of their fixed point spaces from an explicit construction in GAP. Applying Proposition \ref{tools}\ref{qsgood} in each case, we see that $G$ has a regular orbit on $V$ except for the cases given in Table \ref{s62ros}, where the base size is computed from an explicit construction of the module in either GAP or Magma.  
%In particular,  for $V=V_7(3)$, $b(G)=3,4 $ for $G= \mathrm{PSp}_6(2)$,  $2\times\mathrm{PSp}_6(2)$ respectively. 
This concludes our analysis of almost quasisimple $G$ with $E(G)/Z(E(G))\cong \mathrm{PSp}_{6}(2)$.
\end{proof}

{\small
	\begin{table}[h!]
		%			\begin{tabular}{ccccccccccccccc}
		%				\toprule
		%				% & \multicolumn{14}{c}{V} \\
		%				% \cmidrule{2-15}
		%				& $V_{14}(3)$ & $V_8(3)$ & $V_8(9)$ & $V_8(5)$ & $V_8(7)$ & $V_7(3)$ & $V_7(3)$ & $V_7(9)$ & $V_7(27)$ & $V_7(81)$ & $V_7(5)$ & $V_7(25)$ & $V_7(7)$ & $V_7(49)$ \\
		%				\midrule
		%				$\left\lceil \frac{\log |G|}{\log |V|} \right\rceil$ & 1 & 2 & 1 & 2 & 1 & 2 & 2 & 1 & 1 & 1 & 1 & 1 & 1 & 1 \\
		%				$b(G)$ & 2 & 3 & 2 & 2 & 2 & 3 & 4 & 2 & 1 & 1 & 2 & 1 & 2 & 1\\
		%				\bottomrule
		%			\end{tabular}
		\centering 
		\begin{tabular}{@{}ccc@{}}
			\toprule
			& $\left\lceil \frac{\log |G|}{\log |V|} \right\rceil$ & $b(G)$ \\ \midrule
			$V_{14}(3)$ & 1 & 2 \\
			$V_8(3)$ & 2 & 3 \\
			$V_8(9)$ & 1 & 2 \\
			$V_8(5)$ & 2 & 2 \\
			$V_8(7)$ & 1 & 2 \\
			\bottomrule
		\end{tabular}
		\quad
		\begin{tabular}{@{}ccc@{}}
			\toprule
			& $\left\lceil \frac{\log |G|}{\log |V|} \right\rceil$ & $b(G)$ \\ \midrule
			$V_7(9)$ & 1 & 2 \\
			$V_7(5)$ & 2 & 2 \\
			$V_7(3)$ & 2 & 3 \\
			$V_7(3)$ & 2 & 4 \\
			&&\\
%			$V_{14}(3)$ & 1 & 2 \\
%			$V_8(3)$ & 2 & 3 \\
%			$V_8(9)$ & 1 & 2 \\
			\bottomrule
		\end{tabular}
%		\quad 
%		\begin{tabular}{@{}ccc@{}}
%			\toprule
%			& $\left\lceil \frac{\log |G|}{\log |V|} \right\rceil$ & $b(G)$ \\ \midrule
%			$V_8(5)$ & 2 & 2 \\
%			$V_8(7)$ & 1 & 2 \\
%			$V_7(3)$ & 2 & 3 \\
%			$V_7(3)$ & 2 & 4 \\
%			&&\\
%			\bottomrule
%		\end{tabular}
		\caption{Base sizes for small absolutely irreducible modules for  $\mathrm{PSp}_6(2)$.}
		\label{s62ros}
	\end{table}
}
	
	\begin{table}[h!]
		\centering \begin{tabular}{lcccccc}
			\toprule
			$E(G)$ & Module $V$ & $o(x)=2$ & Transvections & $o(x)=3$ & $o(x)=5$ & $o(x)=7$ \\
			\midrule
			$2.\mathrm{PSp}_6(2)$ & $V_{56}(3)$ & \textbf{28} & \textbf{28} & 46 & \textbf{12} & \textbf{8} \\
			$\mathrm{PSp}_6(2)$ & $V_{49}(3)$ & \textbf{28} & \textbf{34} & 40 & \textbf{10} & \textbf{7} \\
			$2.\mathrm{PSp}_6(2)$ & $V_{48}(3)$ & \textbf{24} & \textbf{24} & \textit{20} & \textbf{10} & \textbf{7} \\
			$2.\mathrm{PSp}_6(2)$ & $V_{48}(5)$ & \textbf{24} & \textbf{24} & \textbf{20} & 40 & \textbf{7} \\
			$\mathrm{PSp}_6(2)$ & $V_{35}(7)$ & \textbf{23} & \textbf{25} & \textbf{15} & \textbf{7} & 28 \\
			$\mathrm{PSp}_6(2)$ & $V_{35}(5)$ & \textbf{23} & \textbf{25} & \textbf{15} & 17 & \textbf{5} \\
			$\mathrm{PSp}_6(2)$ & $V_{35}(r)$, $r=3,9$ & \textbf{20} & \textbf{20} & \textit{15} & \textbf{7} & \textbf{5} \\
			$\mathrm{PSp}_6(2)$ & $V_{34}(r)$, $r=3,9$ & \textbf{22} & \textbf{24} & \textit{14} & \textbf{7} & \textbf{5} \\
			$\mathrm{PSp}_6(2)$ & $V_{27}(5)$ & \textbf{17} & \textbf{21} & \textbf{15} & \textit{7} & \textbf{4} \\
			$\mathrm{PSp}_6(2)$ & $V_{27}(r)$, $r=3,9$ & \textbf{17} & \textbf{21} & \textit{15} & \textbf{7} & \textbf{4} \\
			$\mathrm{PSp}_6(2)$ & $V_{26}(7)$ & \textbf{16} & \textbf{20} & \textbf{14} & \textbf{6} & 13 \\
			$\mathrm{PSp}_6(2)$ & $V_{21}(7)$ & \textbf{13} & \textbf{16} & \textbf{11} & \textbf{5} & \textit{3} \\
			$\mathrm{PSp}_6(2)$ & $V_{21}(5^k)$, $k\in \{1,2\}$ & \textbf{13} & \textbf{16} & \textbf{11} & \textit{5} & \textbf{3} \\
			$\mathrm{PSp}_6(2)$ & $V_{21}(3^k)$ & \textbf{12} & \textbf{15} & \textit{11} & \textbf{5} & \textbf{3} \\
			\midrule
			\# elements &  & 5040 & 63 & 16352 & 48384 & 207360\\
			\bottomrule
		\end{tabular}
		\caption{Upper bounds on $\emax$ of projectively prime order elements in $\mathrm{PSp}_6(2)$ for remaining modules with $d \geq 21$.}
		\label{s62-case-analysis}
	\end{table}

\begin{table}[h!]
\bgroup
\def\arraystretch{1.5}
\begin{tabular}{@{}c@{\hskip 0.5cm}c@{\hskip 0.5cm}c@{\hskip 0.5cm}c@{\hskip 0.5cm}c@{\hskip 0.5cm}c@{\hskip 0.5cm}c@{\hskip 0.5cm}c@{\hskip 0.5cm}c@{\hskip 0.5cm}c@{}}
\toprule
              & 2A      & 2B      & 2C      & 2D      & 3A         & 3B         & 3C         & 5A         & 7A         \\ \midrule
Size          & 63      & 315     & 945     & 3780    & 672        & 2240       & 13440      & 48384      & 207360     \\
$V_{15}(5^k)$ & (10,5)  & (11,4)  & (9,6)   & (8,7)   & $(5^3)$    & $(6^2,3)$  & $(7,4^2)$  & \textit{7} & $(3,2^6)$  \\
$V_{15}(7^k)$ & (10,5)  & (11,4)  & (9,6)   & (8,7)   & $(5^3)$    & $(6^2,3)$  & $(7,4^2)$  & $(3^5)$    & \textit{7} \\
$V_{14}(3^k)$ & (10,4)  & (10,4)  & (8,6)   & (8,6)   & \textit{6} & \textit{6} & \textit{6} & $(3^4,2)$  & $(2^7)$    \\
$V_{8}(3^k)$  & $(2^4)$ & $(4^2)$ & $(2^4)$ & $(4^2)$ & \textit{4} & \textit{4} & \textit{4} & $(2^4)$    & $(2,1^6)$  \\
$V_{8}(5^k)$  & $(2^4)$ & $(4^2)$ & $(2^4)$ & $(4^2)$ & $(4^2)$    & $(3^2,2)$  & $(4,2^2)$  & \textit{2} & $(2,1^6)$  \\
$V_{8}(7^k)$  & $(2^4)$ & $(4^2)$ & $(2^4)$ & $(4^2)$ & $(4^2)$    & $(3^2,2)$  & $(4,2^2)$  & $(2^4)$    & \textit{2} \\
$V_{7}(3^k)$  & $(6,1)$ & $(4,3)$ & $(5,2)$ & $(4,3)$ & \textit{5} & \textit{3} & \textit{3} & $(3,1^4)$  & $(1^7)$    \\
$V_{7}(5^k)$  & $(6,1)$ & $(4,3)$ & $(5,2)$ & $(4,3)$ & $(5,1^2)$  & $(3^2,1)$  & $(3,2^2)$  & 3& $(1^7)$    \\
$V_{7}(7^k)$  & $(6,1)$ & $(4,3)$ & $(5,2)$ & $(4,3)$ & $(5,1^2)$  & $(3^2,1)$  & $(3,2^2)$  & $(3,1^4)$  & \textit{1} \\ \bottomrule
\end{tabular}
\egroup
\caption{Eigenspace dimensions of projectively prime order elements for $E(G)/Z(E(G))\cong \mathrm{PSp}_{6}(2)$.}
\label{s62smallcase}
\end{table}
%\subsubsection{$(n,q)=(4,4)$} 
%DONE

\begin{proposition}
Theorem \ref{mainsymp} holds for $E(G)/Z(E(G)) \cong \mathrm{PSp}_{4}(4)$.
\end{proposition}
\begin{proof}
The number of prime order elements in $\mathrm{PSp}_{4}(4).4$ is 299199. Therefore, if $G$ has no regular orbit on the irreducible module $V=V_d(r)$,
\[
r^d\leq 2\times 299199r^{\floor{4d/5}}
\]
and this is false for all choices of $d$, $r$ except for those listed in Table \ref{s44-case-analysis}. In this table, we also provide a set of upper bounds for $\emax$ for elements of prime order. As before, the bold entries are obtained from the corresponding Brauer characters, the italicised entries are obtained from a construction in GAP, and the remainder originate from Proposition \ref{tools}\ref{alphabound}. Applying Proposition \ref{tools}\ref{qsgood} or \ref{eigsp1} in each case, we deduce the existence of a regular orbit for all choices of $G$ preserving each module $V$. The only exception is $V = V_{18}(3)$. Here we compute in GAP that $\mathrm{PSp}_4(4)$ has 14 regular orbits on $V$. If $\mathrm{PSp}_4(4).4$ has no regular orbit on $V$,
	then each vector in a regular orbit of $\mathrm{PSp}_4(4)$ must lie in the fixed point spaces of elements in class 2D. Therefore,  $14|\mathrm{PSp}_4(4)| \leq |2D|(3^9)$. This is false, so $\mathrm{PSp}_4(4).4$ has a regular orbit on $V$. 
\begin{table}[]
\begin{tabular}{@{}ccccccc@{}}
\toprule
 & 2A/2B & 2C & 2D & $o(x)=3$ & $o(x)=5$ & $o(x)=17$ \\ \midrule
\# elements & 255 (each) & 3825 & 1360 & 10880 & 52224 & 230400 \\
$V_{50}(3)$ & \textbf{30} & \textbf{30} & \textbf{30} & 37 & \textbf{30} & \textbf{30} \\
$V_{34}(3)$ & \textbf{22} & \textbf{22} & \textbf{22} & 25 & \textbf{10} & \textbf{10} \\
$V_{33}(5)$ & \textbf{21} & \textbf{21} & \textbf{21} & \textbf{13} & \textit{9} & \textbf{13} \\
$V_{18}(3^k)$, $k\leq 3$ & \textbf{12} & \textbf{10} & \textbf{9} & \textit{6} & \textbf{6} & \textbf{6} \\
$V_{18}(5^k)$, $k=1,2$ & \textbf{12} & \textbf{10} & \textbf{9} & \textbf{6} & \textit{6} & \textbf{6} \\
$V_{18}(7)$ & \textbf{12} & \textbf{10} & \textbf{9} & \textbf{6} & \textbf{6} & \textit{2} \\ \bottomrule
\end{tabular}
\caption{Bounds on $\emax$ for prime order elements in $\mathrm{PSp}_4(4)$.}
\label{s44-case-analysis}
\end{table}

%\subsubsection{$(n,q)=(4,2)$} 
%DONEx2
%Finally, if $E(G)/Z(E(G))\cong \mathrm{PSp}_{4}(2)$, then since
%$\mathrm{Sp}_4(2) \cong S_6$, this case was done by \cite{MR3500766}.
\end{proof}

\section{Proof of Theorem \ref{mainthm}: Orthogonal groups}

\subsection{Odd dimensional orthogonal groups}
%DONEx2

\begin{theorem}
\label{odddimoprop}
Suppose $G$ is an almost quasisimple group with $E(G)/Z(E(G))\cong \mathrm{P}\Omega_{2m+1}(q)$, with $m\geq 3$ and $q$ odd. Let $V=V_d(r)$, $(r,q)=1$ be a module for $G$, with absolutely irreducible restriction to $E(G)$. Also suppose that $(r, |G|)>1$. Then $G$ has a regular orbit on $V$.
\end{theorem}
\begin{proof}
Let $V$ be a $d$-dimensional irreducible $\mathbb{F}_rG$-module.
If  $q>3$, then combining Proposition \ref{tools}\ref{crude} with the bounds $|G/F(G)| <|\mathrm{Aut}(\mathrm{P}\Omega_{2m+1}(q))|<q^{2m^2+m+2}$, $\alpha(G) \leq 2m+1$ from Proposition \ref{gsorthog}, and $d_1(E(G)) \geq (q^{2m}-1)/(q^2-1)-2 $ from Proposition \ref{mindegprop}, we see that $G$ always has a regular orbit on $V$. Similarly Proposition \ref{tools}\ref{crude} also shows that if $q=3$ and $m>3$ with the same bounds on $|G|$ and $\alpha(G)$, and $d_1(E(G)) \geq (q^m-1)(q^m-q)/(q^2-1)$, then $G$ has a regular orbit on $V$.
	
	Finally, suppose $G$ has $\mathrm{soc}(G/F(G))  =\mathrm{P} \Omega_7(3)$. By Proposition \ref{gsorthog}, if $G$ has no regular orbit on $V$, then 
	\begin{equation}
	\label{o73}
	r^d \leq 2\times |\pom_7 (3).2|r^{\floor{\frac{5d}{6}}} +2 i_2(G/F(G))r^{\floor{\frac{6d}{7}}}.
	\end{equation}
	Using \cite{HM} we deduce that $G$ has a regular orbit on $V$ except possibly in one of the cases in Table \ref{po73table}, or if $E(G) = 3.\mathrm{P}\Omega_7(3)$ acting on $V_{27}(4^k)$, which we handle separately.
%\begin{table}[h!]
%\begin{tabular}{cccc}
%\toprule
%$d$ & $E(G)$ & Characteristic & $r\leq$ \\ 
%\midrule
%27 & $3.\pom _7(3)$ & $\neq 3$ & 74 \\ 
%78 & $\pom_7(3)$ & $\neq 3$ & 5 \\ 
%90 & $\pom _7(3)$ & $2$ & 4 \\  
%104 & $\pom _7(3)$ & 2 & 2 \\ 
%\bottomrule
%\end{tabular} 
%\caption{Cases left to consider for $E(G)/Z(E(G)) \cong \pom_7(3)$. \label{O73-cases}}
%\end{table}
%The range of $r$ in the last column of the table was determined by refining \eqref{o73} with the precise values of the number of prime order elements and involutions in $\pom_7(3).2$ and $3. \pom _7(3).2$. These were computed explicitly using GAP \cite{GAP4}.

For conciseness, we also summarise the analysis of all cases except the action of $3.\mathrm{P}\Omega_7(3)$ on $V_{27}(4^k)$ in Table \ref{po73table}. In the table, the bold entries indicate eigenspace bounds computed using the corresponding Brauer character, the italicised entries were computed from a construction in GAP, and the remaining entries originate from determining $\alpha(g)$ for $g$ in the collection using GAP \cite{GAP4}, and subsequently applying Proposition \ref{tools}\ref{alphabound}. Applying Proposition \ref{tools}\ref{qsgood} with the values given in Table \ref{po73table} shows that $G$ has a regular orbit on $V$ in all cases.
{\small
\begin{table}[h!]
\begin{tabular}{ccccccccccc}
\toprule
 & \multicolumn{10}{c}{Projective class}\\
 \midrule
 & 2A & 2B & 2C & 2D & 2E & 2F & $o(x)=3$ & $o(x)=5$ & $o(x)=7$ & $o(x)=13$ \\ 
\midrule
\# elements & 351 & 22113 & 331695 & 378 & 44226 & 265356 & 5307848 & 38211264 & 327525120 & 705438729 \\ 
 $V_{104}(2)$ & 89 & 78 & 78 & 89 & 78 & 78 & \textbf{50} & \textbf{21} &\textbf{ 15 }& \textbf{8} \\ 
$V_{90}(r)$, $r=2,4$ & 77 & 67 & 67 & 77 & 67 & 67 & \textbf{42} & \textbf{18} & \textbf{13} & \textbf{7} \\ 
$V_{78}(r)$, $r=2,4$ & \textit{56} & \textit{56} & \textit{56} & \textit{56} & \textit{56} & \textit{56} & \textbf{36} & \textbf{18} & \textbf{12} & \textbf{6} \\ 
$V_{78}(5)$  & \textbf{56} & \textbf{56} & \textbf{56 }& \textbf{56} & \textbf{56} & \textbf{56} & \textbf{36} & 65 & \textbf{36 }& \textbf{36} \\ 
 $V_{27}(13^k)$  & \textbf{21} & \textbf{17} & \textbf{17}& --&-- & --  & \textbf{15} & \textbf{15} & \textbf{15} & 13 \\ 
$V_{27}(7^k)$  & \textbf{21} & \textbf{17} & \textbf{17 }& --&-- & -- & \textbf{15} & \textbf{15} & 13 & \textbf{15} \\ 
$V_{27}(25^k)$ & \textbf{21} & \textbf{17} & \textbf{17 }& --&-- & -- & \textbf{15} & 18 & \textbf{15 }& \textbf{15} \\ 
\bottomrule
\end{tabular} 
\caption{Sizes and $\emax$ upper bounds of elements of collections of projective prime order elements in $\mathrm{P}\Omega_7(3).2$. \label{po73table}}
\end{table}
}

 Finally, suppose that $V = V_{27}(4^k)$ and recall that $E(G)=3.\mathrm{P}\Omega_7(3)$. Set $H=G/F(G)$. When we construct $V$ in GAP, we find that involutions in classes 2A, 2B and 2C have 21-, 17- and 5-dimensional fixed point spaces respectively. This, along with information from the corresponding Brauer character implies that if $G$ has no regular orbit on $V$,
\begin{align*}
r^{27} \leq & |2A|r^{21}+|2B|r^{17}+|2C|r^{15}+i_3(G)(3r^9)+|3B|(r^{15}+2r^6)+i_5(H)(r^7+4r^5)\\
&+i_7(H)(r^3+6r^4)+\tfrac{13}{12} i_{13}(H)r^3,
\end{align*}
and this is false for $r\geq 4$, so $G$ has a regular orbit on $V$.
\end{proof}
\subsection{Even dimensional orthogonal groups }

\begin{theorem}
\label{evendimoprop}
Suppose $G$ is an almost quasisimple group with $E(G)/Z(E(G))\cong \mathrm{P}\Omega^\epsilon_{2m}(q)$ with $m\geq 4$. Let $V=V_d(r)$, $(r,q)=1$ be an irreducible module for $G$, with absolutely irreducible restriction to $E(G)$. Also suppose that $(r, |G|)>1$. Then one of the following holds.
\begin{enumerate}
\item $b(G)  = \lceil \log |G| / \log |V| \rceil$, or 
\item  $E(G)/Z(E(G)) \cong \pom_8^+(2)$, $d=8$, $r \in \{3,5,25,27\}$ and $b(G)  = \lceil \log |G| / \log |V| \rceil+1$.
\end{enumerate}
\end{theorem}
By Proposition \ref{mindegprop}, if $(\ep,m,q) \neq  (+,4,2)$, then we have $d_1(E(G)) \geq \frac{(q^m-1)(q^{m-1}-q)}{q^2-1}-1$.

Therefore, if $G$ has no regular orbit on $V$ then by Propositions \ref{alphas} and \ref{tools},
\[
r^{\frac{(q^m-1)(q^{m-1}-q)}{q^2-1}-1} \leq 2|G/F(G)|r^{\floor{\frac{2m-1}{2m}d}}
\]
This inequality is false for all valid choices of $(m,q,r, \ep)$ except for $(5,2 ,r_1,\pm)$, $(4,3, r_2,\pm)$, and $(4,2,r_3,-)$, where $r_1\leq 7$, $r_2 =2$ and $r_3 \leq 167$.

We will consider each pair $(m,q)$ individually.
%\subsubsection{$(m,q) = (5,2)$} 
%DONEx2
\begin{proposition}
Theorem \ref{evendimoprop} holds if  $E(G)/Z(E(G)) \cong \mathrm{P} \Omega^\ep_{10}(2)$.
\end{proposition}
\begin{proof}
We first determine the range of $d$ to consider.  The number of prime order elements $i_P(\pom_{10}^\epsilon (2).2)$ in $\pom_{10}^\epsilon (2).2$ is 5642583264255 if $\ep=+$ and 4547966108159 if $\ep = -$. The number of reflections $i_R(\pom_{10}^\epsilon (2).2)$ is 496 and 528 for $\ep = +, -$ respectively. If $G$ has no regular orbit on $V$ then by Proposition \ref{gsorthog},
\[
r^d \leq 2i_P(\pom_{10}^\epsilon (2).2) r^{\floor{7d/8}}+ 2i_R(\pom_{10}^\epsilon (2).2)r^{\floor{9d/10}}
\]
So we only need consider $d\leq 216$ for $r=3$, and $G$ has a regular orbit on the remaining modules. From \cite{HM} we find that we only need to consider two modules for each $\epsilon$, one each of degrees $154+\epsilon$ and $186-(1+\ep)/2$. Examining the Brauer characters, we determine that $\emax(g_{3'}) \leq 118+(1-\ep)$ and $136-(1+\ep)/2$ respectively. Moreover, by Proposition \ref{gsorthog}, all elements of order 3 have $\alpha(x)\leq 8$. Therefore, if $G$ has no regular orbit on the 153-dimensional module when $\epsilon=-$  then
\[
r^{153} \leq 2 i_P(G/F(G)) r^{120} + i_3(G/F(G)) r^{\floor{7\times 153/8}}
\]
which is false for $r=3$. Similar calculations give the desired result for the other three modules.

\end{proof}
%\subsubsection{$(m,q) = (4,3)$}
%DONEx2
\begin{proposition}
Theorem \ref{evendimoprop} holds if $E(G)/Z(E(G)) \cong \mathrm{P} \Omega^\ep_{8}(3)$.
\end{proposition}
\begin{proof}
Recall that we only need to consider $r=2$. By Proposition \ref{mindegprop}, $d_1(E(G))\geq 248$. Using GAP, we compute that every element $x \in G/F(G)$ of prime order at least 5 has $\alpha(x) =2$. We also determine that the number $i_R(G/F(G))$ of reflections in $G/F(G)$ is at most 2214 and 6480 for $\epsilon=-,+$ respectively. Denote the number of prime order elements in $H=G/F(G)$ by $i_P(H)$. We compute that $i_P(H)< 2^{41.56}$.
If $G$ has no regular orbit on $V$ then 
\[
\frac{1}{2}r^d  \leq (i_P(H)-i_2(H)-i_3(H))r^{\floor{d/2}}+ (i_2(H)+i_3(H))r^{\floor{6d/7}} + i_R(H)r^{\floor{7 d/8}}
\]
This is false for all $d \geq 245$ (cf. Proposition \ref{mindegprop}) with $r=2$.
\end{proof}
%\subsubsection{$(m,q) = (4,2)$} 
\begin{proposition}
Theorem \ref{evendimoprop} holds if $E(G)/Z(E(G)) \cong \mathrm{P} \Omega^\ep_{8}(2)$.
\end{proposition}
\begin{proof}
First suppose that $E(G)/Z(E(G)) \cong \pom ^-_8(2)$.
	The number of prime order elements in $\pom_{8}^-(2).2$ is 57548943; of these, 112591 are involutions, including 136 transvections. If $G$ has no regular orbit on $V$ then
	\[
	r^d \leq 2\times 57548943 r^{\floor{6d/7}}+ 2\times 136r^{\floor{7d/8}}.
	\]
	Thus $G$ has a regular orbit on $V$, or $V$ lies in Table \ref{o82-case-analysis}.
%\begin{table}[h!]
%\begin{tabular}{cccc}
%\toprule 
%$d$ & QS group & Characteristic & Max $r$ to consider \\ 
%\midrule
%33 & $\pom_8^-(2)$ & 7 & 7 \\ 
%34 & $\pom_8^-(2)$ & $\neq 2,7$ & 27 \\ 
%50 & $\pom_8^-(2)$ & 3 & 9 \\ 
%51 & $\pom_8^-(2)$ & $\neq 2,3$ & 7 \\ 
%\bottomrule
%\end{tabular} 
%\caption{Cases to consider for $\mathrm{soc}(G/F(G)) \cong \pom ^-(8,2)$\label{o82-cases}}
%\end{table}

We also summarise our analysis of the remaining cases in Table \ref{o82-case-analysis}. In all cases, the bounds on $\emax(x)$ are obtained from the Brauer character if $o(x) \nmid r$, and from constructing the corresponding module in GAP otherwise.

\begin{table}[h!]
		\centering \begin{tabular}{c@{\hskip 0.5cm}cccccc}
			\toprule
			& \multicolumn{6}{c}{Projective prime order}\\
			\midrule
			& $o(x) = 2$ & Transvections & $o(x)=3$ & $o(x)=5$ & $o(x)=7$ & $o(x)=17$ \\ 
			\midrule
			\# elements & 112455 & 136 & 490688 & 1096704 & 9400320 & 46448640 \\ 
			$V_{51}(5)$ & 36 & 36 & 21 & 11 & 21 & 21 \\ 
			$V_{51}(7)$ & 36 & 36 & 21 & 11 & 9 & 21 \\ 
			$V_{50}(3^k)$, $r\in\{1,2\}$ & 36 & 36 & 22 & 10 & 10 & 10 \\ 
			$V_{34}(3^k)$, $k\in\{1,2,3\}$ & 22 & 27 & 20 & 10 & 10 & 10 \\ 
			$V_{34}(5^k)$, $k\in\{1,2\}$ & 22 & 27 & 20 & 10 & 20 & 20 \\ 
			$V_{33}(7)$ & 21 & 26 & 19 & 19 & 5 & 19 \\ 
			\bottomrule 
		\end{tabular} 
		\caption{Upper bounds on maximum eigenspace dimensions on absolutely irreducible modules of $\mathrm{P}\Omega^-_8(2).2$.\label{o82-case-analysis}}
	\end{table}
Using the information in Table \ref{o82-case-analysis} to apply Proposition \ref{tools}\ref{eigsp2}, we find that $G$ has a regular orbit on $V$ in all cases.

Now suppose that $E(G)/Z(E(G)) \cong \pom ^+_8(2)$.
The number of prime order elements in $\pom_8^+(2).S_3$ is 28815119 and so if $G$ has no regular orbit on $V$, then
\[
r^d \leq 2\times 28815119r^{\floor{7d/8}}.
\]
Using \cite{HM}, we give a complete list of the modules we need to consider in Table \ref{o8+2cases}. In Table \ref{o8+2poe} we also give information on the number of elements of various prime orders in $\mathrm{Aut}(\pom_8^+(2))$.
\begin{table}[h!]
\noindent\begin{minipage}{.6\textwidth}
\centering
\begin{tabular}{cccc}
\toprule
$d$ & $E(G)$ & Characteristic & $r\leq$ \\ 
\midrule
104 & $2.\pom_8^+(2)$ & 3,5 & 3 \\
83 & $\pom_8^+(2)$ & 5 & 5 \\ 
56 & $2.\pom_8^+(2)$ & $\neq 2$ & 13 \\  
50 & $\pom_8^+(2)$ & $\neq 2,3$ & 13 \\  
48 & $\pom_8^+(2)$ & 3 & 9 \\ 
35 & $\pom_8^+(2)$ & $\neq 2$ & 31 \\ 
28 & $\pom_8^+(2)$ & $\neq 2$ & 96 \\ 
8 & $2.\pom_8^+(2)$ & $\neq 2$ & 86445357 \\ 
\bottomrule
\end{tabular} 
 \captionof{table}{Cases left to consider when $E(G)/Z(E(G)) \cong \pom ^+_8(2)$. \label{o8+2cases}}
    \end{minipage}%
    \begin{minipage}{.4\textwidth}
\centering
\begin{tabular}{cc}
\toprule
Order & \# elements \\ 
\midrule
2 & 183375 \\ 
3 & 2006720 \\ 
5 & 1741824 \\ 
7 & 24883200 \\ 
\midrule
Total & 28815119 \\ 
\bottomrule
\end{tabular} 
  \captionof{table}{Summary of prime order elements in  $\pom_8^+(2).S_3$. \label{o8+2poe}}
    \end{minipage} 
\end{table}
Define $\epsilon_i$ to be equal to 1 if $(i,r)>1$ and 0 otherwise.
We summarise our analysis of the modules listed in Table \ref{o8+2cases} in Table \ref{o82case-analysis}. The bold entries were obtained from the corresponding Brauer characters of the absolutely irreducible modules. The italicised entries were obtained from an explicit construction of the module in GAP. The entries that are both bold and italicised were obtained using the former method when $(o(x),r)=1$, and the latter otherwise. 
%The underlined entries were computed by restricting $V_{28}=V_{28}(3^k)$ to the maximal subgroup $H=U_3(3):2 \times S_3 \leq \pom_8^+(2).S_3$ and applying Propositions \ref{tools}\ref{alphabound} and \ref{compfactors}. 
The remaining entries were computed using Proposition \ref{tools}\ref{alphabound}. Applying Proposition \ref{tools}\ref{qsgood} to each of the modules in the table shows that $G$ has a regular orbit on $V$ for all valid choices of $G$.

	\begin{table}[h!]
		
		\centering \begin{tabular}{lcccccccccm{1.5cm}}
			\toprule
			Class & Size     & $V_{104}(3)$ & $V_{83}(5)$ & $V_{56}(r)$          & $V_{50}(r)$ & $V_{48}(r)$ & $V_{35}(r)$          & $V_{28}(3^k)$ & $V_{28}(5^k)$ or\\ 
		 & &  &  &  & && & & $V_{28}(7^k)$\\ 
			\midrule
			2A    & 1575     & \textbf{77}  & 72          & \textbf{35}          & \textbf{35} & \textbf{34} & \textbf{20}          & \textbf{16}   & \textbf{16}                    \\
			2B    & 3780     & \textbf{77}  & 72          & \textbf{35}          & \textbf{35} & \textbf{34} & \textbf{23}          & \textbf{16}   & \textbf{16}                    \\
			2C    & 3780     & \textbf{77}  & 72          & \textbf{35}          & \textbf{35} & \textbf{34} & \textbf{20}          & \textbf{16}   & \textbf{16}                    \\
			2D    & 3780     & \textbf{77}  & 72          & \textbf{35}          & \textbf{35} & \textbf{34} & \textbf{20}          & \textbf{16}   & \textbf{16}                    \\
			2E    & 56700    & \textbf{77}  & 72          & \textbf{35}          & \textbf{35} & \textbf{34} & \textbf{20}          & \textbf{16}   & \textbf{16}                    \\
			3A    & 2240     & 89           & 71          & $\bm{35-7\ep_3}$ & \textbf{35} & \textit{26} & \textit{\textbf{21}} & \textit{16}   & \textbf{16}                    \\
			3B    & 2240     & 89           & 71          & $\bm{35-7\ep_3}$ & \textbf{35} & \textit{26} & \textit{\textbf{15}} & \textit{16}   & \textbf{16}                    \\
			3C    & 2240     & 89           & 71          & $\bm{35-7\ep_3}$ & \textbf{35} & \textit{26} & \textit{\textbf{15}} & \textit{16}   & \textbf{16}                    \\
			3D    & 89600    & 89           & 71          & $\bm{35-7\ep_3}$ & \textbf{35} & \textit{26} & \textit{\textbf{15}} & \textit{12}   & \textbf{10}                    \\
			3E    & 268800   & 89           & 71          & $\bm{35-7\ep_3}$ & \textbf{35} & \textit{26} & \textit{\textbf{15}} & \textit{12}   & \textbf{10}                    \\
			5A    & 580608   & \textbf{77}  & 71          & $\bm{35-7\ep_5}$          & 25          & \textbf{10} & \textit{\textbf{11}} & \textbf{8}    & \textit{\textbf{8}}            \\
			5B    & 580608   & \textbf{77}  & 71          & \textbf{$\bm{35-7\ep_5}$} & 25          & \textbf{10} & \textit{\textbf{11}} & \textbf{8}    & \textit{\textbf{8}}            \\
			5C    & 580608   & \textbf{77}  & 71          & \textbf{$\bm{35-7\ep_5}$} & 25          & \textbf{10} & \textit{\textbf{11}} & \textbf{8}    & \textit{\textbf{8}}            \\
			7A    & 24883200 & \textbf{77}  & 71          & $\bm{35-7\ep_7}$ & 25          & \textbf{10} & \textit{\textbf{5}}  & \textbf{7}    & \textit{\textbf{4}}            \\
			2F    & 120      & \textbf{77}  & 72          & \textbf{35}          & 35          & \textbf{34} & \textbf{28}          & \textbf{21}   & \textbf{21}                    \\
			2F'    & 120      & \textbf{77}  & 72          & \textbf{35}          & 35          & \textbf{34} & \textbf{28}          & \textbf{21}   & \textbf{21}                    \\
			2F''    & 120      & \textbf{77}  & 72          & \textbf{35}          & 35          & \textbf{34} & \textbf{28}          & \textbf{21}   & \textbf{21}                    \\			
			2G    & 37800    & \textbf{77}  & 72          & \textbf{35}          & 35          & \textbf{34} & \textbf{20}          & \textbf{15}   & \textbf{15}                    \\
			2G'    & 37800    & \textbf{77}  & 72          & \textbf{35}          & 35          & \textbf{34} & \textbf{20}          & \textbf{15}   & \textbf{15}                    \\
			2G''    & 37800    & \textbf{77}  & 72          & \textbf{35}          & 35          & \textbf{34} & \textbf{20}          & \textbf{15}   & \textbf{15}                    \\
			3F    & 14400    & --           & --          & --                   & 35          & \textit{26} & --                   & \textit{14}    & \textbf{14}                    \\
			3F'   & 14400    & --           & --          & --                   & 35          & \textit{26} & --                   & \textit{14}     & \textbf{14}                    \\
			3G    & 806400   & --           & --          & --                   & 35          & \textit{26} & --                   & \textit{10 }     & \textbf{10}\\
			3G'    & 806400   & --           & --          & --                   & 35          & \textit{26} & --                   & \textit{10}      & \textbf{10}\\
			\bottomrule
		\end{tabular}
		\caption{Sizes and upper bounds on $\emax$ for conjugacy classes of elements of prime order in $\pom^+_8(2).S_3$.}
		\label{o82case-analysis}
	\end{table}

	Finally, suppose that $V$ is the $8$-dimensional irreducible module of $2.\pom_8^+(2)$, and note that according to the Brauer character tables given in \cite[p. 233-241]{modatlas}, $G/F(G)$ contains no graph automorphisms of order 3 (that is, triality automorphisms). 
If $r\leq 9$, then $|V|<|G|$ and there is no regular orbit. If $r=3$ then also $|V|^2 <|G|$, so $b(G)\geq 3$.
If $G$ has no regular orbit on $V$ in characteristics 5 or 7 then
\begin{align*}
r^8 \leq & ( |2A| + |2C|+|2D| + |2E|) (2 r^4) + |2B| (r^6 + r^2) + |3A|(r^6 + 2 r) + (|3B|+|3C|)(2 r^4)\\
   &+|3D| (r^2 + 2 r^3) + 
  |3E| (r^4 + 2 r^2) + |5A| (r^4 + 4 r) + (|5B|+|5C|)(4 r^2) + 
  |7A| (r^2 + 6 r) \\
  &+ |2F|(r^7 + r) + |2G|(r^5 + r^3).
\end{align*}
This is false for $r>125$. When $r=5$, we construct the module in Magma and find that $b(G)=3$, and from Lemma \ref{fieldext} that $b(G)=2$ when $r=25$. We also use the same techniques to compute that there is a regular orbit of $G$ on $V=V_8(r)$ when $r=49,125$. 

If instead $V$ is in characteristic 3, and $G$ has no regular orbit on $V$, then
\begin{align*}
r^8 \leq &( |2A| + |2C|+|2D| + |2E|) (2 r^4) + |2B| (r^6 + r^2)  + \frac{1}{2}i_3(G)r^4 + |5A| (r^4 + 4 r) + (|5B|+|5C|)(4 r^2) \\
&+ |7A| (r^2 + 6 r) + |2F|(r^7 + r) + |2G|(r^5 + r^3).
\end{align*}
This is false for $r\geq 243$. We compute that $b(G)=4,2,2,1$ when $r=3,9,27,81$ using a combination of GAP, Magma and Lemma \ref{fieldext}.
\end{proof}
\section{Proof of Theorem \ref{mainthm}: Exceptional groups}
The main result in this section is as follows.
\begin{theorem}
\label{exprop}
Suppose $G$ is an almost quasisimple group with $E(G)/Z(E(G))$ an exceptional simple group of Lie type. Let $V=V_d(r)$, $(r,q)=1$ be a module for $G$, with absolutely irreducible restriction to $E(G)$. Also suppose that $(r, |G|)>1$. Then one of the following holds:
\begin{enumerate}
\item $b(G) = \lceil \log |G|/ \log |V| \rceil$,
\item $ E(G)=2.\,^2{}B_2(8) $, $(d,r)=(8,5)$ and $b(G) = 2$,
\item $G=6\circ(2.G_2(4))$ or $6\circ(2.G_2(4).2)$ , $(d,r) = (12,7)$ and $b(G)=2$.
%\item{\color{red} This is given as a reg orb in the proof} $G=^3D_4(2)$, $(d,r) = (25,3)$ and $b(G)=2$ Reg orbfrom report in email 15/04/19
\end{enumerate}
\end{theorem}

We now proceed with the proof of Theorem \ref{exprop}.
If $G$ has no regular orbit on $V=V_d(r)$, then by Proposition \ref{tools}\ref{alphabound} and \ref{crude},
\[
r^d \leq 2 |G/F(G)|r^{\lfloor (1-1/\alpha(G))d\rfloor}.
\]
From Propositions \ref{mindegprop} and \ref{alphas} we deduce that either $G$ has a regular orbit on $V$, or $E(G)/Z(E(G))$ is one of ${}^2B_2(8)$, ${}^2F_4(2)'$, ${}^3D_4(2)$,${}^3D_4(3)$, $G_2(3)$, $G_2(4)$, $G_2(5)$ or $F_4(2)$.

We examine the remaining cases in a series of propositions.
%\subsection{$^2G_2(q)$, $q =3^{2n+1}$ and $^2E_6(q)$}
%DONEx2
%\begin{proposition}
%Proposition \ref{exprop} holds for almost quasisimple $G$ with $E(G)/Z(E(G)) \cong  {}^2G_2(q)$, $q =3^{2n+1}\geq 27$, or $^2E_6(q)$.
%\end{proposition}
%\begin{proof}
%%We consider $n\geq 1$ since $^2G_2(3)' \cong L_2(8)$.
%%Now, $|G|\leq (2n+1)q^3 (q^3+1)(q-1)$ and by Proposition \ref{mindegprop} $d\geq q(q-1)$. Also, $\alpha(x)\leq 3$ for all $x\in G\setminus F(G)$ by \cite[Proposition 5.8]{gs}. If $G$ has no regular orbit on $V$ then, applying \ref{crude},
%%\[
%%\frac{q(q-1)}{3} \leq \log_2 (2,2n+1)q^3 (q^3+1)(q-1))
%%\]
%%This is false for all $q =3^{2n+1}$, $n\geq 1$. Therefore, $G$ always has a regular orbit on $V$.
%In all cases, if $G$ has no regular orbit on $V=V_d(r)$, then by Proposition \ref{tools}\ref{alphabound} and \ref{crude},
%\[
%r^d \leq 2 |G/F(G)|r^{\lfloor (1-1/\alpha(G))d\rfloor}.
%\]
%
%
%If $E(G)/Z(E(G)) \cong  {}^2G_2(q)$, 
%this follows from applying Proposition \ref{tools}\ref{crude} with $|G|\leq (2n+1)q^3 (q^3+1)(q-1)$, $d_1(G) \geq q(q-1)$ and $\alpha(x)\leq 3$ for all $x\in G/F(G)$ by Proposition \ref{alphas}.
%If instead $E(G)/Z(E(G)) \cong\, ^2{}E_6(q)$, then $G$ has a regular orbit on all $V$ by Proposition \ref{tools}\ref{crude} with $|G| \leq 3q^{36}(q^{12}-1)(q^9+1)(q^8-1)(q^6-1)(q^5+1)(q^2-1) \log_p(q)$, $d_1(G) \geq q^9(q^2-1)$ and $\alpha(x)\leq 9$.
%\end{proof}
%\subsection{$^2B_2(q)$, $q =2^{2n+1}$}
%DONEx2
\begin{proposition}
Theorem \ref{exprop} holds if $E(G)/Z(E(G)) \cong  {}^2B_2(8)$.
\end{proposition}
\begin{proof}
%First suppose $q>8$. The result follows from applying Proposition \ref{tools}\ref{crude} with $|G| \leq (2n+1)q^2(q^2+1)(q-1)$, $d_1(G) \geq  (q/2)^{1/2} (q-1)$ and $\alpha(x)\leq 3$ for non-identity $x$ by Proposition \ref{alphas}.  Now suppose $q=8$. The number of elements of prime order in $^2B_2(8).3$ is 28391, so if $G$ has no regular orbit on $V$ then $\frac{d}{3} \leq \log_r (2\times 28391)$.
By \cite{HM}, we see that we only need to consider modules with $(d,r) = (8,5)$ or $(14,5)$. The generators for both modules are given in \cite{onlineATLAS} and we find that there is no regular orbit of $G$ on $V$ for $(d,r)=(8,5)$, but there is a base of size 2. If instead $(d,r)=(14,5)$, we find that $b(G)=1$.
\end{proof}
%\subsection{$^2F_4(q)$, $q =2^{2n+1}$}
%DONEx2
\begin{proposition}
Theorem \ref{exprop} holds if $E(G)/Z(E(G)) \cong  {}^2F_4(2)'$.
\end{proposition}
\begin{proof}
%First suppose $n\geq 1$.
%We have $|G|\leq (2n+1)q^{12}(q^6+1)(q^4-1)(q^3+1)(q-1)$, $\alpha(x)\leq 7$ for non-identity $x\in G$ and $d_r^1(G) \geq (q/2)^{1/2} q^4(q-1)$and $G$ has a regular orbit on $V$ by Proposition \ref{tools}\ref{crude}.
%
%Now suppose $n=0$, so $q=2$. 
The number of prime order elements in $^2F_4(2)'.2$ is 3304079. By the proof of \cite[Proposition 5.5]{gs}, $\alpha(x) \leq 4$ for involutions and $\alpha(x)\leq 3$ for odd prime order elements, and we use GAP to further deduce that $\alpha(g_{13})=2$. Moreover, $d\geq 26$ by Proposition \ref{mindegprop}, and we compute that $i_2(G/F(G))\leq 13455$ and $i_{13}(G/F(G))\leq 2764800$ . Therefore, if $G$ has no regular orbit on $V$, then by Proposition \ref{tools}\ref{crude},
\[
r^d\leq 2\times (3304079-2764800-13455)r^{\floor{2d/3}}+ 2\times 2764800r^{\floor{d/2}}+ 2\times 13455 r^{\floor{3d/4}}
\]
which is false unless $r=3$ and $d\leq 39 $. Consulting \cite{HM}, we find that we have only 26-dimensional modules of $^2F_4(2)'$ to consider.

%$d\leq 57$ for $r=3$, $d\leq 39$ for $r=5$ and $d\leq 28$ for $r=9$. Consulting \cite{HM}, we find that we have only 26- and 27-dimensional modules of $^2F_4(2)'$ to consider. First let $d=27$.
%Examining the Brauer character tables for $G$, we find that $\emax(g_{r'}) \leq 16$. Therefore, if  $r=5$, and $G$ has no regular orbit on $V$,
%\[
%r^{27}\leq 2\times3304079r^{16}+ 359424r^{18}
%\]
%%which is false, and a similar calculation gives the result for for $r=9$. When $r=3$, the module is not realised due to a $i_1$ irrationality in the corresponding Brauer character. We now consider the 26-dimensional modules. 
%
%The corresponding Brauer characters contain $i_2$ irrationalities which do not exist for $r=5$, so we only need consider $r=3,9$.
 Examining the Brauer character table in characteristic 3, we deduce that in each case, $\emax(g_{\{2,3\}'}) \leq 6$ and $\emax(g_2) \leq  16$. We further compute in GAP that $\alpha(x) = 2$ for $x \in G$ of order 3. Therefore, since $i_2({}^2F_4(2)) =13455$ and   $i_3({}^2F_4(2))=166400 $, if $G$ has no regular orbit on $V$ then
\[
r^{26} \leq 2\times 3304079r^{6}+166400r^{13} + 13455(r^{16}+r^{10})
\]
which is false for $r=3$.
\end{proof}
%\subsection{$^3D_4(q)$}
%DONEx2
\begin{proposition}
Theorem \ref{exprop} holds if $E(G)/Z(E(G)) \cong  {}^3D_4(q)$, $q\in \{2,3\}$.
\end{proposition}
\begin{proof}
We have $|G| \leq 3q^{12}(q^8+q^4+1)(q^6-1)(q^2-1)\log_pq $, $\alpha(x) \leq 7$  and $d_1(E(G))\geq q^3(q^2-1)+ q-1$. 
%Therefore, $G$ has a regular orbit on $V$ for $q\geq 4$ by Proposition \ref{tools}\ref{crude}.
Suppose that $q=3$. By Proposition \ref{tools}\ref{crude}, we are left to examine the cases where $r=2$, $d\leq 315$. According to \cite[A.15]{MR2303194}, we only need consider the 218-dimensional irreducible module of $G$. Using the construction of $V$ in GAP, we deduce that $\emax(g_{\{2,3\}'}) \leq 32$. We compute that the number of prime order elements in $^3D_4(3)$ is 6101289775853. 
 Using Magma, we compute that the number of elements of projective prime order 3 in $G$ is at most 7448676326. Therefore, if $G$ has no regular orbit on $V$, then 
	\[
	r^{218} \leq 2(i_2(G/F(G))+i_3(G/F(G)))r^{\floor{\frac{6\times 218}{7}}} + 2\times 6101289775853r^{32},
	\]
	which is false for $r=2$.

Now suppose that $q=2$.
Applying Proposition \ref{tools}\ref{crude} and inspecting \cite{HM}, we find that we must consider $d=25,26,52$ for $r\leq 102, 102$ and 9 respectively. The numbers of prime order elements in $^3D_4(2).3$ are given in Table \ref{poe3d423}.
\begin{table}[h!]
%\begin{tabular}{|c|c|}
%\hline 
%Prime $r_0$ & Number of elements \\ 
%\hline 
%2 & 69615 \\ 
%\hline 
%3 & 2457728 \\ 
%\hline 
%7 & 4852224 \\ 
%\hline 
%13 & 48771072 \\ 
%\hline 
%Total & 56150639 \\ 
%\hline 
%\end{tabular} 
\begin{tabular}{cccccc}
\toprule
Prime $r_0$ & 2 & 3 & 7 & 13 & Total \\ 
\midrule
Number of elements & 69615 & 2457728 & 4852224 & 48771072 & 56150639 \\ 
\bottomrule
\end{tabular} 
\caption{Prime order elements in $^3D_4(2).3$.\label{poe3d423}}
\end{table}
We proceed according to characteristic. In characteristic 13, we only need consider $V= V_{26}(13)$. From the Brauer character we determine that $\emax(g_{13'}) \leq  16$. Moreover, in GAP, we find that elements $g\in G$ of order 13 have $\dim C_V(g) = 2$. Therefore, if $G$ has no regular orbit on $V$,
\[
r^{26} \leq 2\times 56150639r^{16}+48771072r^2.
\]
This inequality is false for $r=13$. 
We now turn our attention to modules in characteristic 7. We need to consider $d=26$ for $r=7,49$ and $d=52$ for $r=7$. Using the Brauer character and a construction of the module in GAP, we find that when $d=26$, we have $\emax(g_7) \leq 8$ and $\emax(g_{7'}) \leq 16$. On the other hand, when $d=52$, then by Proposition \ref{tools}\ref{alphabound}, $\emax(g_7) \leq 44$ and from the Brauer character we determine that $\emax(g_{7'}) \leq 36$.
%\begin{table}[h!]
%\begin{tabular}{|c|c|c|}
%\hline 
%$d$ & Max eigsp dim of 7-regular element $(d_{7'})$& Max eigsp dim of 7-element $(d_{7})$ \\ 
%\hline 
%26 & 16 & 8 \\ 
%\hline 
%52 & 36 & 44 \\ 
%\hline 
%\end{tabular} 
%\end{table}
Therefore, if $G$ has no regular orbit on one of these modules, then
\[
r^{d} \leq 2\times 56150639r^{\emax(g_{7'}) }+4852224r^{\emax(g_7) }.
\]
This is false in both cases for $r\geq 7$.
Now suppose $V$ is an irreducible module in characteristic 3. Here we must consider $d=52$ for $r = 3,9$, and $d=25$ for $r\leq 81$.
The classes of elements of order 3 in $^3D_4(2).3$ and their sizes are given in Table \ref{3elts3d423}.
\begin{table}[h!]
\begin{tabular}{ccccccc}
\toprule
Class & 3A & 3B & 3C & 3C' & 3D & 3D' \\ 
\midrule
Size & 139776 & 326144 & 17472 & 17472 & 978432 & 978432 \\ 
\bottomrule
\end{tabular} 
\caption{Orders of conjugacy classes of elements of order 3 in $^3D_4(2).3$.\label{3elts3d423}}
\end{table}

	Let $V$ be the 52-dimensional module. From the Brauer character we see that $\emax(g_2) \leq 36$, and $\emax(g_{\{2,3\}'}) \leq 10$. We also construct $G\leq \gl(V)$ in GAP and determine the dimensions of fixed point spaces of elements of order 3. In this way, we deduce that if $G$ has no regular orbit on $V$, then
	\[
	r^{52} \leq 2\times 56150639r^{10}+2\times 69615r^{36}+ (2\times 978432+ 326144)r^{18}+(139776+2\times 17472)r^{22}.
	\]
	This inequality is false for $r\geq 3$, so $G$ has a regular orbit on $V$. 
	
	Finally, suppose that $d=25$. From the Brauer character we determine that $\emax(g_2) \leq 16$, and that $\emax(g_{7,13}) \leq 7$. We construct $G\leq \gl(V)$ in Magma and determine that $\dim C_V(g_3)=9$, unless $g_3$ is in class 3C or 3C', in which case $\dim C_V(g_3) = 13$.
%	Constructing the module in GAP, we observe that elements in classes 3A and 3B have 9-dimensional fixed point spaces. We also find that elements in classes 3C and 3D appear in the maximal subgroup $H = S_3\times L_2(8).3 < {}^3D_4(2).3$. Now $V \downarrow H$ has four trivial composition factors, and three composition factors of degree 7. Elements of order 3 have at most a 3-dimensional fixed point space on the 7-dimensional composition factor, so by Proposition \ref{compfactors}, $\dim C_V(g) \leq 4\times 1+3\times 3 = 13$ for an element in class 3C or 3D. 
	Computing the conjugacy classes of $G$ in GAP, we find that every element in class 3C' or 3D' is the square of an element in class 3C or 3D respectively. Since $C_V(x) \subseteq C_V(x^2)$ for all $x \in G$, we can exclude elements in classes 3C' and 3D' from the right hand side of the inequality below.
	Therefore, if $G$ has no regular orbit on $V$, then
	\[
	r^{25} \leq 2\times 56150639r^{7}+2\times 69615r^{16}+ (17472)r^{13}+(139776+326144+978432)r^9.
	\]
	This is false for $r\geq 9$.
	If $r=3$, then we construct the module in Magma and find an explicit regular orbit representative of $G$ on $V$.
\end{proof}
%\subsection{$G_2(q)$}
\begin{proposition}
Theorem \ref{exprop} holds for almost quasisimple $G$ with $E(G)/Z(E(G)) \cong G_2(5)$.
\end{proposition}
\begin{proof}
%First suppose $q\neq  3,4$. We have $|G| \leq 2q^6 (q^6-1)(q^2-1)\log_pq $, $\alpha(x) \leq 5$ and $d_1(G)\geq q(q^2-1)$. Therefore, $G$ has a regular orbit on $V$ for $q\geq 7$ by Proposition \ref{tools}\ref{crude}. It remains to consider $q=3,4,5$. 
%%\subsubsection{$G_2(5)$}
%%DONEx2
The number of elements of prime order in $G_2(5)$ is 1242340499, and so applying Proposition \ref{tools}\ref{alphabound} and \ref{crude}, and examining \cite{HM}, we see that we only need to consider the irreducible 124-dimensional module for $G$ over $\mathbb{F}_2$. Examining the Brauer character, we find that $\emax(g_{2'}) \leq 48$. Therefore, along with our observation that there are 406875 involutions in $G_2(5)$, we find that if $G$ has no regular orbit on this module then 
\[
2^{124} \leq 2\times 1242340499\times 2^{48}+406875\times 2^{\floor{\frac{4}{5} \times 124}}
\]
which is false. Therefore, $G$ has a regular orbit on $V$.
%\subsubsection{$G_2(4)$}
\end{proof}
\begin{proposition}
Theorem \ref{exprop} holds for almost quasisimple $G$ with $E(G)/Z(E(G)) \cong G_2(4)$.
\end{proposition}
\begin{proof}
The number of elements of prime order in $G_2(4).2$ is 55534959. Applying Proposition \ref{tools}\ref{crude} and examining \cite{HM}, we see that we only need consider modules of dimension $12, 64,78$ if $r=3$, and $d=12$ for $5\leq r\leq 484$. If $V$ is the 78-dimensional irreducible module of $G_2(4).2$ over $\F_3$, we deduce that $\emax(g_{3'}) \leq 46$ from the Brauer character. Therefore, since the number of elements of order 3 in $G_2(4).2$ is 1401920,  if $G$ has no regular orbit on $V$, then 
\[
3^{78} \leq 2\times 55534959\times 3^{46} + 1401920\times 3^{\floor{4\times 78/5}}
\]
which is false. Similarly, if $d=64$, then $\emax(g_{3'}) \leq 36$, and an analogous calculation shows that $G$ has a regular orbit on $V$ here. 
Finally, we consider the 12-dimensional module of $2.G_2(4)$ for $r\leq 484$. This extends to a module of $2.G_2(4).2$. If $r=3,5$ then $|G|>|V|$ and $G$ has no regular orbit on $V$. When $r=3$, constructing the module in GAP we find that there is a base of size 2. For the remaining cases, we will summarise our analysis in Table \ref{2g24-analysis}. The entries in the third column were obtained from the corresponding Brauer tables, while the entries in the fourth were derived from constructions in GAP.
Applying Proposition \ref{tools}\ref{qsgood} with the entries in the table, we see that $G$ has a regular orbit on $V$ for $r\geq 9$ and a base of size two for $r=5$. When $r=7$, we determine using orbit enumeration techniques in Magma that $2.G_2(4)$ and $2.G_2(4).2$ have regular orbits on $V$, while $6\circ (2.G_2(4))$ and $6\circ(2.G_2(4).2)$ do not, and so $b(G)=2$ for these groups by Proposition \ref{fieldext}.

\begin{table}[h!]
\begin{tabular}{cccc}
\toprule
 &  & \multicolumn{2}{c}{Eigenspace Dimensions}\\
 \cmidrule{3-4}
 Proj. Class & Size & $o(x) \nmid r$ & $o(x) \mid r$ \\ 
 \midrule
2A & 4095 & (8,4) & -- \\ 
2B & 65520 & $(6^2)$ & -- \\ 
3A & 4160 & $(6^2)$ & 6 \\ 
3B & 1397760 & $(4^3)$ & 4 \\ 
5A & 838656 & $(4,2^4)$ & 4 \\ 
5B & 838656 & $(4,2^4)$ & 4 \\ 
5C & 838656 & $(3^4)$ & 3 \\ 
5D & 838656 & $(3^4)$ & 3 \\ 
7A & 11980800 & $(2^6)$ & 2 \\ 
13A & 19353600 & $(1^{12})$ & 1 \\ 
13B & 19353600 & $(1^{12})$ & 1 \\ 
\bottomrule
\end{tabular} 
\caption{Eigenspace dimensions of projective prime order elements in $2.G_2(4).2$ acting on $V_{12}(r)$. \label{2g24-analysis}}
\end{table}
\end{proof}
\begin{proposition}
Theorem \ref{exprop} holds for almost quasisimple $G$ with $E(G)/Z(E(G)) \cong G_2(3)$.
\end{proposition}
\begin{proof}
%\subsubsection{$G_2(3)$}
%DONE
The number of prime order elements in $G_2(3).2$ is 1329587. Using analogous calculations to those above and inspecting \cite{HM}, we deduce that $G$ has a regular orbit on $V$ unless possibly if $r=2$ and $d=64,78,90$, $r=2,4,7$ and $d=27$, or $d=14$ and $r\leq 128$. By the proof of \cite[Proposition 5.6]{gs}, $\alpha(x)\leq 3$ if $x$ is an inner involution, $\alpha(x)\leq 5$ if $x$ is of order 3, and $\alpha(x)\leq 4$ otherwise. We have $i_3(G_2(3).2)=59696$, and $i_2(G_2(3).2) =10179$, of which 7371 are inner involutions.
Therefore, if $G$ has no regular orbit on the 90-dimensional modules of $G_2(3)$, then  %fuses on .2 and reg orbit b/c dim
\[
r^{90} \leq 2\times 1329587r^{\floor{3\times 90/4}}+ 2 \times 59696r^{\floor{4\times 90/5}}
\]
which is false for $r=2$. 
If $V$ is the 78-dimensional $\mathbb{F}_2G$-module, then examining the Brauer character we find that $\emax(g_{2'}) \leq 30$. So if $G$ has no regular orbit on $V$, 
\[
r^{78} \leq 2\times 1329587r^{30}+10179r^{\floor{3\times 78/4}},
\]
which again is false for $r=2$.
The 64-dimensional modules of $G_2(3)$ in characteristic 2 are not realised over $\mathbb{F}_2$ due to a $b_{27}$ irrationality. Similarly, the 27-dimensional module of $3.G_2(3)$ is not realised over $\mathbb{F}_2$ because of a $z_3$ irrationality.  If $d=27$ and $r=4$ or 7, then from the Brauer character we determine that $\emax(g_{\{2,p\}'}) \leq 9$, and using GAP we determine that $\alpha(g_7)=2$, so if $G$ has no regular orbit on $V$,
\[
r^{27} \leq 2\times1329587r^9+ 2i_2(G/F(G))r^{\floor{3\times 27/4}}+ 2\times i_7(G/F(G))r^{\floor{ 27/2}}
\]
which is false for $r=4,7$.
It therefore remains to consider the 14-dimensional module of $G_2(3).2$ for $r\leq 128$.
If $r=2$, then $|V|<|G|$ and we compute in GAP that $b(G)=2$. We also compute that if $r=4$, then $G$ has a regular orbit on $V$. 
\begin{table}[h!]
\begin{tabular}{ccccc}
\toprule
 & $o(x) = 2$ & $o(x) = 3$ & $o(x) = 7$ & $o(x) = 13$ \\ 
\midrule
Number of elements & 10179 & 59696 & 606528 & 653184 \\ 
Largest eigsp. in char. 2 & 8 & 8 & 2 & 2 \\ 
Largest eigsp. in char. 7 & 8 & 8 & 2 & 2 \\ 
Largest eigsp. in char. 13 & 8 & 8 & 2 & 2 \\ 
\bottomrule 
\end{tabular} 
\caption{Information about projectively prime order elements acting on an absolutely irreducible 14-dimensional module for $G_2(3).2$.\label{g23tab}}
\end{table}
Otherwise, if $G$ has no regular orbit on $V$ then by Proposition \ref{tools}\ref{qsgood} and the information in Table \ref{g23tab},
\[
r^{14} \leq 2 (10179 r^8 + 606528 r^2 + 653184 r^2)+ 59696 (r^8 + 2 r^3)
\]
which is false for $r\geq 7$.
\end{proof}

%\subsection{$F_4(q)$}
\begin{proposition}
Theorem \ref{exprop} holds for almost quasisimple $G$ with $E(G)/Z(E(G)) \cong F_4(2)$.
\end{proposition}
\begin{proof}
%Here $|G| \leq 2q^{24}(q^{12}-1)(q^8-1)(q^6-1)(q^2-1)\log_pq$ and $\alpha(G) \leq 8$ by Proposition \ref{alphas}.
%First suppose $q$ is odd. Then $G$ has a regular orbit on $V$ by an application of Proposition \ref{tools}\ref{crude} with $d_1(G) \geq q^8+q^4-2$.
%
%Now suppose $q\geq 4$ is even. By Proposition \ref{mindegprop}, $d_1(G) \geq q^7(q^3-1)(q-1)/2$. Therefore if $G$ has no regular orbit on $V$, then
%\[
%r^d \leq 4q^{24}(q^{12}-1)(q^8-1)(q^6-1)(q^2-1)\log_2q r^{\floor{7d/8}}.
%\]
%which is false with $d=d_1(G)$, and therefore for all $V$, for $q\geq 4$.
The number of prime order elements in $F_4(2).2$ is 650870214978815. So if $G$ has no regular orbit on $V$ then
\[
r^d \leq 2\times 650870214978815r^{\floor{7d/8}}
\]
which is false unless $d\leq 248$ and $r\leq 125$. Examining \cite{HM}, we find that we only need to consider the 52-dimensional module of $2.F_4(2)$.
 We record the classes of elements of projective prime order in $2.F_4(2).2$ and their eigenspace dimensions in cases where they are semisimple in Table \ref{F42dim52}. For each characteristic dividing the order of $2.F_4(2).2$, we also compute a bound of the fixed point space dimensions on unipotent elements and these are also recorded in Table \ref{F42dim52}.
\begin{table}[h!]
\begin{tabular}{cc}
\toprule
Class in $F_4(2).2$ & Eigensp. dims \\ 
\midrule
2A & (36,16) \\ 
2B & (32,20) \\ 
2C & (28,24) \\ 
3A & $(22,15^2)$ \\ 
3B & $(22,15^2)$ \\ 
3C & $(16,18^2)$ \\ 
5A &$ (12, 10^4)$ \\ 
7A & $(10,7^6)$ \\ 
13A &$ (4^{13}) $\\ 
17A & $(4,3^{16}) $\\ 
2D & $(26^2)$ \\ 
\bottomrule
\end{tabular} 
\begin{tabular}{ccc}
\toprule
$char(V)=r_0$ &  $\dim C_V(g_{r_0}) \leq $ & Method \\ 
\midrule
3 & 22 & Construction in GAP \\ 
5 & 12 & Construction in GAP \\ 
7 & 26 &  Compute $\alpha(g_{7})=2$ \\ 
13 & 26 &  Compute $\alpha(g_{13})=2$ \\ 
17 & 26 & Compute $\alpha(g_{17})=2$ \\ 
%7 & 10 & Restrict $V$ to $\mathrm{PSp}_8(2)$ \\ 
%13 & 4 & Restrict $V$ to $^2F_4(2)$ \\ 
%17 & 35 & Restrict $V$ to $\mathrm{PSp}_4(4)$ \\ 
\bottomrule
\end{tabular} 
\caption{Eigenspace dimensions of projectively prime order elements on an absolutely irreducible 52-dimensional module for $2.F_4(2)$.\label{F42dim52}}
\end{table}
Applying Proposition \ref{tools}\ref{qsgood}, we deduce that $G$ has a regular orbit on $V$ in each case.
\end{proof}
%\begin{proposition}
%Theorem \ref{mainthm} holds for $G$ with $E(G)/Z(E(G)) \cong E_6(q)$, $E_7(q)$ or $E_8(q)$.
%\end{proposition}
%\begin{proof}
%The result follows from applying Proposition \ref{tools}\ref{crude} with the values in Table \ref{etables}.
%\begin{table}[h!]
%\begin{tabular}{cccc}
%\toprule
%Group & $|G|<$ & $\alpha \leq $ & $d_1(G)$ \\ 
%\midrule
%$E_6(q)$ & $6q^{79}$ & 9 & $(q^5+q)(q^6+q^3+1)-1$ \\ 
%$E_7(q)$ & $q^{134}$ & 10 & $q^{17}-q^{15}$ \\ 
%$E_8(q)$ & $q^{249}$ & 11 & $q^{29}-q^{27}$ \\ 
%\bottomrule
%\end{tabular} 
%\caption{Values for $E_6(q)$, $E_7(q)$ and $E_8(q)$. \label{etables}}
%\end{table}

%
%
%\end{proof}
This completes the proof of Theorem \ref{exprop}, and hence Theorem \ref{mainthm}.
\section*{Acknowledgements}
This paper represents part of the PhD work of the author under the supervision of Professor Martin W. Liebeck, and the author would like to thank Professor Liebeck for his guidance. The author acknowledges the support of an EPSRC International Doctoral Scholarship at Imperial College London.
The author would also like to thank Professors Eamonn O'Brien and J\"urgen M\"uller for aiding with a number of computationally intensive cases included in this paper, and Dr. Tim Burness for his careful reading of the content of this paper included in the author's PhD thesis. The author is also grateful to the referee for many helpful suggestions and in particular for pointing out an omitted case in \cite[Proposition 2.2]{MR1829482} which affected Theorem \ref{mainthm} here.  She wishes to thank Professor Frank L\"ubeck for providing an argument for dealing with the omitted case in \cite{lubeck2021orbits}.
\bibliographystyle{plain}
\bibliography{writeup.bib}
\end{document}